\documentclass[11 pt]{article}

\usepackage[utf8]{inputenc}
\usepackage{amsmath,amsfonts,stmaryrd,amsthm,thmtools,hyperref,graphicx}

\usepackage{caption,subcaption}
\usepackage{tikz}
\usetikzlibrary{arrows.meta}
\usetikzlibrary{shapes.geometric}

\usepackage{fullpage}
\usepackage{authblk}
\usepackage[normalem]{ulem}

\newtheorem{thm}{Theorem}
\newtheorem{lemma}[thm]{Lemma}
\newtheorem{prop}[thm]{Proposition}
\newtheorem{cor}[thm]{Corollary}
\theoremstyle{definition}

\newtheorem{rem}[thm]{Remark}

\newtheorem{condition}[thm]{Condition}

\numberwithin{definition}{section}
\numberwithin{proc}{section}
\numberwithin{equation}{section}
\numberwithin{condition}{section}
\numberwithin{condition}{section}
\numberwithin{prop}{section}
\numberwithin{thm}{section}
\numberwithin{lemma}{section}
\numberwithin{rem}{section}
\numberwithin{cla}{section}
\numberwithin{obs}{section}
\numberwithin{cor}{section}

\usepackage{enumerate}

\usepackage{xcolor}
\definecolor{webgreen}{rgb}{0,.5,0}
\definecolor{Maroon}{HTML}{800000}

\hypersetup{colorlinks, breaklinks, urlcolor=Maroon, linkcolor=Maroon, citecolor=webgreen} 

\usepackage[numbers]{natbib}
\newcommand*{\doi}[1]{doi:\,\texttt{\href{http://dx.doi.org/#1}{#1}}}

\usepackage{stmaryrd}
\usepackage{bbold} 

\newcommand{\bE}{\mathbb{E}}

\newcommand\bfe{{\mathbf e}}

\newcommand\bfh{{\mathbf h}}

\newcommand\bfw{{\mathbf w}}

\newcommand\vbfd{{\vec{\mathbf{d}}}}

\newcommand\dsG{{\mathbb G}}

\newcommand\dsN{{\mathbb N}}

\newcommand\dsR{{\mathbb R}}

\newcommand\dsZ{{\mathbb Z}}

\usepackage{physics}
\usepackage{mathtools}
\allowdisplaybreaks

\newcommand{\ind}[1]{\mathbf{1}_{#1}}

\newcommand{\N}{{\mathbb{N}}}

\newcommand{\da}{\downarrow}
\newcommand{\ua}{\uparrow}

\newcommand\cA{{\cal{A}}}
\newcommand\cB{{\cal{B}}}
\newcommand\cC{{\cal{C}}}

\newcommand\cE{{\cal{E}}}
\newcommand\cF{{\cal{F}}}
\newcommand\cG{{\cal{G}}}

\newcommand\cL{{\cal{L}}}

\newcommand\cN{{\cal{N}}}

\newcommand\cP{{\cal{P}}}

\newcommand\cR{{\cal{R}}}

\newcommand\cT{{\cal{T}}}
\newcommand\cU{{\cal{U}}}
\newcommand\cV{{\cal{V}}}
\newcommand\cW{{\cal{W}}}
\newcommand\cX{{\cal{X}}}
\newcommand\cY{{\cal{Y}}}
\newcommand\cZ{{\cal{Z}}}

\newcommand{\E}[1]{{\mathbb{E}}\left[#1\right]}
\newcommand{\ee}[1]{{\mathbb{E}}[#1]}
\newcommand{\e}{{\mathbb{E}}}
\newcommand{\V}[1]{{\mathrm{Var}}\left(#1\right)}
\newcommand{\Vv}[1]{{\mathrm{Var}}(#1)}

\newcommand{\p}[1]{{\mathbb{P}}\left(#1\right)}
\newcommand{\pt}[1]{{\mathbb{P}}_t\left(#1\right)}

\newcommand{\eql}{\,{\buildrel \cL \over =}\,}

\newcommand{\Be}{\mathop{\mathrm{Be}}}

\DeclarePairedDelimiter{\floor}{\lfloor}{\rfloor}
\DeclarePairedDelimiter{\ceil}{\lceil}{\rceil}

\newcommand{\tin}{\mathrm{\text{i-sb}}}
\newcommand{\tout}{\mathrm{\text{o-sb}}}

\newcommand{\hnu}{{\hat \nu}}
\newcommand{\nupm}{{\nu^{\pm}}}
\newcommand{\xihat}{{\hat \xi}}

\newcommand{\xistar}{{\xi^{*}}}
\newcommand{\xitilde}{{\tilde \xi}}

\newcommand{\new}[1]{#1}

\DeclareMathOperator{\dist}{dist}
\DeclareMathOperator{\diam}{diam}

\newcommand\din{D_{{\tin}}}

\newcommand\dout{D_{{\tout}}}

\newcommand\doutp{\dout^{+}}

\newcommand\doutm{\dout^{-}}

\newcommand\dtilde{\tilde{D}_{{\tout}}}
\newcommand\dtildep{\tilde{D}_{{\tout}}^{+}}
\newcommand\dtildem{\tilde{D}_{{\tout}}^{-}}

\newcommand\vecGn{\vec{\dsG}_{n}}
\newcommand\vecGns{\vec{\dsG}_{n}^{\mathrm{s}}}

\newcommand\GW{\mathrm{BGW}}

\newcommand{\procall}{({X}_r)_{r \ge 0}}

\newcommand{\procstar}{({X}_r^{*})_{r \ge 0}}

\newcommand{\procspine}{(\hat{X}_j)_{j \ge 0}}
\newcommand{\procstart}{({X}_t^{*})_{t \ge 0}}
\newcommand{\proctildet}{(\tilde{X}_t)_{t \ge 0}}
\newcommand{\procspinet}{(\hat{X}_t)_{t \ge 0}}
\newcommand{\trunc}[1]{X_{r}^{(#1)}}
\newcommand{\proctrunc}[1]{(\trunc{#1})_{t \ge r \ge 0}}

\newcommand{\walk}{({Z}_t)_{t \ge 0}}
\newcommand{\walke}{({Z}_t^{\bfe})_{t \ge 0}}

\usepackage{xspace}

\newcommand\pimin{\pi_{{\min}}}
\newcommand\pimax{\pi_{{\max}}}
\newcommand\pimine{\pi_{{\min}}^{\bfe}}
\newcommand\piminez{\pi_{0}^{\bfe}}

\newcommand\tent{\tau_{\mathrm{\,ent}}}
\newcommand\tcov{\tau_{\mathrm{\,cov}}}
\newcommand\thit{\tau_{\mathrm{\,hit}}}
\newcommand\TX{{\text{TX}}}

\newcommand\hatw{\hat{\bfw}}

\title{Minimum stationary values of sparse random directed graphs}

\author[*]{Xing Shi Cai}
\author[**]{Guillem Perarnau}
\affil[*]{\small\it Mathematics Department, Uppsala University, Sweden. Email:~{\tt xingshi.cai@tutanota.com}.}
\affil[**]{\small\it Departament de Matem\`atiques (MAT), Universitat Polit\`ecnica de Catalunya (UPC), Barcelona, Spain. Email:~{\tt guillem.perarnau@upc.edu}.}

\begin{document}
\maketitle

\begin{abstract}
    We consider the stationary distribution of the simple random walk on the directed configuration
    model with bounded degrees. Provided that the minimum out-degree is at least \(2\), with high
    probability (whp) there is a unique stationary distribution (\new{uniqueness} regime).  We show that
    the minimum positive stationary value is whp \(n^{-(1+C+o(1))}\) for some constant \(C \ge 0\)
    determined by the degree distribution, answering a question raised by Bordenave, Caputo and
    Salez~\cite{bordenave2018}. In particular, $C$ is the competing combination of two factors: (1)
    the contribution of atypically ``thin'' in-neighbourhoods, controlled by subcritical branching
    processes; and (2) the contribution of atypically ``light'' trajectories, controlled by large
    deviation rate functions. Additionally, we give estimates for the expected lower tail of
    the empirical stationary distribution. As a by-product of our proof, we obtain that the
    hitting and the cover time are both \(n^{1+C+o(1)}\) whp.  Our results are in sharp contrast to
    those of Caputo and Quattropani~\cite{caputo2020a} who showed that under the additional
    condition of minimum in-degree at least 2 (ergodicity regime), stationary values only have
    logarithmic fluctuations around \(n^{-1}\).
\end{abstract}

\section{Introduction}\label{sec:introduction}

\subsection{The directed configuration model}

The directed configuration model was introduced by Cooper and Frieze in~\cite{cooper2004}.  Let
\([n]\coloneqq \{1,\dots,n\}\) be a set of \(n\) vertices.  Let \(\vbfd_n=((d^{-}_1,d_1^{+}),\dots,
(d^{-}_n,d^{+}_n))\) be a bi-degree sequence with \(m\coloneqq \sum_{v\in [n]} d^{+}_{v} = \sum_{v \in
[n]}d^{-}_{v}\). 
Let $\delta^-$ and $\delta^+$ be the minimum in- and out-degree respectively. Let $\Delta^-$ and $\Delta^+$ be the maximum in- and out-degree respectively. The directed configuration model, which we denote by \(\vecGn=\vecGn(\vbfd_n)\), is
the random directed multigraph on \([n]\) generated by giving \(d^{-}_{v}\) \emph{heads}
(in-half-edges) and \(d^{+}_{i}\) \emph{tails} (out-half-edges) to vertex \(v\), and then pairing the
heads and the tails uniformly at random. Observe that \(\vecGn\) is not necessarily simple: loops and multiedges are allowed.

The directed configuration model is of practical importance as many complex real-world networks are
directed.  For instance, it has been used to study neural networks~\cite{amini2010}, Google's
PageRank algorithm~\cite{chen2017}, and social networks~\cite{li2018}.

The original paper by Cooper and Frieze~\cite{cooper2004} studies the birth of a linear size strongly
connected component in \(\vecGn\). Their result was recently improved by
Graf~\cite{graf2016} and by the two authors~\cite{cai2020}.

Lately, there has been some progress on the distances in the directed configuration model for
bi-degree sequences with finite covariances. Typical distances in $\vecGn$ were studied by van der
Hoorn and Olvera-Cravioto~\cite{hoorn2018}.  Caputo and Quattropani~\cite{caputo2020a} showed that
the diameter of $\vecGn$ is asymptotically equal to the typical distance, provided that
$\delta^{\pm}\geq 2$ and $\Delta^\pm=O(1)$. In our previous work~\cite{cai2020a}, we showed that the
diameter has different behaviour if no constraints on the minimum degree are imposed.

One motivation to the study of distances in \(\vecGn\) is their close connection to certain properties
of random walks, in particular, to their stationary distributions, denoted by $\pi$. While $\pi$
is trivially determined by the degree sequence in undirected graphs, in the directed case $\pi$ is a
complicated random measure that depends on the geometry of the random digraph.  Cooper and
Frieze~\cite{cooper2012} initiated the study of $\pi$ in random digraphs, determining it in the
strong connectivity regime of the directed Erd\H os-R\'enyi random graph. They also established a
relation between the minimum stationary value and stopping times such as the hitting
and the cover time.  Extremal stationary values for the $r$-out random digraph were studied by
Addario-Berry, Balle, and the second author~\cite{addario-berry2020}.

Regarding the directed configuration model, Bordenave, Caputo and Salez~\cite{bordenave2018,bordenave2019} studied the mixing time of a
random walk on $\vecGn$ and showed that it exhibits cutoff, \new{a sharp threshold phenomenon for the convergence of the random walk to equilibrium}. Additionally, 
they proved that for a vertex $v\in[n]$, $\pi(v)$ is essentially determined by the local
in-neighbourhood of $v$ and well-approximated by a deterministic law (see~\autoref{rem:empirical}). 
These results provide a precise description of typical stationary probabilities but fall short of capturing
the exceptional values of $\pi$. 
 
In~\cite{bordenave2018}, the authors raised the question of studying the extremal values of the
stationary distribution in $\vecGn$.  Let $\pimin$ and $\pimax$ be the smallest and largest positive
values of $\pi$, respectively. From now on and throughout this paper, we will assume that all
degrees are bounded; i.e., $\Delta^\pm=O(1)$.  In this context, the condition $\delta^+\geq 2$ is
essentially necessary to avoid (possibly many) trivial stationary measures~(see~\autoref{rem:SKDJ}).
Under the additional condition $\delta^{-}\geq 2$, Caputo and Quattropani~\cite{caputo2020a} showed
that the random walk is ergodic with high probability (whp), \new{so we call the case $\delta^{\pm}\geq 2$} the \emph{ergodicity
regime}, and that there exists $C\geq 1$ such that, whp
\begin{align}
C^{-1}\frac{\log^{1-\gamma_0} n}{n}&\leq \pimin \leq C \frac{\log^{1-\gamma_1} n}{n},\label{SMNR}\\
C^{-1}\frac{\log^{1-\kappa_1} n}{n}&\leq \pimax \leq C \frac{\log^{1-\kappa_0} n}{n},\label{SMNS}
\end{align}
where $\gamma_0\geq \gamma_1\geq 1$ are defined in terms of $\delta^-$ and $\Delta^+$, and
$\kappa_0\leq \kappa_1\leq 1$ are defined in terms of $\delta^+$ and $\Delta^-$.  \new{While in general these bounds are not tight,} $\gamma_0= \gamma_1$ and $\kappa_0= \kappa_1$ if there are
linearly many vertices with degrees $(\delta^-,\Delta^+)$ and $(\delta^+,\Delta^-)$, respectively.
Finally, using the bound on $\pimin$, the authors showed that in the ergodicity regime the cover
time satisfies whp
\begin{equation}\label{RESG}
    C^{-1} n \log^{\gamma_1}{n} \leq \tcov \leq C n \log^{\gamma_0}{n} .
\end{equation}

The main purpose of this paper is to study the extremal values of the stationary distribution
outside the ergodicity regime. If $\delta^+\geq 2$ but no condition on the minimum in-degree is
imposed, the random walk might fail to be ergodic, but the stationary measure is whp unique; we call
it the \emph{uniqueness regime}.  While~\eqref{SMNR} and~\eqref{SMNS} indicate that the extremal values
exhibit logarithmic fluctuations in the ergodicity regime, our main result shows in that uniqueness
regime $\pimin$ can have polynomial deviations with respect to the typical stationary values. As an easy
consequence of our proof, we determine the hitting and the cover time up to subpolynomial
multiplicative terms.

\subsection{Notation and results}\label{sec:res}

Before stating our results, we need to define some parameters of the bi-degree sequence \(\vbfd_{n}\).  Although \(n\) does not appear in most of the notation, the reader should keep in mind
that all the parameters defined here depend on \(n\).

Let \(n_{k,\ell}\coloneqq |\{v: (d_{v}^-,d_{v}^+)=(k,\ell)\}|\) be the number of pairs \((k,\ell)\) in
\(\vbfd_{n}\). 
Let \(D=(D^{-},D^{+})\) be the degree pair (number of heads and tails) of a uniform
random vertex. In other words, $\p{D=(k,\ell)} = n_{k,\ell}/n$.

The bivariate generating function of \(D\) is defined by
\begin{equation}\label{FDMH}
    G_D(z,w)
    \coloneqq 
    \sum_{k ,\ell \ge 0} \p{D=(k,\ell)} z^{k} w^{\ell}
    .
\end{equation}
Let \(\lambda \coloneqq m/n\).  Define the \emph{out-size-biased} distribution of
\(D\), \( \dout = (\doutm, \doutp)\), by
\begin{equation}
    \label{UHJC}
    \p{\dout=(k,\ell)}\coloneqq \frac{\ell}{\lambda}\cdot \p{D = (k,\ell)}
    ,
    \qquad
    \text{for }
    k \ge 0, \ell \ge 0
    .
\end{equation}
In words, \(\dout\) is the degree distribution of a vertex incident to a uniform random tail. 
Consider a branching process with offspring distribution  \(\dout^-\) and let \(s^{-}\) be its survival probability. \new{Note that $s^-=1$ if and only if $\delta^-\geq 1$.}

Define  the \emph{expansion rate} by
\begin{equation}\label{SPEK}
\nupm \coloneqq \frac{1}{\lambda}\frac{\partial^2 G_D}{\partial z \partial w}(1,1) = \new{\frac{1}{\lambda }\sum_{k,\ell\geq 0} k\ell \cdot \p{D=(k,\ell)}},
\end{equation}
and the \emph{subcritical in-expansion rate} by
\begin{equation}\label{eq:hnu}
\hnu^-\coloneqq
\frac{1}{\lambda}
\frac{\partial^2 G_D}{\partial z\partial w} (1-s^{-},1) \new{= \frac{1}{\lambda }\sum_{k,\ell\geq 0} k(1-s^-)^{k-1} \ell  \cdot \p{D=(k,\ell)}}\in [0,1)
,
\end{equation}
\new{where if $s^-=1$ and $k\in\{0,1\}$, we use the conventions $0/0=0$ and $0^0=1$. With these conventions, if $\hnu^-=0$, then $s^-=1$.}

\new{Observe the symmetry in the definition of $\nupm$. In fact, this parameter can be understood as the expected size of the in-neighbourhood of any tail, or equivalently, as the expected size of the out-neighbourhood of any head. In particular, the neighbourhoods of any half-edge grow at an expected rate roughly $\nupm$. This symmetry is broken when considering subcritical growth: $\hnu^-$, which is less than $1$, turns out to be approximately the expected rate of growth of the in-neighbourhood of any tail, conditional on the event that this neighbourhood dies out in few steps. (See~\autoref{sec:strategy} for more details.)
}

\new{For $\hnu^- >0$}, let \(\dtilde = (\dtildem,\dtildep)\) be the random vector with distribution
\new{
\begin{equation}\label{XLAJ}
    \p{\dtilde=(k,\ell)}
    \coloneqq
    \frac{k \left(1-s^-\right){}^{k-1}\ell}{\hnu^{-}\lambda}
    \cdot \p{D = (k,\ell)}
    ,
    \qquad
    \text{for }
    k \ge 1, \ell \ge 0
    ,
\end{equation}
}
\new{which can be thought as the degree distribution of a vertex incident to a uniform random tail, conditional on only one of its in-neighbours being in a large strongly connected component. (See~\autoref{sec:strategy} for a more precise description.)}


Define the \emph{subcritical in-entropy} as 
\begin{equation}\label{SKFN}
\hat{H}^- \coloneqq \E{\log\dtildep} 
=\frac{1}{\hnu^- \lambda} \sum_{k,\ell\geq 0} k  (1-s^-){}^{k-1} \ell\log \ell\cdot  \p{D=(k,\ell)},
\end{equation}
\new{if $\hnu^- >0$, where $0\log{0}=0$ by convention,
and assign an arbitrary value to $\hat{H}^-$ if $\hnu^- =0$.}
This parameter can be seen as an average row entropy of certain transition matrix (see~\cite{bordenave2019}) and is related to the probability that the random walk follows a typical trajectory under subcritical in-growth. 

The \emph{large deviation rate function} (or \emph{Cram\'er function}) of \(Z=\log \dtildep\) is defined by
\begin{equation}\label{YBQH}
    I(z)\coloneqq \sup_{x \in \mathbb{R}} \{xz -\log \bE[e^{x Z}]\}
    ,
    \qquad
    \text{for }
    z \in \dsR
    .
\end{equation}

Note that \new{$I$ is non-negative}, attaining its minimum at $\hat{H}^{-}$, for which \(I(\hat{H}^{-})=0\). 
Let
\begin{equation}\label{ODRS}
    \phi(a) \coloneqq 
\begin{cases} 
    \frac{1}{a} \left(|\log{\hnu^-}| + I(a\hat{H}^-)\right), &\text{if }\hnu^->0,\\
    \new{\infty,} &\new{\text{if }\hnu^-=0}.
\end{cases}
\end{equation}
and let $a_{0}$ be a minimising value in $(0,\infty)$. In fact, since $I(a\hat{H}^-)$ is minimised at $a=1$, necessarily $a_0\in [1,\infty)$.

We remark that all parameters defined above, as well as $\phi$ and $a_0$, depend implicitly on $\vbfd_{n}$ and so on $n$. While we do not make the dependence explicit to keep the notation simple, this should be kept in mind. \new{Additionally, we will not make any assumption on the limiting behaviour of the degree distribution, as it is often the case when studying the configuration model.}
 
Conditioned on \(\vecGn\), a simple random walk on \(\vecGn\) is a Markov process \(\walk \) with
state space \([n]\). Given the current vertex \(Z_{t}\), the walk chooses an out-neighbour of
\(Z_{t}\) uniformly at random as \(Z_{t+1}\), which is always possible as we assume
$\delta^+\geq 2$. If as \(t \to \infty\) the distribution of \(Z_{t}\) converges to the same
distribution \(\pi\) regardless of the choice of \(Z_{0}\), i.e., if there exists a probability
density function \(\pi\) on \([n]\)
such that
\begin{equation}\label{BRDP}
    \lim_{t \to \infty} \sup_{u,v \in [n]} \abs{\p{Z_{t} = v \mid Z_{0} = u} - \pi(v)}
    =
    0
    ,
\end{equation}
we say \(\walk\) has a \emph{unique stationary distribution} \(\pi\).
We define 
\begin{equation}\label{ADMV}
    \pi_{\min} \coloneqq \min\left\{\pi(v) : v \in [n], \pi(v) > 0\right\}
    ,
    \qquad
    \pi_{\max} \coloneqq \max\left\{\pi(v) : v \in [n]\right\}
    ,
\end{equation}
if the walk has a unique stationary distribution, and assign arbitrary values to them otherwise.

Our main result is the following:
\begin{thm}\label{thm:main}
    Assume that \(\delta^{+} \ge 2\) and \(\Delta^{\pm} \le M\) where \(M \in \dsN\) is a fixed
    integer. For every $\varepsilon>0$, with high probability \new{as $n\to \infty$},
    \begin{equation}\label{MHUZ1}
         n^{-(1+\hat{H}^-/\phi (a_{0}))-\varepsilon} \leq \pimin  \leq  n^{-(1+\hat{H}^-/\phi (a_{0}))+\varepsilon}.
    \end{equation}
\end{thm}

Clearly, if $v\in [n]$ is chosen uniformly at random, then $\E{\pi(v)}=1/n$. Our theorem shows that the minimum stationary value exhibits a polynomial deviation from the expected value. Namely, the additional exponent \(\hat{H}^-/\phi(a_{0})\) comes from the fact that whp some vertices that are exceedingly difficult to reach by a simple random walk.
From the intuitive point of view, a path could be hard to follow either because it is long or because it contains many high branching vertices (i.e. vertices with large out-degrees).
On the one hand, vertices that are furthest from the bulk of the graph (i.e. most of the vertices) can only be accessed from there through long paths; however, there are not many such vertices and all long paths leading to them behave typically in terms of branching.
On the other hand, vertices that are closer to the bulk can be reached from there through shorter paths; however, there are many such vertices and some of the paths leading to them behave atypically bad from the branching point of view.
Therefore, determining the minimal stationary value can be seen as a competitive combination of two properties \new{of long paths that connect the bulk to a vertex}, which implicitly appear in the definition of $\phi(a)$ in~\eqref{ODRS}: 
\begin{itemize}
\item[(1)] being long, which is controlled by the term $|\log \hnu^-|$; and
\item[(2)] having vertices of large out-degree, which is controlled by the term $I(a\hat{H}^-)$.
\end{itemize}
The optimal ratio between these two factors is given by the value $a_0$ that minimises $\phi(a)$.

\begin{rem}\label{rem:crit}(Critical distance)
For any digraph $G$, let $\text{dist}(u,v)$ be the length of the shortest directed path from $u$ to $v$ if it exists, and $\text{dist}(u,v)=\infty$ otherwise. Define the \emph{diameter} and the \emph{critical distance} of $G$ as follows:
\begin{align}\label{FODE}
\diam(G)&= \max_{u,v\in [n]} \text{dist}(u,v),\\
\text{crit}(G)&= \max_{u,v\in [n]\atop \pi(v)=\pimin} \text{dist}(u,v).
\end{align}
Then, the ratio between $\text{crit}(G)$ and $\diam(G)$ can be used to measure \new{which is the trade-off between (1) and (2)} in determining $\pimin$. The authors have determined the asymptotic behaviour of the diameter in~\cite{cai2020a}, which we now present. Define $s^+$ to be the survival probability of a branching process with offspring distribution having generating function $\frac{1}{\lambda}\frac{\partial G_D}{\partial z}(1,w)$ and $\hnu^+\coloneqq \tfrac{1}{\lambda} \tfrac{\partial^2 G_D}{\partial z\partial w} (1,1-s^{+})$. Since $\delta^+\geq 2$, we have $s^+=1$ and $\hnu^+=0$. It follows from~\cite{cai2020a} that
\begin{align}\label{SDIM}
\left(\frac{1}{|\log\hnu^-|} + \frac{1}{\log \nupm}\right)^{-1}\frac{\diam(\vecGn(\vbfd_n))}{\log n} \to 1 \quad \text{ in probability}.
\end{align}
From the proof of~\autoref{thm:main} (see~\eqref{SODW}), one can extract that
\begin{align}\label{SKEW}
\left(\frac{1}{|\log\hnu^-|+I(a_0 \hat{H}^-)} + \frac{1}{\log \nupm} \right)^{-1}\frac{\text{crit}(\vecGn(\vbfd_n))}{\log n} \to 1 \quad \text{ in probability},
\end{align}
which only coincides with the diameter when $a_0=1$.
\end{rem}

\begin{rem}\label{rem:mindeg}(Minimum in-degree)
\autoref{thm:main} is of particular interest for \(\delta^{-} \in\{0,1\}\). \new{If $\delta^-\geq 2$, then $s^-=0$ and since $\p{D=(1,\ell)}=0$ for all $\ell \geq 0$, we have that} $\hnu^-=0$ implying \(\hat{H}^-/\phi(a_{0})=0\), and the result is trivial. (See~\eqref{SMNR} for much more precise bounds on $\pimin$ when $\delta^{-}\geq 2$.) The most natural case is $\delta^-=0$;  then $s^-<1$ and $\hnu^->0$. In the particular case $\delta^-=1$, we have $s^-=1$ and a simple computation gives $\hnu^-= \p{\doutm=1}>0$.
In particular, for \(\delta^{-} \in\{0,1\}\), if $\p{\doutm=\delta^-}$ is bounded away from $0$ as $n\to \infty$, then the additional exponent \(\hat{H}^-/\phi(a_{0})\) in \eqref{MHUZ1} is also bounded away from $0$ as $n\to \infty$, yielding a superlinearly small $\pimin$ in contrast to~\eqref{SMNR}.
\end{rem}

\begin{rem}
    Alternatively,~\eqref{MHUZ1} can be read as 
    \begin{align}\label{SPDK}
        \left(1+\frac{\hat{H}^-}{\phi (a_{0})}\right)^{-1}\frac{\log \pimin^{-1}}{\log n}
        \to 1 \quad \text{ in probability,}
    \end{align}   
     From our proof, one can obtain a upper bound \new{on the distance to the limit of}
    order ${(\log n)}^{-1/2}$. Whp bounds for $\pimin$ that are tight up to a polylogarithmic factor, like the
    ones in~\eqref{SMNR}, are unlikely to hold in this setting due to the use of large deviation theory.
    In fact, we believe that the rate of convergence ${(\log n)}^{-1/2}$ cannot be vastly improved, as
    this is the rate of convergence in Cramer's theorem (see \autoref{thm:cramer}).
\end{rem}

\begin{rem} (Simple graphs)
Let \(\vecGns\) be \(\vecGn\) conditioned on being a simple directed graph\footnote{As the degrees are uniformly bounded and $\sum_{v\in [n]} d^{+}_{v} = \sum_{v \in [n]}d^{-}_{v}$, it can be checked that the hypothesis of the Fulkerson-Chen-Anstee theorem are satisfied for $n$ sufficiently large with respect to $M$, and $\vbfd_n$ is digraphic.}. Then
    \(\vecGns\) is distributed uniformly among all simple directed graphs with degree sequence
    \(\vbfd_{n}\). It is well-known that the probability of \(\vecGn\) being simple is bounded away from
    zero when the maximum degrees are bounded (see, e.g.,~\cite{blanchet2013, janson2009}). Thus,~\autoref{thm:main} also holds for
    \(\vecGns\).
\end{rem}

\begin{rem}[Minimum out-degree and uniqueness of $\pi$]\label{rem:SKDJ}
  \autoref{thm:main} requires $\delta^+\geq 2$. If $\delta^+=0$, then the stationary distribution
    is either trivial or non-unique. If $\delta^+=1$ and $\p{D^+=1}$ is bounded away
    from $0$ as $n\to \infty$, \new{then with constant probability  $\vecGn$ will have more than one vertex whose out-neighbourhood is a single loop. Each of these vertices form their own strongly connected component, giving rise to multiple trivial stationary distributions.}~\autoref{prop:strongly} shows
    that under $\delta^+\geq 2$, $\walk$ has a unique stationary distribution whp. Uniqueness of the
    equilibrium measure whp can be also shown if $\delta^+=1$ and $\p{D^+=1}=o(1)$. It is likely
    that the conclusion of~\autoref{thm:main} still holds in this situation. 
\end{rem}

\begin{rem}[Maximum stationary value]\label{rem:SDKA}
In this paper we turned our attention to $\pimin$. By an averaging argument, $\pimax\geq 1/n$. Moreover, one can
check that the proof of the second inequality in~\eqref{SMNS} (see~\cite[Section 3.5]{caputo2020a})
does not use any condition on $\delta^-$. Therefore, in the setting of
\autoref{thm:main}, for every $\varepsilon>0$ and whp
\begin{align}\label{FDOE}
n^{-1-\varepsilon} \leq \pimax\leq  n^{-1+\varepsilon}. 
\end{align}
We refer the interested reader to~\cite{cai2021a} for the behaviour of $\pimax$ when the maximum
in-degree goes to infinity as $n\to \infty$, and its connection to the Pagerank random walk.
\end{rem}

\begin{rem}[Explicit polynomial exponents for $\pimin$]
    Since for a general distribution there is no closed-form expression for $I(z)$, \autoref{thm:main} provides an implicit polynomial exponent. Nevertheless, $I(z)$ can be
    explicitly computed for some particular bi-degree sequences, yielding explicit polynomial
    exponents.  In~\autoref{sec:examples}, we give two such examples: one where $I(z)=\infty$ for
    any $z\neq \hat{H}^-$, and one where $I(z)$ can be expressed in terms of the large deviation rate function of a
    Bernoulli random variable. If explicit bounds are required, there is
    an extensive literature on concentration inequalities for the sum of independent bounded random
    variables, such as Bernstein's and Bennett's inequalities, (see, e.g.,~\cite{petrov1975}) which
    provide lower bounds on $I(z)$ in terms of the  moments of the distribution.
    Alternatively, rigorous numerical bounds can be computed with interval arithmetic libraries such
    as~\cite{sanders2020}.
\end{rem}

\begin{rem}(Large polynomial exponents)
    We can choose \(\vbfd_{n}\) to make the polynomial exponent in~\eqref{MHUZ1} arbitrarily large. 
    For instance, fix $M\in \mathbb{N}$ and for $n$ multiple of $M-1$ consider the degree sequence of length $n$ with $n/(M-1)$ vertices of degree $(M,2)$ and the rest of degree $(1,2)$.
	As $\delta^-\geq 1$, then $s^-=1$. Since the graph is $2$-out-regular, we have $\hnu^-=2$, \(\hat{H}^-=\log{2}\), and \(I(a\hat{H}^-)=\infty\) for $a\neq 1$ and $I(\hat{H})=0$. So the minimum of $\phi(a)$ is attained at $a_0=1$ and \(\phi(a_{0}) = \abs{\log \hnu^-}\). From the generating function of $D$, we obtain \(\hnu^- = 1 + O(M^{-1})\) as $M\to \infty$.
    Thus, \(\hat{H}^-/\phi(a_{0}) = \Omega(M)\) as  $M\to \infty$ and we can make the polynomial exponent in~\eqref{MHUZ1} as large as we want by increasing
    \(M\).
\end{rem}

\begin{rem}\label{rem:empirical}(The lower tail)
Consider the empirical measure $\psi=\frac{1}{n}\sum_{v \in[n]} \delta_{\{n\pi(v)\}}$; that is,
\(n\) times the stationary value of a uniform random vertex. In~\cite{bordenave2018}, it was shown that, \new{under suitable conditions,} there exists
a deterministic law $\cL$ such that 
$d_{\cW}(\psi,\cL) \to 0$ in probability, where $d_{\cW}$ is the $1$-Wasserstein metric. Our proof of~\autoref{thm:main} allows us to control the
lower tail of $\psi$: for every $\alpha \in [0,\hat{H}^-/\phi(a_0)]$, 
letting $\beta = \frac{\alpha \phi \left(a_0\right)}{\hat{H}^-} \in [0,1]$, we have  
\begin{equation}\label{EABU}
    \bE[\psi((0, n^{-\alpha}])] 
    =
    \frac{1}{n}\sum_{v \in [n]} \p{0<\pi(v) \le n^{-(1+\alpha)}}
    = 
    n^{-\beta+o(1)}.
\end{equation}
By computing the second moment, it can be shown that $\psi((0, n^{-\alpha}])$ is concentrated
around its expected value. See~\autoref{sec:empirical} for the proof of~\eqref{EABU}.
\end{rem}

\new{A strongly connected component is \emph{closed} if there is no path from it to any other component.}
Let $G$ be a directed graph with vertex set $[n]$ with \new{a unique closed} strongly connected component $\cC_0$ with vertex set $\cV_0$. 
Let $\walk$ be a simple random walk on $G$. Let
\new{$\tau_{v}\coloneqq\inf\{t\geq 0: Z_t=v\}$}. The \emph{hitting time} of $\walk$ is defined as
\begin{equation}\label{IZSG}
    \thit\coloneqq \max_{\substack{u\in [n] \\ v\in \cV_0}} \bE[\new{\tau_{v}\mid Z_0=u}]
.
\end{equation}
Let \new{$\tau^C\coloneqq \inf\{t\geq 0: \cV_0 \subseteq \cup_{r=0}^t \{Z_r\}\}$}. The \emph{cover time} of $\walk$ is defined as
\begin{equation}\label{ALFW}
\new{\tcov\coloneqq \max_{u\in [n]} \bE[\tau^C\mid Z_0=u]}
.
\end{equation}
 
As a application  of the proof of~\autoref{thm:main}, we determine the hitting and the cover time of the directed configuration model, up
to subpolynomial terms.
\begin{thm}\label{thm:hitcov}
    Under the hypothesis of~\autoref{thm:main}, for any $\varepsilon>0$, whp
    \begin{equation}\label{MHUZ}
		n^{1+\hat{H}^-/\phi(a_{0})-\varepsilon}\leq \thit\leq  \tcov \leq  
        n^{1+\hat{H}^-/\phi(a_{0})+\varepsilon}.
    \end{equation}
\end{thm}

To simplify the notation, we avoid using \(\ceil{\cdot}\) and \(\floor{\cdot}\) to make certain
parameters integers. Such omissions should be clear from the context and do not affect the validity of the proofs.


\section{Outline of the proof}\label{sec:strategy}

First of all, we will assume throughout the proof that the minimum in-degree satisfies $\delta^{-}\in\{0,1\}$, in which case $\hnu^->0$; we will crucially use the latter in the proofs. See~\autoref{rem:mindeg} for a discussion of the case \(\delta^{-}\geq 2\).

One of main parts of the proof is to understand the depth and shape of the in-neighbourhoods of vertices. In fact, it will be more convenient to explore the in-neighbourhoods of heads. Our proof will be guided by the following well-known heuristics for the (directed) configuration model. Let $f^-$ be a head whose in-neighbourhood we aim to explore. Pair $f^-$ with $f^+$, a tail chosen uniformly at random among all tails in $\vecGn$. Let $V$ be the (random) vertex incident to $f^+$. Since a vertex with degree $(k,\ell)$ has $\ell$ tails that could be paired with $f^-$, the degree distribution of $V$ is $\dout$, \new{defined in~\eqref{UHJC}}. On the one hand, this provides $\doutm$ new heads to keep exploring the in-neighbourhood from $V$. On the other hand, a random walk at $V$ has \new{only one way out of $\doutp$ many to move to $f^-$}. This idealized picture is not far from what actually happens.

To understand the growth of in-neighbourhoods we will couple them with marked branching processes (see~\autoref{sec:branching} for a precise definition) governed by the distribution
\begin{align}\label{eq:SADD}
\eta=(\xi,\zeta)\eql \dout,
\end{align}
where $\xi$ is the offspring distribution corresponding to the number of heads incident to a vertex chosen as explained above, and  $\zeta$ is the marking distribution corresponding to the number of tails of that vertex. In this sense the  parameters $\nu,\hnu,\hat{H}$ that will be defined in \autoref{sec:branching}, should be understood as the branching process analogues of the degree distribution dependent parameters $\nupm,\hnu^-,\hat{H}^-$ defined in the introduction. 

In~\autoref{sec:branching} we will do the analysis of the marked branching process. \new{By~\eqref{SPEK} and our hypotheses, we conclude that $\nupm \geq \delta^+\geq 2$ and the marked branching process typically grows exponentially fast. When coupled with the digraph exploration process}, this corresponds to the behaviour of typical heads in $\vecGn$; but those play no role when determining extremal parameters. However, as $\delta^-\in \{0,1\}$, exceptionally, the branching process can remain ``small'' without going extinct. Precisely, the probability that a process survives but stays small for $t$ generations is roughly $(\hnu^-)^t$, \new{where $\hnu^-$ is the subcritical in-expansion rate defined in~\eqref{eq:hnu}}. In addition, and also exceptionally, the marks along \new{short branches of the process} can be extraordinary large; this happens with probability related to the rate function $I(z)$ of the random variable $\log \tilde{D}_{\tout}^+$.~\autoref{thm:LB} and~\autoref{thm:UB} estimate from below and from above, respectively, the probability of growing small positive processes with large marks, and are the main ingredient of our proof. Approximately and informally speaking, they state that for $a\geq 1$ and sufficiently large $t\in \mathbb{N}$ and $\omega\new{=\omega(t)}$ satisfying some conditions, with probability $e^{-a\phi(a)t}$, 
\begin{itemize}
\item[i)] the weight of paths of length $t$ to the root (defined as a function of the marks) is at most $e^{-a\hat{H}t}$;
\item[ii)] there is at least one and less than $\omega$ elements in the $t$-th generation of the process.
\end{itemize}
From i), we can read the importance of the minimizing value $a_0$: since $\vecGn$ has linearly many heads, we expect to see events related to heads of probability roughly $1/n$. \new{That is, for a fixed $a\geq 1$, if we choose $t$ to be 
\begin{align}\label{SODW}
t_0(a):=\frac{\log{n}}{a\phi(a)},
\end{align}
then $e^{-a\phi(a)t}=1/n$, and we expect to have heads $f^-$ such that all the paths of length $t_0(a)$ ending at them have weight at most 
\begin{align}\label{SOEI}
e^{-a\hat{H}t_0(a)} = n^{-\hat{H}/\phi(a)}.
\end{align}
}

%
Therefore, the weight of paths of length $t_0(a)$ is minimized at $a=a_0$. As it turns out, this weight is a good approximation of the probability that a random walk starting at the \new{$t_0(a)$-th in-neighbourhood of $f^-$, reaches the head after $t_0(a)$ steps.} 

From ii), the size of the $t_0(a)$-th in-neighbourhood, $N$, is at most $\omega$, which in our proof is \new{polylogarithmic in $n$. By~\autoref{rem:SDKA}, $\pi(N)$ is of order $n^{-1+o(1)}$}. 

By stationarity, for any given $t\geq 0$, $\pi(v)$ can be expressed as the sum over all vertices $u$ \new{at distance $t$ to $v$}, of $\pi(u)$ times $P^t(u,v)$, i.e. the probability a random walk moves from $u$ to $v$ in $t$ steps. By the previous considerations, we find that the minimum of the stationary distribution is $n^{-(1+\hat{H}/\phi(a_0))}$ up to subpolynomial terms.


In~\autoref{sec:graph}, we make explicit the coupling between branching processes and digraph exploration processes. This allows us to transfer the results of~\autoref{sec:branching} to the directed configuration model, which is done in~\autoref{sec:pi:min}. Before the proof of~\autoref{thm:main}, in~\autoref{sec:SCC} we show that under the hypothesis of the theorem, there exists a unique closed strongly connected component, and thus $\pimin$ is well-defined. \new{As with the exploration of in-neighbourhoods, it is more convenient to study the random walk on heads}, as presented in~\autoref{sec:RW_heads}. 
We then in~\autoref{sec:lower} proceed to bound $\pimin$ from below. The idea of the proof is simple: Given two heads $e$ and $f$ we compute the probability the walk moves from $e$ to $f$ in a number of steps (that depends on the profile of the in-neighbourhood of $f$); see~\autoref{lem:PT}. This is done by, first, growing the out-neighbourhood of $e$ using the notion of ``nice paths'' introduced by~\cite{bordenave2018} and, second, growing the in-neighbourhood of $f$ until it it has polylogarithmic size, building on the ideas developed in~\autoref{sec:branching} and~\autoref{sec:graph}. A Bernstein-type inequality for permutations allows us to show concentration of the edges from the out- to the in-neighbourhoods previously constructed; \new{here, to ensure bounded increments, we are required to study a truncated version of path-weights (\autoref{prop:martingale})}. 
\autoref{sec:UB} contains the proof of the upper bound on $\pimin$. The proof requires to find a head $f$ that has the right in-neighbourhood profile of depth $t_0(a_0)$, as defined in~\eqref{SODW}. This is done via a second moment argument and using again the results in~\autoref{sec:branching} and~\autoref{sec:graph}. Finally, in~\autoref{sec:empirical} we prove~\autoref{rem:empirical} on the behavior of the lower tail of the empirical distribution.

\autoref{sec:app} is devoted to prove~\autoref{thm:hitcov}. The argument is simple and uses Matthew's bound. We also provide a couple of interesting examples where the polynomial exponent can be computed explicitly.

\section{Supercritical marked branching processes}\label{sec:branching}

In this section, we prove some general results for  marked branching processes with distribution $\eta$.
This part of the paper may be of independent interest.

Let us stress that throughout this section we make sure that all error terms are uniformly bounded only in terms of \new{the maximum values that $\eta$ can attain and the degree of precision of the statements, and conditioned on that}, \emph{do not depend} on the law of $\eta$. This is crucial later in the paper as $\eta$ implicitly depend on the number of vertices $n$, \new{but it is uniformly bounded by hypothesis.}

\subsection{Marked branching processes}\label{sec:mark}

Let \(\eta=(\xi,\zeta)\) be a random vector on \(\dsZ^2_{\ge 0}\) and let
\new{\((\eta_{i,t})_{i\ge 1, t \ge 0}\), with $\eta_{i,t}=(\xi_{i,t},\zeta_{i,t})$}, be a sequence of iid (independent and identically
distributed) copies of \(\eta\).

The \emph{branching process}, also known as the
\emph{Bienaym\'e-Galton-Watson tree}, \((X_t)_{t \ge 0}\) with \emph{offspring distribution} \(\xi\) is defined by
\begin{equation}\label{eq:BP}
    X_{t}=
    \begin{cases}
        1
        &
        \qquad \text{if } t =0
        ,
        \\
        \sum_{i=1}^{X_{t-1}} \xi_{i,t-1}
        &
        \qquad \text{if } t \ge 1
        .
    \end{cases}
\end{equation}
We call $X_t$ the $t$-th \emph{generation} and refer to $(X_r)_{t\geq r\geq 0}$ as the \emph{first $t$ generations of the process}.

\new{As usual, we associate to the branching process a rooted (possibly infinite) tree of \emph{individuals} labelled with indices as follows:  Initially, we add a single individual with index $(1,0)$, the \emph{root} of the tree, corresponding to $X_0=1$. Then, iteratively and for $t\geq 1$, for each individual with index $(i,t-1)$ with $i$ from $1$ to $X_{t-1}$, we add $\xi_{i,t-1}$ individuals, we call them \emph{children} of $(i,t-1)$ and we give them indices $(j,t)$ for increasing and distinct integers $j\geq 1$. Given an individual \( (i,t)\), we call \(i\) its \emph{cousin
index} and \(t\) its \emph{generation index}. The associated tree is denoted by $\GW$ and the tree generated by the first $t$ generations of the process, which contains the individuals up to generation $t+1$, by $\GW_t$.}

\new{The \emph{marked branching process} with  distribution \(\eta=(\xi,\zeta)\) is the pair $(X_t,L_t)_{t\geq 0}$ where $(X_t)_{t\geq 0}$ is a branching process with offspring distribution $\xi$ and, for $t\geq 0$, $L_t=(\zeta_{1,t},\zeta_{2,t},\dots)$ is a sequence of marks in $\mathbb{Z}_{\geq 0}$. It is useful to think about it as the individual $(i,t)$ being marked with $\zeta_{i,t}$.
To keep the notation light and since all branching process appearing in this paper are marked, we will abuse notation and write \((X_t)_{t \ge 0}\) for the marked branching process $(X_t,L_t)_{t\geq 0}$, and $\GW_t$ for the tree generated by the first $t$ generations of the process where every individual at generation $r\leq t$ is assigned its corresponding mark.
}

Let $G_{\eta}$ be the bivariate probability generating function of $\eta$, i.e.,
\begin{equation}\label{BAEI}
    G_{\eta}(z,w) \coloneqq \sum_{k,\ell \ge 0} \p{\eta = (k,\ell)} z^{k}w^\ell,
\end{equation}and let \(G_{\xi}(z)\coloneqq G_{\eta}(z,1)\) be the probability generating function of $\xi$.
Define
\begin{equation}\label{XQMZ}
    \nu \coloneqq \E{\xi} = G_{\xi}'(1) 
    .
\end{equation}

All the results in this section will hold under the following assumption.
\begin{condition}\label{cond:BP}
The distribution $\eta=(\xi,\zeta)$ satisfies 
\begin{itemize}
	\item[(i)] supercritical: \new{$\nu>1$};
	\item[(ii)] bounded support: there exists $M\in \N$ such that $\xi,\zeta\leq M$;
	\item[(iii)] \new{small values probability for} $\xi$: $\p{\xi\in\{0,1\}}>0$;
	\item[(iv)] minimum value for $\zeta$:  $\zeta\geq \new{2}$.
\end{itemize}
\end{condition}



Given $i\in [X_t]$ and $r\in \{0,\dots,t\}$, let $p^{r}(i, t)\in [X_{t-r}]$ be
the cousin index of \( (i,t)\)'s ancestor \(r\) generations away.  We write \(p(i, t)=p^{1}(i, t)\)
and note \(p^{0}(i, t)=i\). For $i\in [X_t]$, define
\begin{align}\label{OSKE}
    \Gamma_{i,t} &\coloneqq \prod_{r=1}^{t} \frac{1}{\zeta_{p^{t-r}(i, t),r}}
    \text{ and }
    \Gamma_t \coloneqq \sum_{i\in [X_t]} \Gamma_{i,t}.
\end{align}

We provide some insight as to the importance of these parameters. One may think about the mark of the element $(i,r)$ as assigning $\zeta_{i,r}$ ``doors'' to it, from which only one leads to its parent and the other ones exit the process. Imagine a particle that starting at the element $(i,t)$ successively moves through uniformly randomly chosen doors until it exits the process or reaches the root. Then, $\Gamma_{i,t}$ is precisely the probability the particle reaches the root, while $\Gamma_{t}$ is the expected number of particles in generation $t$ that reach the root.

In this section, we will mostly be interested in the sequence of random variables $(\Gamma_t)_{t \ge
1}$.  Let \new{$(\cF_t)_{t\geq 0}$ be a filtration where $\cF_t$} is the $\sigma$-algebra generated by the collection of random variables 
$(\eta_{i, r})_{i \ge 1, t \new{\geq} r 
\ge 0}$.
\new{Note that $X_{t+1}$, $\GW_t$ and $\Gamma_t$ are measurable with respect to $\cF_{t}$}. Moreover, $\Gamma_t>0$ if and only if $X_{t}>0$.

\begin{rem}[Weighted branching trees]\label{rem:WBP}
The marked branching process $(X_t)_{t\geq 0}$ together with the random process $(\Gamma_t)_{t\geq 0}$ can be understood as a variant of a \emph{weighted branching process} (WBP)~(see e.g.~\cite{rosler1993}) where the offspring and the weight distributions are not independent.
In the literature, these processes are called \emph{weighted branching trees} (WBT) and have already been used to study the directed configuration model~\cite{chen2016}. In the context of WBT, $\Gamma_{t,i}$ represents the weight of the path from the root to the $i$-th individual at generation $t$, and $\Gamma_t$ represents the total weight of the paths at generation $t$. For the sake of consistency with previous related work~\cite{caputo2020a}, we keep the notation $\Gamma_t$ here instead of the notation coming from WBP and WBT, but we will occasionally use properties that are known to hold for WBT without proving them. For instance, since $\bE[\xi/\zeta] \in (0,\infty)$ by~\autoref{cond:BP}, $\Gamma_t \E{\xi/\zeta}^{-t}$ is a martingale with respect to \((\cF_{t})_{t \ge 0}\) (see e.g.~\cite{chen2016}).
\end{rem}

%
%

\subsection{Conditioned branching processes}

Before introducing the results, some more definitions are needed.

\subsubsection{Conditioned on extinction}

Let \(s \coloneqq \p{\cap_{t \ge 0} \{X_{t}
> 0\}}\) be the \emph{survival probability} of $(X_t)_{t\geq 0}$; by~\autoref{cond:BP} (i), $s\in (0,1]$.
The \emph{conjugate probability
distribution} of \(\xi\), denoted by \(\xihat\), is defined as
\begin{equation}\label{eq:conj}
    \p{\xihat=k} \coloneqq (1-s)^{k-1}\p{\xi=k} \qquad \text{for }k\geq 0,
\end{equation}
when \(s < 1\), while
\begin{equation}\label{XHNO}
    \p{\xihat = 1}  = \p{\xi =1}, \qquad \p{\xihat = 0} = 1-\p{\xi = 1},
\end{equation}
when \(s = 1\).
Define the \emph{subcritical expansion rate} as
\begin{equation}\label{PXNE}
    \hnu \coloneqq \e[\xihat] =  G_{\xi}'(1-s) \in (0, 1].
\end{equation}
\new{Note that $\hnu$ is always positive.} Indeed, if $s<1$, then clearly $\hnu>0$; otherwise $s=1$, which implies $\p{\xi=0}=0$, and by~\autoref{cond:BP} (iii), $\hnu=\p{\xi=1}=\p{\xi\in\{0,1\}}>0$.

\new{Next lemma states that, given a bound $M$ on the coordinates of $\eta$, some of the key parameters of the branching process are uniformly bounded away from $0$ or $1$. This will be important in \autoref{sec:pi:min} to obtain explicit error terms, when the distribution $\eta$ depends on $n$ but has bounded support.}
\begin{lemma}\label{lem:bounded_away}
Let $(X_t)_{t\geq 0}$ satisfy~\autoref{cond:BP}. Then, there exists $c=c(M)>0$ such that
\begin{align}
s\geq c,\qquad \p{\xihat=0}\geq c, \qquad  \text{and} \qquad \hnu \leq 1-c. 
\end{align}
\end{lemma}
\begin{proof}
By~\autoref{cond:BP} (i), $s\geq c$ for an absolute constant $c>0$ \new{(see e.g.~\cite[Theorem 3.1]{vanderhofstad2016}) proving the first inequality.}  Moreover, by~\autoref{cond:BP} (i)-(ii)
\begin{align}\label{WMEO}
\p{\xi\leq 1}\leq 1-\frac{1}{M-1}\leq 1-\frac{1}{M}.
\end{align}
If $s=1$, by~\eqref{XHNO},~\eqref{PXNE} and~\eqref{WMEO}, 
\begin{align}
\p{\xihat=0}\geq \frac{1}{M} , \qquad \hnu = \p{\xihat=1}\leq 1-\frac{1}{M}. 
\end{align}
So let us assume that $s<1$, which implies $\p{\xi=0}>0$. We use the following upper bound on the extinction probability (see~\cite[Theorem 2.1.b]{from2007})
\begin{align}\label{SDKA}
1-s\leq \frac{\p{\xi=0}}{1-\p{\xi=0}-\p{\xi=1}}\leq M\p{\xi=0}, 
\end{align}
where we used~\eqref{WMEO} in the last inequality. As $s<1$, by~\eqref{eq:conj} and~\eqref{SDKA}
\begin{align}
\p{\xihat=0}= \frac{\p{\xi=0}}{1-s}\geq \frac{1}{M},
\end{align}
proving the second inequality.

It remains to bound $\hnu$. 
\new{If $\p{\xi=0}\leq 1/3M$, by~\eqref{SDKA}, $1-s\leq 1/3$.} Since $G''_\xi(x)\geq 0$ for $x\in [0,1)$, we obtain
\new{
\begin{align*}
\hnu = G'_\xi(1-s)&\leq G'_\xi(1/3)\\
&= \sum_{k\geq 1} k 3^{-(k-1)}\p{\xi=k}\\
&\leq \p{\xi= 1}+(2/3) \Big(1-\p{\xi\leq 1}\Big)\\
&\leq 2/3+ (1/3)\p{\xi\leq 1}\\
&\leq 1-1/3M,
\end{align*}
where we used~\eqref{WMEO} in the last inequality}. 

\new{Suppose now that $\p{\xi=0}>1/3M$. Consider $f(x)=G_\xi(x)-x$. Then $f(0)=\p{\xi=0}$ and $f(1-s)=f(1)=0$. By the mean value theorem, there exists $x_0\in (0,1)$ such that 
$$
f'(x_0) = - \p{\xi=0}.
$$
Since $f'(x)$ is increasing such $x_0$ satisfies $f(x_0)<0$ and
$$
\hnu=f'(1-s)+1\leq f'(x_0)+1 =1-\p{\xi=0}< 1-1/3M,
$$
which concludes the proof.}

\end{proof}


The following \emph{duality} property of branching processes is well-known:
\begin{thm}[see, e.g., Theorem 3.7 in~\cite{vanderhofstad2016}]\label{thm:branching}
    Let \( (X_{t})_{t \ge 0}\) be a branching process with offspring distribution \(\xi\) and
    survival probability \(s\).
    If \(s < 1\), then the branching process \( (X_{t})_{t \ge 0}\) conditioned on extinction is distributed as a branching process
    with offspring distribution \(\xihat\).
\end{thm}

\subsubsection{Conditioned on survival}
\label{sec:cond:surv}

Let \( (X_{t}^{*})_{t \ge 0} \subseteq (X_{t})_{t \ge 0}\) be the subprocess of the individuals that have infinite
progeny. Thus, $\p{X_{0}^*=0}=1-s$ and $\p{X_{0}^*=1}=s>0$.  Conditioning on the event $\{X_{0}^{*} = 1\}$, i.e.,
survival of \((X_{t})_{t \ge 0}\), \( (X_{t}^{*})_{t \ge 0}\) is a branching process with offspring distribution \(\xistar\), defined by
\begin{equation}\label{JDOE}
    \p{\xi^*=k}
    = \new{\p{X_1^{*}=k \mid X_{0}^{*}=1}}=\frac{1}{s}\sum_{k' \ge k} \p{\xi=k'} \binom{k'}{k} s^k (1-s)^{k'-k} 
    = \frac{s^{k-1}}{k!} G_{\xi}^{(k)}(1-s)
    ,
    \text{for }
    k \ge 1
    ,
    \qquad
\end{equation}
\new{where $G_{\xi}^{(k)}$ denotes the $k$-th derivative of $G_{\xi}$. Indeed, provided that an individual has $k'$ children, the number of them that have infinite progeny is a binomial random variable with $k'$ independent trials and success probability $s$.}

\new{By the Taylor formula on $G'_\xi$ we obtain}
\begin{equation}\label{CQVU}
    \E{\xistar}
    =
    \sum_{k\geq 1} \frac{k s^{k-1} G_{\xi}^{(k)}(1-s)}{k!}
    =
    \sum_{k\geq 0} \frac{s^k G_{\xi}^{(k+1)}(1-s)}{k!}
    =
    G_{\xi}'(1) = \nu
    .
\end{equation}
Moreover,
\begin{equation}\label{SOWK}
\p{\xi^*=1}
=
G_{\xi}'(1-s)
= \hat\nu
.
\end{equation}
Note that if $s=1$, then  $(X_{t}^{*})_{t \ge 0} = (X_{t})_{t \ge 0}$ and all the probabilities and expected values above coincide with the ones of $\xi$.

Let \(\proctildet \subseteq \procstart\) be the subprocess of the individuals that
have \emph{exactly one} \new{children in $\procstart$}. 
(Note that \(\proctildet\) is not necessarily a connected process: \new{it produces a forest composed of paths.})
Then
conditioned on the individual \( (i,t) \) being counted in \( \proctildet\), \(\xi_{i,t}\)
is distributed as $\tilde{\xi}$, defined by
\begin{align}\label{SJWU}
    \p{\tilde \xi=k}
    = \p{X_{1}=k\mid X_{1}^*=1}
    =
    \new{\frac{k}{\hat\nu}\cdot\p{\hat\xi=k}}
    ,
    \quad \text{for }
    k\geq 1
    .
\end{align}
In particular, if $s=1$ then $\p{\tilde\xi=1}=1$.

Let $\tilde \eta = (\tilde \xi, \tilde \zeta)$ be the distribution of \(\eta_{i,t}\) conditioned on \(
(i,t) \in \proctildet\).  Define the  \emph{subcritical entropy} of $\eta$ as\new{
\begin{align}\label{EFND}
\hat{H}\coloneqq \bE[\log\tilde \zeta] = \frac{1}{\hat \nu}\sum_{k,\ell\geq 1} k \log\ell\cdot \p{(\hat \xi,\zeta)=(k,\ell)}
,
\end{align}}

When comparing~\eqref{EFND} with~\eqref{SKFN}, one observes that the summands differ by a factor  $\ell$. This is explained by the fact that we will choose $\eta=D_{\text{o-sb}}$, which already carries the additional term $\ell$; see~\eqref{UHJC}.


Later, we will also consider the \emph{inhomogeneous} branching process \(\procspinet\) in which the root has offspring
distribution \(\tilde \xi -1\) and all other individuals have offspring distribution \(\hat{\xi}\).
Note that such a process will almost surely become extinct.

\subsection{Large deviation theory}\label{sec:LDT}

We will use Cramér's theorem, a classical result in large deviation theory.
\begin{thm}[see, e.g., Corollary~2.2.19 in~\cite{dembo2010}]\label{thm:cramer}
Let $Z_1,\dots,Z_t$ be iid copies of a random variable $Z$ satisfying $\bE[e^{\lambda Z}]<\infty$ for all $\lambda\in \mathbb{R}$. Define
\begin{equation}\label{IGCL}
    \bar{Z}_t= \frac{1}{t}\sum_{r=1}^t Z_r
    .
\end{equation}
Then, for any $z \geq \bE[Z]$
\begin{equation}\label{ZLVD}
\lim_{t\to \infty}\frac{1}{t}\log \p{\bar{Z}_t\geq z}= -I(z)
,
\end{equation}
where 
\begin{equation}\label{JRPL}
    I(z)\coloneqq \sup_{\lambda\in \mathbb{R}} \{\lambda z -\log \bE[e^{\lambda Z}]\}
    ,
    \qquad
    \text{for }
    z \in \dsR
    ,
\end{equation}
is the
\emph{Fenchel-Legendre transform} of the cumulant generating function of \(Z\).
\end{thm}

In the applications of Cram\'er's theorem, we will have $t$ large, albeit finite. 
Thus, it will be convenient \new{to avoid the limit statement in~\eqref{ZLVD}, and use results with an explicit error term} which we obtain by inspecting the proof of the theorem. By Markov's inequality, for any $z\geq \E{Z}$
\begin{align}\label{WOEW}
\p{\bar{Z}_t\geq z} \leq \inf_{\lambda\in \mathbb{R}} e^{-\lambda z t } \E{e^{\lambda Z}}^t = e^{-I(z)t}.
\end{align}
For the lower bound we use the \emph{exponential tilting} trick. By hypothesis, $Z$ has finite variance.  Fix $z\geq \E{Z}$. 
\new{The supremum in~\eqref{JRPL} is attained~\cite[Lemma~2.2.5]{dembo2010}}, so let $\lambda_0=\lambda_0(z)>0$ be the value of $\lambda$ that \new{maximises} $ \lambda z -\log \bE[e^{\lambda Z}]$. If $\mu$ is the law of $Z$, let $Y$ be the random variable defined by
\begin{align}\label{SPEP}
\p{Y\leq y}=\frac{1}{\bE[e^{\lambda_0 Z}]} \int_{-\infty}^y e^{\lambda_0 t}\,d \mu(t),
\end{align}
which satisfies $\E{Y}=z$ and has finite variance $\sigma_Y^2$. Then, for any $\delta>0$ and $t\geq 2\sigma^2_Y/\delta^2$, we have
\begin{align}\label{EPWQ}
\p{\bar{Z_t}\geq z-\delta}
&\geq \p{\bar{Z_t}\in [z-\delta ,z+\delta]}
\geq e^{-(I(z)+\delta)t}\p{Y\in [z-\delta ,z+\delta]}
\geq \frac{1}{2}e^{-(I(z)+\delta)t},
\end{align}
where the last step uses Chebyshev's inequality.


From now on we will set $Z$ to be the discrete random variable with distribution \(\log\tilde \zeta\); recall that $\E{Z}=\hat H$. By~\autoref{cond:BP}, it has support in \(\{\log 2, \log 3,\dots, \log M\}$, where \(M\) is a fixed integer, and finite moment generating function.  Thus, $I(z)$ will refer to the large deviation rate function of $Z$. \new{In particular, \(I(z)\) is continuous, 
    non-decreasing for $z\geq \hat H$, non-increasing for \(z\leq \hat H\), and $I(\hat{H})=0$.}
    The proofs of these properties follow along the lines of those
     in~\cite[Lemma~2.2.5]{dembo2010}, and we omit them here.

\subsection{Subcritical growth: a lower bound}\label{sec:LB}

\new{In this section we show that, with some exponentially large probability, the marked branching process satisfies some desirable conditions, which will be later used to show the existence of vertices in the random graph with small stationary value.}
\begin{thm}
    \label{thm:LB}
    Let \(\procall\) be a marked branching process with distribution \(\eta=(\xi,\zeta)\)
    satisfying~\autoref{cond:BP} \new{with $M\in \mathbb{N}$}. Then there exists $c=c(M)$ such that for any $a \in [1, \log(M)/\hat{H}]$, $\delta>0$, $t\in \mathbb{N}$ sufficiently large with respect to $\delta$ and $M$, and $\omega\geq t$,
    \begin{equation}\label{DMEX}
        \p{\left\{0 < \Gamma_t< e^{-(a \hat{H}-\delta) t}\right\}\cap \bigcap_{r=1}^t \{0<X_r<\omega\}}
        \geq
        c\new{\hnu} \exp \left\{-\left(|\log{\hnu}| + I(a\hat{H})+\delta \right)t\right\},
    \end{equation}
    \new{where $\Gamma_t$ is as defined in~\eqref{OSKE}.}
\end{thm}

The event inside the probability in the LHS of~\eqref{DMEX} can be split into two natural parts. The first one, bounding $\Gamma_t$, implies that \new{\emph{all}} leaf-to-root paths in the tree have \new{many} marks that are anomalously large, which amounts to the contribution $I(a\hat{H})$ \new{in the RHS of~\eqref{DMEX}} and relates to property (2) in the discussion below~\autoref{thm:main}. The second one, regarding $X_r$, enforces the tree to survive \new{up to generation $t$} but also to stay anomalously small for its height, implying that there are less than $\omega$ leaf-to-root paths of length $t$, which amounts to the contribution $|\log{\hnu}|$ \new{in the RHS of~\eqref{DMEX}} and relates to property (1). \new{It is worth stressing that in coming applications $\omega$ will be chosen polynomial on $t$.}


\begin{proof}
For the sake of simplicity, we first prove the theorem assuming that \(1\) is in the support of
\(\xi\). The modifications needed otherwise, are detailed at the end of the proof.

Consider the events 
\begin{equation}\label{eq:SOEV}
E_1=\{X_{\new{t+1}}^*=1\},\, E_2=\{X_{\new{t+1}}=1\},\, \text{ and } \,E_3=\bigcap_{r=1}^t \{0<X_r<\omega\}.
\end{equation}
The idea of the proof is to lower bound the probability \new{of the event in the LHS of~\eqref{DMEX} conditioned on the intersection of these
events}.

When the event \(E_1\) happens, \new{then $X_r^*=1$ for $r\in \{0,1,\dots, \new{t+1}\}$} and we call the first $\new{t+1}$ generations of
$\procstar$ the \emph{spine}. We may assume without loss of generality that the spine of
individuals \new{with infinite progeny} corresponds to the first individual in each generation, since reordering cousin indices
does not change the value of \(\Gamma_{t}\) or $X_r$.  Moreover, the number of children (in
\(\procall\)) and the mark of each individual in the spine is jointly distributed as $\tilde \eta =
(\tilde \xi, \tilde \zeta)$, \new{defined as in~\autoref{sec:cond:surv}}.  

For \(\mathbf{l}= (\ell_{1},\dots,\ell_{\new{t+1}})\in \cL \coloneqq \{1\} \times [M]^{t}\), let
\(F(\mathbf{l})\) be the intersection of the event $E_1$ and the event \(\{\xi_{1,0} = \ell_{\new{t+1}},
\dots, \xi_{1,t} = \ell_{1}\}\), i.e., the \(r\)-th generation of the spine has \(\ell_{\new{t+1}-r}\) children.  (We
require \(\ell_{1}=1\), which by assumption is in the support of \(\xi\), so that \(\new{\cup_{\mathbf{l}\in \cL}}F(\mathbf{l}) \cap E_{2}\) is not
empty.) Conditioned on the event \(F(\mathbf{l})\), the branching process
\(\procall\) is \new{almost surely} equivalent to the following construction: First start with a path
of length \(\new{t+1}\) which is identified as the spine.  Then, for every \(r \in \{0,\dots,\new{t}\}\), attach an independent copy of the branching process \(\procspine\)
(defined \new{at the end} of~\autoref{sec:cond:surv}) to the individual at generation  \(r\) of the spine, conditioned on its root having \(\ell_{\new{t+1}-r}-1\) children. \new{Finally, attach an independent copy of the branching process \((X_j)_{j\geq 0}\) to the individual at generation $t+1$ of the spine, conditioned on survival}.

Conditioned on $E_1$, let us write
\begin{equation}\label{PUUA}
    \Gamma_t = \Gamma_{1,t}+ \sum_{i=2}^{X_t} \Gamma_{i,t}\eqqcolon\Gamma_t^* + \Gamma_t^0,
\end{equation}
\new{where $\Gamma_t^*$ represents the contribution of the spine to the total weight and $\Gamma_t^0$ can be seen as a spurious contribution, which we will disregard for the purposes of this lemma.}

As $E_2$ implies $\{\Gamma_t^0=0\}$, the probability \new{in the LHS of~\eqref{DMEX}} is at least
\begin{equation}\label{WKDM}
\begin{aligned}
&\p{
\{0<\Gamma^*_t< e^{-(a \hat{H}-\delta) t}\}\cap E_1\cap E_2\cap E_3}
\\
&\hspace{2cm}\geq \p{E_1}\p{\Gamma^*_t< e^{-(a \hat{H}-\delta) t}\mid E_1} \min_{\mathbf{l}\in \cL} \p{E_2\cap E_3\mid
F(\mathbf{l})} ,
\end{aligned}
\end{equation}
where we used that $E_2\cap E_3$ and $\{\Gamma^*_t< e^{-(a \hat{H}-\delta) t}\}$ are \new{conditionally} independent \new{given $F(\mathbf{l})$}.

Let us first bound the probability of $E_1$. By~\eqref{SOWK}, one has 
\begin{equation}\label{FASF}
\begin{aligned}
\p{E_1}
&= \new{\p{\cap_{r=0}^{t+1}\{X_{r}^{*}= 1\} }}\\
&=\p{X_{0}^{*}= 1}
\prod_{r = 1}^{\new{t+1}}
\p{X_{r}^{*}=1 \mid X_{r-1}^{*}=1}	\\
&=
\p{X_{0}^*=1}\p{\xi^*=1}^{\new{t+1}}\\
&=
s \hnu^{\new{t+1}}
,
\end{aligned}
\end{equation}
\new{where we used that $X_{r}^{*}$ and $(X_{j}^{*})_{0\leq j\leq r-2}$ are conditionally independent given $X_{r-1}^{*}$.} Recall that $s$ is bounded away from $0$ \new{by a constant only depending on $M$}, by~\autoref{lem:bounded_away}.

We now bound \new{the probability $\Gamma^*_t$ is small} conditioned on $E_1$. Let \(\zeta_{r}\new{\coloneqq \zeta_{1,r}}\) be the mark of the $r$-th generation individual in the spine. \new{Conditioned on $E_1$, $(\zeta_r)_{1\leq r\leq t}$ is a sequence of iid copies of $\tilde \zeta$, which we denote by  $(\tilde\zeta_r)_{1\leq r\leq t}$}. Letting $Z_r=\log \tilde{\zeta}_r$, we have
\begin{equation}\label{PJVW}
\Gamma^*_t =\prod_{r=1}^{t} (\tilde\zeta_{r})^{-1} = e^{-\sum_{r=1}^{t} Z_r}
.
\end{equation}
As $t\to \infty$, it follows from~\eqref{EPWQ} that
\begin{align}\label{PDIE}
    \p{\Gamma^*_t< e^{-(a\hat{H}-\delta) t}\mid E_1}\geq \frac{1}{2} e^{-(I(a\hat{H})+\delta) t}
    .
\end{align}
We finally obtain a bound on the probability of $E_2\cap E_3$ conditioned on $F(\mathbf{l})$, uniformly over $\mathbf{l}\in \cL$.  \new{Recall the construction of the branching process from the spine.} Since all the branching processes growing from the spine are mutually independent, we have
\begin{align}\label{FKFO}
    \p{E_2 \mid F(\mathbf{l})}
&= \prod_{r=0}^{\new{t}} \p{\hat X_{\new{t+1}-r}=0 \mid \hat X_{1}=\ell_{\new{t+1}-r}-1}
= \prod_{r= 2}^{\new{t+1}} \p{\hat X_{r}=0 \mid \hat X_1=\ell_r-1} 
,
\end{align}
where the last step uses that \(\ell_{1}=1\).
For \(r \ge 2\), using that \(\ell_{r} \le M\), we have
\begin{equation}\label{FEOJ}
\p{\hat X_{r}=0 \mid \hat X_1= \ell_r-1} \geq \p{\hat X_{2}=0 \mid \hat X_1= M} =  \p{\xihat = 0}^M.
\end{equation}

Also, by Markov inequality, there exists a constant \(r_{0}\) \new{only depending on $M$ (see \autoref{lem:bounded_away})} such that for all $r\geq r_{0}$
\begin{equation}\label{PUDH}
    \p{\hat X_{r} \ge 1 \mid \hat X_1= \ell_r-1}
    \leq \E{\hat X_{r} \mid \hat X_1= M}
    = M \hnu^{r-1} 
    \le 1/2
    .
\end{equation}
It follows that
\begin{equation}\label{FKWO}
    \p{E_2\mid F(\mathbf{l})}
    \geq 
    \p{\xihat = 0}^{r_{0}M} \prod_{r > r_{0}} (1-M\hnu^{r-1})
    >
    c_{0}
    ,
\end{equation}
for some constant $c_{0}>0$ that only depends on $M$ (see~\autoref{lem:bounded_away}).  

To bound the probability of $E_3$, we use the same argument as in~\cite[Theorem 3.4]{cai2020a}. Note that $X_r>0$ is already implied by $E_1$, it suffices to bound the probability $X_r$ is not too large. By linearity of the expectation,
\begin{align*}
    \E{X_r\mid F(\mathbf{l})} = 1+ \sum_{j=1}^r (\ell_{\new{t+1}-r+j}-1) \hnu^{j-1}\leq c_1,
\end{align*}
for some $c_1$ only depending on $M$ (see~\autoref{lem:bounded_away}).
By independence of the branching processes growing from the spine and the moment formula in~\cite[pp. 4]{athreya1972} for $\hnu\in (0,1)$,
\begin{align*}
\V{X_r\mid F(\mathbf{l})} 
\le
\sum _{j=1}^{r} (\ell_{\new{t+1}-r+j}-1)\frac{\Vv{\hat{\xi}} \hat{\nu }^{j-2} \left(\hat{\nu }^{j-1}-1\right) }{\hat{\nu }-1}
\leq c_2,
\end{align*}
for some $c_2$ only depending on $M$ (see~\autoref{lem:bounded_away}).
Thus, we have \(\ee{X_{r}^{2}\mid F(\mathbf{l})} \leq c_2+c_1^2\) and
it follows from Chebyshev's inequality that
\begin{equation}\label{eq:X:r:up}
    \p{E_{3}^c \mid F(\mathbf{l})}
    \le
    \sum_{r=1}^{t}
    \p{X_{r} \ge \omega \mid F(\mathbf{l})}
    \leq
        \sum_{r=1}^{t} 
        \frac{\E{X_{r}^{2} \mid F(\mathbf{l})}}{\omega^{2}}
    \leq \frac{(c_2+c_1^2)t}{\omega^2}
    \leq \frac{c_0}{2},
\end{equation}
since $\omega\geq t$ and $t$ is large enough with respect to $M$.

From~\eqref{FKWO} and~\eqref{eq:X:r:up}, we obtain
\begin{equation}\label{SEND}
\p{E_2\cap E_3\mid F(\mathbf{l})} \geq \p{E_2\mid F(\mathbf{l})}- \p{E_3^c\mid F(\mathbf{l})} \geq  c_{0}/2.
\end{equation}
The desired bound follows from plugging~\eqref{FASF},~\eqref{PDIE} and~\eqref{SEND} into~\eqref{WKDM}, and noting that $s$ and $c_0$ are bounded away from $0$ by~\autoref{lem:bounded_away}. 

If the minimal positive support of \(\xi\)  is \(k_0\geq 2\), then the only change needed
is to let \(E_{2} = \{X_{\new{t+1}} = k_0\}\). The extra \(k_0 - 1\) individuals in generation \(t\)
contribute at most \(k_0 M \Gamma_{1,t} \) to \(\Gamma_{t}\). Thus the same argument still
works.
\end{proof}

\subsection{Subcritical growth: an upper bound}\label{sec:UB}

\new{In this section we show that the probability that the marked branching process satisfies some undesirable conditions is exponentially small, which will be later used to give a lower bound on the stationary distribution value for all vertices in the random graph.}

Given a fixed $\gamma > 0$, define
\begin{equation}\label{SPQW}
\cB_t(\gamma)\coloneqq \cap_{i\in [X_{t}]} \{\Gamma_{i,t} \geq \gamma\}.
\end{equation}
In this section we will prove the following theorem:
\begin{thm}
    \label{thm:UB}
    Let \(\procall\) be a marked branching process with distribution \(\eta=(\xi,\zeta)\)
     satisfying~\autoref{cond:BP} \new{with $M\in \mathbb{N}$}.
    Then there exists $C=C(M)$ such that for any $a\geq 1$, $\delta>0$, $t\in \mathbb{N}$ sufficiently large with respect to $\delta$ and $M$, and $\omega\in (t^2,e^{\sqrt{t}})$,
    \begin{equation}\label{QKPL}
        \p{(\cB_t(e^{-(a \hat{H}+\delta) t}))^c\cap \{0<X_t<\omega\}}\leq  \omega^C
        \exp\left\{-\left(|\log{\hnu}| + I(a \hat{H}) -\delta\right) t\right\}
        .
    \end{equation}
\end{thm}

\new{This theorem is a counterpart of  \autoref{thm:LB}. As in there, the event inside the probability in the LHS of~\eqref{QKPL} can be split into two parts. The main difference is that here we think about them as undesirable events we would like to avoid. The first one, is the complement of $\cB_t(\gamma)$, which is the event of \emph{all} leaf-to-root paths having a heavy weight. The second one is exactly as in \autoref{thm:LB}. Analogously, these two events can be related to the RHS of~\eqref{QKPL} and to the properties (2) and (1) respectively in the discussion below \autoref{thm:main}.}

\subsubsection{An inhomogeneous branching process}

Fix $t\in \mathbb{N}$ and let $\proctrunc{t} \subseteq \procall$ be the finite subprocess containing individuals
in the first \(t\) generations that have some progeny in generation $t$. Note that $X_t^{(t)}=X_t$. Similar to $(X_r^*)_{r\geq 0}$,
\(\proctrunc{t}\) is non-decreasing in \(r\). Conditioned on the event $\{X_t>0\}$, \(\proctrunc{t}\) can
be seen as an inhomogeneous branching process where the offspring distribution of the individuals in
generation $r=t-a$ is $\xi^{(a)}$, defined by
\begin{equation}\label{KWDA}
    \p{\xi^{(a)}=k}
= \frac{1}{s_a} \sum_{k'\geq k}\p{\xi=k'} \binom{k'}{k} s_{a-1}^k (1-s_{a-1})^{k'-k}
= \frac{(s_{a-1})^{k}}{s_a k!} G_{\xi}^{(k)}(1-s_{a-1}), 
\qquad \text{for } k \ge 1\,,
\end{equation}
where $s_a\coloneqq\p{X_a>0}$. \new{Indeed, provided that the root has $k'$ children, the number of them that have progeny at generation $a$ is a binomial random variable with $k'$ independent trials and success probability $s_{a-1}$.}

Note the similarity between $\xi^{(a)}$ and $\xi^*$ which is defined in~\eqref{JDOE}. We have $s_a=s + O(\hnu^a)$ (see~\cite[Eq. (3.6)]{cai2020a}), where the asymptotics notation is as $a\to \infty$. Using the Taylor expansion of $G_{\xi}^{(k)}$
around $1-s$, we get
\begin{equation}\label{ODNE}
    \p{\xi^{(a)}=k
    } = \frac{s^{k-1}}{k!} G_{\xi}^{(k)}(1-s)+ O(\hnu^a)
    = \p{\xistar=k}+O(\hnu^a),
    \qquad \text{for } a\geq 0,  k \geq  1.
\end{equation}
In particular, by~\eqref{SOWK},
\begin{equation}\label{KSNW}
    \p{\xi^{(a)}=1}
=
\p{\xistar=1}+O(\hnu^a)
=
\hnu+O(\hnu^a).
\end{equation}
Since  $\xi \leq M$ by \autoref{cond:BP}, it follows from~\eqref{ODNE} that
\begin{equation}\label{SPEJ}
    \E{\xi^{(a)}}
    = \E{\xistar} +O(\hnu^a)
    = \nu+O(\hnu^a)
    .
\end{equation}

\subsubsection{Control the surviving process}

Denote by $\pt{\cdot}$ the probability conditioned on survival at time $t$, i.e. 
\begin{align}\label{eq:Pt}
\pt{\cdot}\coloneqq \p{\cdot\mid X_t>0}.
\end{align}
The argument for the following lemma is similar to that of Theorem~3.4 in our previous work~\cite{cai2020a}.  We give a proof for completeness.

\begin{lemma}\label{lem:DSKD}
    Let \(t\) and \(\omega\) be as in~\autoref{thm:UB}.
    Set $t_{0}\coloneqq \left(\frac{1}{|\log \hnu|}+\frac{1}{\log \nu}\right)\log \omega$. 
Then there exists a constant \(C_{0}\) depending only on $M$ such that for all $0\leq r\leq t-t_0$
\begin{align}\label{OEIW}
    \mathbb{P}_{t_0+r}(X_r^{(t_0+r)}<\omega) \leq C_{0} \hnu^{r-t_{0}} 
    .
\end{align}
In particular,
\begin{align}\label{LSNE}
    \p{0<X_t<\omega} \leq C_{0} \hnu^{t-2t_{0}} 
    .
\end{align}
\end{lemma}

\new{The bound on the RHS of~\eqref{LSNE} can be understood, up to error terms, as follows: at each generation, the price to pay for keeping the branching process alive but small is $\hnu=\p{\xi^*=1}$. This indicates that, the most likely way to build a narrow branching process is essentially to only have one single path that survives up to time $t$, plus some other parts that get extinguished before generation $t$.}

\begin{proof}
Recall that, for any $0\leq r\leq t-t_0$, conditioned on survival at time $r+t_0$, $(X_j^{(t_0+r)})_{t_0+r\ge j \ge 0}$ is an inhomogeneous branching process
where the $j$-th generation has offspring distribution $\xi^{(t_0+r-j)}$, defined as in~\eqref{KWDA}. Recall that $\E{\xi^*}=\nu>1$. By~\eqref{ODNE} and the choice of $t_{0}$, for any $0\leq j\leq r$, we have 
\begin{equation}\label{TKIN}
    \p{\xi^{(t_0+r-j)}=k
    }
    = 
    \p{\xistar=k}+O(\hnu^{t_{0}})
    = 
    \p{\xistar=k}+O(\omega^{-1})
    ,
    \qquad \text{for } k \geq  1,
\end{equation}
and by~\eqref{SPEJ}
\begin{equation}\label{TZBJ}
    \E{\xi^{(t_0+r-j)}}
    =
    \nu(1+O(\omega^{-1}))
    .
\end{equation} 
Choose $\varepsilon>0$ sufficiently small with respect to $M$, so $((1-\varepsilon)\nu)^{t_0}\geq 2\omega$; this is possible as $\hnu$ is bounded away from $1$ by~\autoref{lem:bounded_away}.
Let $\xi^\da$ be a \new{probability} distribution such that  $\xi^\da$ is stochastically dominated by each $\xi^{(t_0+r\new{-j})}$ for $0\leq \new{j \leq  r}$, and $\nu^\da\coloneqq \E{\xi^\da}\geq \nu(1-\varepsilon)>1$. Indeed such offspring distribution exists by~\eqref{TZBJ}.
 Let $(X^\da_j)_{j \ge 0}$ be a branching process with offspring distribution $\xi^\da$. The processes
$(X_j^{(t_0+r)})_{r \ge j \ge 0}$ and $(X^\da_j)_{r\ge j \ge 0}$ can be coupled so $X^{(t_0+r)}_j\geq X^\da_j$
almost surely for every $0\leq j\leq r$.

Let $a_r\coloneqq \mathbb{P}_{t_0+r}(X_{r}^{(t_0+r)}<\omega)$ \new{be the probability we would like to bound from above}. For $r=t_0$, there exists $c_0>0$ \new{only depending on $M$} such that
\begin{equation}\label{RWPB}
    \begin{aligned}
        1-a_{t_{0}}
        & 
        =\mathbb{P}_{2t_{0}}(X_{t_{0}}^{(2t_0)}\geq \omega)
        \geq \p{X_{t_{0}}^\da\geq \omega} 
        \geq  \p{X_{t_{0}}^\da \geq \frac{1}{2}(\nu^\da)^{t_{0}}} 
      \geq c_0
        ,
    \end{aligned}
\end{equation}
\new{where we used our choice of $\varepsilon>0$.}
For the last inequality we used the following fact that can be easily proved\footnote{
For a proof, let $C = \frac{\sigma^2}{\nu(\nu-1)}+1$. By Cauchy-Schwartz inequality, we deduce
$$
(1-c)^2 \E{Z_t}^2 \leq \E{Z_t\mathbb{I}(Z_t\geq c\,\E{Z_t})}^2 \leq \E{Z_t^2}\p{Z_t\geq c\,\E{Z_t}},
$$
and since $\E{Z_t^2}\leq C\E{Z_t}^2$, the inequality holds.
}: there exists $C>0$ such that for all $c\in (0,1)$ and supercritical branching process $(Z_t)_{t\geq 0}$ with offspring having mean $\nu\new{>1}$ and finite variance $\sigma^2$, we have 
\begin{align}\label{NRIW}
\p{Z_t\geq c\,\E{Z_t}}\geq \frac{(1-c)^2}{C}.
\end{align}

\new{Let us now bound $a_r$ for $r>t_0$.} Using~\eqref{TKIN}, we have the
simple recursive inequality:
\begin{equation}\label{KXXJ}
    \begin{aligned}
        a_r 
        & 
        \leq \p{\xi^{(t_0+r)}=1}a_{r-1}+ \left(1-\p{\xi^{(t_0+r)}=1}\right)a_{r-1}^2 
        \\
        & 
        \new{= \hnu a_{r-1}+ (1-\hnu)a_{r-1}^2+ O(\omega^{-1} a_{r-1})}
        .
    \end{aligned}
\end{equation}
\new{The first inequality is justified as follows: If the root has one child, then the desired event has probability exactly $a_{r-1}$. Otherwise, the root has at least two children and the corresponding subprocesses rooted at them must each have size less than $\omega$ at generation $r-1$. (In fact, their sum must be less than $\omega$.)}

This recursion has exactly the same form as~\cite[Eq. (2.4)]{riordan2010} and can be solved in the same
way to show that there exists a constant \(C_{0}\) \new{only depending on $M$} such that for all $0\leq r\leq t-t_0$
\begin{align}\label{LSNE1}
    a_r    
    \leq C_{0} \hnu^{r-t_{0}} 
    ,
\end{align}
proving~\eqref{OEIW}.

The second statement easily follows from the first one. Since $X_r^{(t)}$ is increasing for $0\leq r\leq t$ and using~\eqref{LSNE1} with $r=t-t_0$
\begin{align}\label{LSNE12}
    \p{0<X_t<\omega}=
    \pt{X_t^{(t)}<\omega}\p{X_t>0}\leq \pt{X_{t-t_0}^{(t)}<\omega}=a_{t-t_0}\leq C_{0} \hnu^{t-2t_{0}} 
    .
\end{align}
\end{proof}

\new{As previously discussed, \autoref{lem:DSKD} hints that the most likely way to keep a branching process narrow is to have a single path surviving up to time $t$. In the next lemma we will see the cost of having a more complex structure that survives until  generation $t$.}

Let $x$ be an individual at generation $t$ of $\procall$. 
Let $y_{0},y_1,\dots, y_t=x$ be the path connecting the root $y_{0}$ to $x$, which we refer to as
the \emph{spine associated to $x$}. \new{By permuting the cousin index of individuals, we may assume that $y_r$ is the first individual of generation $r$, for all $0\leq r\leq t$}. An index $r\in\{0,\dots, t-1\}$ is a \emph{ramification}\footnote{The word ``ramification'' means ``a complex or unwelcome
consequence of an action or event.''} of the spine, if $y_r$ has offspring at least $2$ in
$\proctrunc{t}$. Let $R(x)$ be the number of ramifications of the spine associated to $x$.
\new{One can decompose the set of individuals in each generation of $\proctrunc{t}$
according to their first common ancestor with $x$, i.e., the first of their ancestors that belongs to the
spine \(y_{0},\dots,y_{t}=x\), we call this the \emph{spine decomposition} of $\proctrunc{t}$.} See~\autoref{fig:spine}.

\begin{figure}[ht]
\centering
			\begin{tikzpicture}[scale=0.6]

			\node[circle,draw,fill=black] (V0)	at (18,0) {};	
			\node[circle,draw,fill=black] (V1)	at (15,0) {};	
			\node[circle,draw,fill=black] (V2)	at (12,0) {};	
			\node[circle,draw,fill=black] (V3)	at (9,0) {};	
			\node[circle,draw,fill=black] (V4)	at (6,0) {};	
			\node[circle,draw,fill=black] (V5)	at (3,0) {};	
			\node[circle,draw,fill=black] (V6)	at (0,0) {};	
			
			\draw[] (V1) --  (V0);
			\draw[] (V2) --  (V1);
			\draw[] (V3) --  (V2);
			\draw[] (V4) --  (V3);
			\draw[] (V5) --  (V4);
			\draw[] (V6) --  (V5);
			
			\node[circle,fill=lightgray] (V11) at (15,4) {};	
			\draw[lightgray] plot [smooth, tension=1] coordinates { (15.2,4) (16.5,3)  (17.9,0.2)};;
			
			\node[circle,fill=lightgray] (V21) at (12,4) {};	
			\node[circle,draw,very thick,fill=blue] (V22) at (12,3) {};	
			\node[circle,fill=lightgray] (V23) at (12,-4) {};	
			
			\draw[lightgray] (V21) --  (V11);
			\draw[] plot [smooth, tension=1] coordinates { (12.2,3) (13.5,2.3)  (14.9,0.2)};;
			\draw[lightgray] plot [smooth, tension=1] coordinates { (12.2,-4) (13.5,-3)  (14.9,-0.2)};;

			\node[circle,draw,very thick,fill=blue] (V31)	at (9,3) {};	
			\node[circle,draw,very thick,fill=green] (V32)	at (9,-3) {};	
			\node[circle,fill=lightgray] (V33)	at (9,-4) {};	
			\node[circle,fill=lightgray] (V34)	at (9,-5) {};	
			
			\draw[] (V31) --  (V22);
			\draw[lightgray] (V33) --  (V23);
			\draw[] plot [smooth, tension=1] coordinates { (9.2,-3) (10.5,-2.3)  (11.9,-0.2)};
			\draw[lightgray] plot [smooth, tension=1] coordinates { (9.2,-5) (11,-4.75)  (11.9,-4.2)};
			
			\node[circle,draw, very thick,fill=blue] (V41)	at (6,3) {};
			\node[circle,fill=lightgray] (V42)	at (6,1) {};		
			\node[circle,fill=lightgray] (V43)	at (6,-2) {};		
			\node[circle,draw,very thick,fill=green] (V44)	at (6,-3) {};
			\node[circle,fill=lightgray] (V45)	at (6,-5) {};
			
			\draw[] (V41) --  (V31);
			\draw[] (V44) --  (V32);
			\draw[lightgray] (V45) --  (V34);
			\draw[lightgray] plot [smooth, tension=1] coordinates { (6.2,1) (8,0.75)  (8.9,0.2)};
			\draw[lightgray] plot [smooth, tension=1] coordinates { (6.2,-2) (8,-2.25)  (8.9,-2.8)};

			\node[circle,fill=lightgray] (V51)	at (3,4) {};		
			\node[circle,draw,very thick,fill=blue] (V52)	at (3,3) {};	
			\node[circle,draw,very thick,fill=blue] (V53)	at (3,2) {};	
			\node[circle,fill=lightgray] (V54)	at (3,-1) {};		
			\node[circle,draw,very thick,fill=green] (V55)	at (3,-3) {};	
			\node[circle,fill=lightgray] (V56)	at (3,-5) {};	
			
			\draw[] (V52) --  (V41);
			\draw[] (V55) --  (V44);
			\draw[lightgray] (V56) --  (V45);
			\draw[lightgray] plot [smooth, tension=1] coordinates { (3.2,-1) (5,-0.75)  (5.9,-0.2)};
			\draw[] plot [smooth, tension=1] coordinates { (3.2,2) (5,2.25)  (5.9,2.8)};
			
			\draw[lightgray] plot [smooth, tension=1] coordinates { (3.2,4) (5,3.75)  (5.9,3.2)};

			\node[circle,draw,very thick,fill=blue] (V61)	at (0,3) {};	
			\node[circle,draw,very thick,fill=blue] (V62)	at (0,2) {};	
			\node[circle,draw,very thick,fill=red] (V63)	at (0,1) {};	
			\node[circle,draw,very thick,fill=red] (V631)	at (0,-1) {};	
			\node[circle,draw,very thick,fill=green] (V64)	at (0,-2) {};	
			\node[circle,draw,very thick,fill=green] (V65)	at (0,-3) {};	
			\node[circle,draw,very thick,fill=green] (V66)	at (0,-4) {};				
			\draw[] (V61) --  (V52);
			\draw[] (V62) --  (V53);
			\draw[] (V61) --  (V52);
			\draw[] (V65) --  (V55);
			\draw[] plot [smooth, tension=1] coordinates { (0.2,-4) (2,-3.75)  (2.9,-3.2)};
			\draw[] plot [smooth, tension=1] coordinates { (0.2,-2) (2,-2.25)  (2.9,-2.8)};
			\draw[] plot [smooth, tension=1] coordinates { (0.2,1) (2,.75)  (2.9,0.2)};
			\draw[] plot [smooth, tension=1] coordinates { (0.2,-1) (2,-.75)  (2.9,-0.2)};
			
			\node at (-0.7,0) {$x$};

			\node at (19,0) {$y_0$};
		
			\end{tikzpicture}
\caption{Instance of a branching process \new{$(X_r)_{t\geq r\geq 0}$ with $t=6$}. Black individuals form the spine of $x$. Grey individuals correspond to the process $(X_r-X_r^{(t)})_{t\geq r\geq 0}$ (that is, individuals with no \new{progeny} in the $6$-th generation. There are three ramifications, i.e. $R(x)=3$, namely at indices $1$, $2$ and $5$. The colors indicate the spine decomposition of  $\proctrunc{t}$.}\label{fig:spine}
\end{figure}

Next result refines~\eqref{LSNE} by taking into consideration the number of ramifications, as it is unlikely to have many of them.
\begin{lemma}\label{lem:LWIM}
    Let \(\omega,t, t_0\) be as in~\autoref{lem:DSKD}. 
    Then there exists a constant \(C_{1} > 0\) only depending on $M$ such that, for every individual $x$ at generation $t$ of $\procall$, we have
    \begin{equation}\label{JFMK}
        \p{0<X_t<\omega, R(x)\geq \ell} \leq 
        \hnu^{t+(t_{0}-C_{1})(\ell-\new{3t_0})}
        , 
        \quad
        \text{for } \new{3t_0} \leq \ell \le t
        .
    \end{equation}
\end{lemma}

\new{The bound on the RHS of~\eqref{JFMK} can be understood, up to error terms, as follows: the contribution $\hnu^{t}$ is the price to pay for the event $\{0<X_t<\omega\}$, as already computed in~\autoref{lem:DSKD}, while each ramification of the spine contributes with a multiplicative term $\hnu^{t_0}$ to the total probability.} Thus, the most likely situation is that $\proctrunc{t}$ has few ramifications, i.e. it is essentially composed by the spine. We will use this later to control the distribution of the marks in the leaf-to-root paths of the branching tree.

\begin{proof}
Conditioned on \(\{X_{t} > 0\}\), the number of children of
\(y_{0},\dots,y_{t-1}\) in \(\proctrunc{t}\) are distributed as independent random variables
\(\xi_{0},\dots,\xi_{t-1}\), where \(\xi_{r}\eql \xi^{(t-r)}\).

Therefore, we can generate the inhomogeneous branching process $(X^{(t)}_r)_{t \ge r \ge 0}$ \new{conditioned on survival up to generation $t$} as follows: (i) construct the
spine $y_0,\dots, y_t=x$; (ii) for every $\new{0\leq r\leq t-1}$ attach to $y_r$ a total of $\xi_{r}-1$ independent copies of
$(X_j^{(t-(r+1))})_{t-(r+1) \ge j \ge 0}$ \new{conditioned on $X_{t-(r+1)}^{(t-(r+1))}>0$}, which we denote by
\((W_{j}^{r,2})_{t-(r+1) \ge j \ge 0}, \dots, (W_{j}^{r,\xi_{r}})_{t-(r+1) \ge j \ge 0}\). 
Focussing only on generation \(t\) and \new{denoting $j=t-r$, this decomposition
gives the following recursion:
\begin{equation}\label{HSVH}
X_t
=
1+\sum_{j=1}^{t} \sum_{k= 2}^{\xi_{t-j}} W_{j-1}^{t-j,k},
\end{equation}
where we recall that $\xi_{t-j}\eql \xi^{(j)}$, and $W_{j-1}^{t-j,k}$ are independent copies of $(X_{j-1}\mid X_{j-1}>0)$, since $X_{j-1}^{(j-1)}=X_{j-1}$.}

\new{For $j\in [t]$}, consider the random variable $Z_j$ defined by
\begin{align}\label{SOWR}
Z_j& \eql
\begin{cases}
    0 & \text{with probability } b_{j} ,\\
    \left(X_{j-1} \mid X_{j-1} > 0\right)  & \text{with probability }1-b_{j} ,
\end{cases}
\end{align}
where \(b_{j} \coloneqq \p{\xi^{(j)}=1}\). We claim that, for $j\in [t]$, $Z_j$ is stochastically dominated by $\sum_{k= 2}^{\xi_{t-j}} W_{j-1}^{t-j,k}$. Indeed, they can be coupled so they are both zero when $\xi_{t-j}=1$ and \new{$Z_j \eql W_{j-1}^{t-j,2}$ otherwise. (Note that with this coupling, $Z_j$ disregards the contribution of the terms $k\geq 3$.)} 

As $Z_0=1$, it follows from~\eqref{HSVH} and~\eqref{SOWR} that
\begin{equation}\label{HSVH1}
X_t
\succeq
\sum_{j=0}^{t} Z_j
,
\end{equation}
where \(\succeq\) denotes \emph{stochastic domination} and
$Z_0,\dots,Z_{t}$ are independent random variables as described in~\eqref{SOWR}.

The process attached to the spine at generation $t-j$ will be difficult to control for small values of $j$, as the approximation of $\xi^{(j)}$ by $\xi^*$ fails for such values; see~\eqref{ODNE}. Thus, we avoid analysing the cases $j< \new{3t_0}$, assume that all of them could be ramifications and disregard their contribution in~\eqref{HSVH1} obtaining
\begin{equation}\label{HSVH2}
X_t
\succeq
\sum_{j=\new{3t_0}}^{t} Z_j
.
\end{equation}
Let $R_{0}(x)$ be the number of ramifications of the spine associated to $x$ \new{in the first $t-\new{3t_0}$ generations}. For every $\ell\geq 0$, define $p_{\ell}\coloneqq \p{0<X_t<\omega, R_{0}(x)=\ell}$ and   $B\coloneqq\prod_{\new{3t_0} \le j \le t } b_{j}$.

Suppose that there is a ramification with index $t-j$. Then there is a child $y$ in the offspring of $y_{t-j}$ different than $y_{t-j+1}$. Moreover, if $\{0<X_t<\omega\}$  holds, the branching process $(X_r')_{r\geq 0}$ rooted at $y$ \new{(distributed as $(X_r)_{r\geq 0}$)} satisfies $\{0<X_{j-1}'<\omega\}$. As $y_{t-j}$ has at most $M$ children \new{by \autoref{cond:BP}}, the probability of a ramification \new{at time $t-j$} is at most $c_{j}\coloneqq M \p{0<X_{j-1}<\omega}$; \new{thus, $c_j/b_j$ can be seen as a bound on the price to pay for a ramification. Finally note that the existence of ramifications are independent events for distinct indices due to the construction given above}.

It follows from a union bound that
\begin{equation}\label{SAPX}
    p_{\ell} \le
    B \sum_{j_{1} < \dots < j_{\ell}}
    \prod_{l=1}^{\ell}
    (c_{j_l}/b_{j_l})  
    ,
\end{equation}
where the sum is over all choices of \(\ell\) ordered and strictly increasing $j_1,\dots, j_\ell\in \{\new{3t_{0}},\dots,
t\}$, which indicates that the index $t-j_l$ is a ramification.
Since \(j \ge \new{3t_0}\) and by~\eqref{KSNW}, we have that \(b_{j} =\hnu + O(\hnu^{j})= \hnu + O(\hnu^{\new{3t_0}})\)  .
Thus,~\eqref{SAPX} implies that
\begin{equation}\label{ANCM}
    \begin{aligned}
        p_{\ell} 
        &
        \le
            (\hnu+O(\hnu^{\new{3t_0}}))^{t-\new{3t_0}-\ell}
        \sum_{j_{1} < \dots < j_{\ell}}
        \prod_{l=1}^{\ell}
        c_{j_{l}}
        \\
        &
       \leq
        2
        \hnu^{t-\new{3t_0}-\ell}
        \sum_{j_{1} < \dots < j_{\ell}}
        \prod_{l=1}^{\ell}
        c_{j_{l}}
        ,
    \end{aligned}
\end{equation}
where the last step uses that $(1+O(\hnu^{2 t_0}))^t\leq e^{O(\hnu^{\new{3t_0}}t)} \leq 2$ since  \(\hnu^{t_{0}}\leq \omega^{-1}\) by the choice of $t_0$, \(t^2\le \omega\) by the choice of $\omega$ and $t$, and $\omega$ is sufficiently large. 

One can bound $c_j$ using~\autoref{lem:DSKD} and, since $\hnu<1$, we obtain
\begin{equation}\label{VZYN}
    \sum_{j_{1}<\dots<j_{\ell}}
    \prod_{l=1}^{\ell}
    c_{j_{l}}
    \le 
    \left(
        \sum_{j=\new{3t_0}}^{t} c_{j}
    \right)
    ^{\ell}
    \le
    \left(
         C_{0} M  \sum_{j=\new{3t_0}}^{\infty} 
       \hnu^{j-\new{2t_{0}}-1}
    \right)
    ^{\ell}
    =
    \left( 
        C_0' \hnu^{t_{0}-1}
    \right)
    ^\ell
    ,
\end{equation}
for some constant $C_0'$ only depending on $M$ (see~\autoref{lem:bounded_away}).

Putting this back into~\eqref{ANCM} we have
\begin{equation}\label{YLVD}
    p_{\ell} 
   \leq 2 \hnu^{t-\new{3t_0}}\left(C_0'\hnu^{t_{0}-2}\right)^\ell
    ,
\end{equation}

Since there certainly are at most $\new{3t_0}$ ramifications in the last $\new{3t_0}$ indices of the spine (excluding \(x\)), it follows that for \(\ell \geq \new{3t_0}\),
\begin{equation}\label{QGUB}
    \begin{aligned}
    \p{0<X_t<\omega, R(x)\geq \ell} 
    &\leq  \p{0<X_t<\omega, R_{0}(x)\geq \ell-\new{3t_0}} 
	\\    
	&
    = 
    \sum_{l= \ell-\new{3t_0}}^{t} p_{l} 
	\\
	&    
    \leq 
    \frac{\new{2}C_0'}{1-\hnu^{t_{0}-2}}\cdot
    \hnu^{t-\new{3t_0}}
        \left(C_{1}
    \hnu^{t_{0}-\new{2}}
    \right)
    ^{\ell-\new{3t_0}}
    \\
    &
    \leq
    \hnu^{t+(t_{0}-C_{1})(\ell-\new{3t_0})}
    ,
    \end{aligned}
\end{equation}
for some constant \(C_{1}> 0\) depending only on $M$, provided that $t$ is large enough with respect to $M$.
\end{proof}

\subsubsection{Finishing the proof of~\autoref{thm:UB}}
\label{sec:finish:up}

Define
\begin{equation}\label{PUKK}
    \ell(t)\coloneqq
        \new{3t_0}+\frac{t\sqrt{t_0}}{t_{0}-C_1}
    \geq \new{3t_0}
    ,
\end{equation}
where $t_0$ is as in~\autoref{lem:DSKD} and $C_1$ is the constant appearing in~\autoref{lem:LWIM}. \new{The lemma implies that, up to time $t$, it is highly improbable that there are more than $\ell(t)$ ramifications if the process stays small,}
\begin{align}\label{GDAF}
\p{0<X_t<\omega, R(x)\geq\ell(t)}  \leq \hnu^{t\sqrt{t_0}}.
\end{align}


Let $E_1=\{0<X_t<\omega\}$. 
Let $x_1,\dots,x_{X_t}$ denote the individuals in the $t$-th generation of $\procall$. Let
$E_2$ be the event $ \cap_{j=1}^{X_t} \{R(x_j)< \ell(t)\}$. By a union bound over the choice
of $j\in [\omega]$ and using~\eqref{GDAF}
\begin{equation}\label{BIOB}
\p{E_1\cap E_2^c}
\leq 
\new{\sum_{j=1}^{\omega} \p{j\leq X_t < \omega, R(x_j)\geq\ell(t)}}
\leq \omega \hnu^{t\sqrt{t_0}} .
\end{equation}
By~\autoref{lem:DSKD}, we have
\begin{equation}\label{OVGA}
    \p{E_1\cap E_2}\leq \p{E_1} \leq  C_0 \hnu^{t-2t_0}
    .
\end{equation} 
Let $\gamma\coloneqq e^{-(a\hat{H}+\delta)t}$. It follows from~\eqref{BIOB} and~\eqref{OVGA} that the desired probability is
\begin{align}\label{SJEO}
        \p{(\cB_t(\gamma))^c\cap E_1}
        &
        \leq \p{(\cB_t(\gamma))^c\mid E_1\cap E_2}\p{E_1\cap E_2}+ \p{E_1\cap  E_2^c}\nonumber
        \\ 
        &
        \leq  
        C_0 \hnu^{t-2t_0}
        \sum_{j = 1}^{\omega}
        \p{
            \new{j \leq X_{t} < \omega,
            \Gamma_{j,t} <} \gamma
            \mid 
            E_1 \cap E_2
        }
        +
        \omega \hnu^{t\sqrt{t_0}}\nonumber
        \\
        &
        \leq  \omega^{C_2} \hnu^{t}
        \p{
            \Gamma_{1,t} < \gamma
            \mid 
            E_1 \cap E_2
        }
        +
        \omega \hnu^{t\sqrt{t_0}}
        , 
\end{align}
for some constant $C_2$, depending only on $M$. The two last lines respectively use a union bound on $j\in [\omega]$ and the symmetry of all individuals at a given generation.

Thus, it suffices to upper bound the probability of $\{\Gamma_{1,t} < \gamma\}$ conditioned on
\(E_1 \cap E_2\). Let $\cA_{t,\omega}$ be the set of rooted trees $T$ of height $t$ with less than $\omega$ leaves, and such that the spine associated to each leaf has at most $\ell(t)$ ramifications. The trees in
$\cA_{t,\omega}$ are the candidates for $\GW_t$, \new{the tree associated to the first $t$ generations of
the process $(X_r)_{r\geq 0}$}, when conditioned on $E_1 \cap E_2$.  We will obtain an upper bound for the
probability of $\{\Gamma_{1,t} < \gamma\}$ conditioned on $E_1 \cap E_2 \cap \{\GW_{t} \cong T\}$, uniformly for all $T\in
\cA_{t,\omega}$, and thus it will also be an upper bound for \( \p{\Gamma_{1,t} < \gamma \mid E_1 \cap E_2 } \).

\new{Fix $T\in \cA_{t,\omega}$ and condition on $E_1 \cap E_2 \cap \{\GW_{t} \cong T\}$. Let \(y_{0},\dots,y_{t}\) be the individuals of the spine associated to \(x_{1}\), by relabelling we may assume each has cousin index $1$ in its generation. We have
$$
\Gamma_{1,t}= \prod_{r=1}^{t} \frac{1}{\zeta_{1,r}}.
$$
In order to bound $\Gamma_{1,t}$ from above we define a random variable $W_t$ that satisfies $W_t\preceq \Gamma_{1,t}$ and then show that the probability $W_t$ is large is exponentially small. For this purpose, define 
\begin{equation}\label{SPMV}
W_t \coloneqq \prod_{r=1}^t w_r,
\end{equation}
where $w_r$ are now defined.}

If an index $r\in \{0,\dots,t-t_0\}$ is not a ramification of the spine associated to $x_1$ in $T$, we first sample
the number of children of \(y_{r}\) in \(\procall\) according to the distribution $\xitilde^{(t-r)}$, where  for every $k\geq 1$ and $a \ge t_0$,
\begin{equation}\label{OYXN}
\begin{aligned}
    \p{\xitilde^{(a)}=k}
    &\coloneqq
    \p{X_{1}=k\mid X_{1}^{(a)}=1}
    \\
    &=
    \frac{k s_{a-1} (1-s_{a-1})^{k-1}\p{\xi=k}}{s_{a} \p{\xi^{(a)}=1}}
    \\
    &=
    \p{\xitilde = k}
    +
    O(\hnu^{a}),
\end{aligned}
\end{equation}
where \(\xitilde\) was defined in~\eqref{SJWU}.
\new{That is, we are sampling the offspring of $y_r$ by conditioning on only one of its children having surviving progeny up to generation $t$, as there is no ramification in $r$.}
Conditioned on \(\xi_{1,r} \eql \xitilde^{(t-r)}\), we let $w_r=1/\zeta_{1,r}$. 

If $r\in\{0,\dots,t-t_0\}$ is a ramification of $T$, or if $r\in\{t-t_0\new{+1},\dots,t\}$, we
simply set $w_r=1/M\).

\new{By~\eqref{OYXN} and the value of $t_0$, we can couple \(\xitilde^{(t)}, \dots, \xitilde^{(0)}\)
with \(\xitilde_0, \dots,\xitilde_{t}\), iid copies of \(\xitilde\), such that 
\(\p{\xitilde^{(t-r)} \ne \xitilde_{r}} < \omega^{-1}\) for every $r\leq t-t_0$.}
Thus, the number \new{of $r\in\{0,\dots,t-t_0\}$} where the two
sequences differ is stochastically bounded from above by a binomial random variable with  \(t\) trials and probability \(\omega^{-1}\). 
Let \(E_3\) be the event that \((\xitilde^{(t)}, \dots, \xitilde^{(t_0)})\) and \( (\xitilde_0,
\dots,\xitilde_{t-t_0})\) differ at at most $m(t)\coloneqq t(\log t)^{-1/2}$ positions. It follows that
\begin{equation}\label{BPYE}
    \p{E_3^{c}}
    \le
		(t/\omega)^{m(t)}
    \le
        t^{-m(t)}
    =
        e^{-t\sqrt{\log{t}}}
    ,
\end{equation}
where we used that \(t \le \sqrt{\omega}\).  Thus, we obtain
\begin{equation}\label{HRST}
    \p{
        \Gamma_{1,t} < \gamma
        \mid 
        E_1 \cap E_2 \cap \{\GW_t\cong T\}
    }
    \le
    \p{
        \Gamma_{1,t} < \gamma
        \mid 
        E_1 \cap E_2 \cap E_3 \cap \{\GW_t\cong T\}
    }
    +
    e^{-t\sqrt{\log{t}}}
    .
\end{equation}

Let \((\tilde{\zeta}_{r})_{r \ge 0}\) be iid copies of \(\tilde{\zeta}\) as defined in~\autoref{sec:cond:surv}.
Conditioning on \(E_1 \cap E_2 \cap E_3 \cap \{\GW_t\cong T\}\) and since $\tilde\zeta_r\leq M$, we have
\begin{equation}\label{SSUI}
    \Gamma_{1,t} \new{\succeq W_t} \geq
    M^{-(\ell(t)+t_{0}+m(t))}
    \prod_{r=1}^{t} (\tilde \zeta_r)^{-1}
    .
\end{equation}
\new{Here the exponent of $M$ accounts for the at most $\ell(t)$ ramification points, the last $t_0$ generations where the error of approximating $\xitilde^{(a)}$ by $\xitilde$ is too large, and the at most $m(t)$ positions where the coupling between $\xitilde^{(a)}$ and $\xitilde$ fails. }

Let $Z_r = \log \tilde{\zeta}_r$. 
It follows from~\eqref{WOEW} that
\begin{equation}\label{NFCE}
    \begin{aligned}
        &
        \p{
            \Gamma_{1,t} < \gamma
            \mid 
            E_1 \cap E_2 \cap E_3 \cap \{\GW_t\cong T\}
        }
        \\
        &
        \le
        \new{\p{
            W_{t} < \gamma
            \mid 
            E_1 \cap E_2 \cap E_3 \cap \{\GW_t\cong T\}
        }}
        \\
        &
        \le
        \p{
            \frac{1}{t}
                \sum_{r=1}^{t} Z_{r}
            >
            (a\hat{H}+\delta)  - \frac{(\ell(t)+t_{0}+m(t)) \log M}{t}
        }
        \\
        &
        \le
        \p{
            \frac{1}{t}
                \sum_{r=1}^{t} Z_{r}
            >
            a\hat{H}
        }
        \\
        &
        \leq
        e^{- (I(a\hat{H}) -\delta)t}
        ,
    \end{aligned}
\end{equation}
where in the last step we used that \(\ell(t),m(t),t_{0}$ are asymptotically smaller than $t$, and that $t$ is large enough with respect to $\delta$ and $M$.
Putting this into~\eqref{HRST}, we have
\begin{equation}\label{JEBA}
    \p{
        \Gamma_{1,t} < \gamma
        \mid 
        E_1 \cap E_2 \cap \{\GW_t\cong T\}
    }
    \le
    2e^{-(I(a\hat{H}) -\delta)t}
    .
\end{equation}
\autoref{thm:UB} follows by putting the above into~\eqref{SJEO} and taking $C$ sufficiently large with respect to $M$.

\subsubsection{A corollary}

Similarly as in~\eqref{ODRS}, consider
\begin{equation}\label{PBJK}
    \phi(a)\coloneqq \frac{1}{a}\left(|\log{\hnu}| + I(a\hat{H})\right),
\end{equation}
and let $a_{0}$ be the value $a\in [0,\infty)$ that minimises $\phi(a)$, which satisfies $a_0\geq 1$.  

Define the random variable
\begin{equation}\label{SOEP}
t_{\omega}\coloneqq \inf\{t\geq 0:\,X_t \geq \omega\}.
\end{equation}

The following corollary \new{rephrases the result obtained in \autoref{thm:UB}: it parametrises it in terms of a target probability $p$ instead of a fixed time $t$, replacing the latter by the random time $t_\omega$. This phrasing will be convenient later.}
\begin{cor}\label{cor:DWPD}

For any $\varepsilon>0$, $M\in \mathbb{N}$ and $p\new{>0}$ sufficiently small with respect to $M$ and $\varepsilon$, the following holds. 
Let \(\omega \in
(|\log p|^3,e^{|\log p|^{1/3}})\) and 
$\gamma_p \coloneqq p^{(1+\varepsilon)\frac{\hat{H}}{\phi(a_0)} }$ \new{where $\hat{H}$ and $\phi$ are defined respectively as in~\eqref{EFND} and~\eqref{PBJK}}. 

\new{Let \((X_t)_{t \ge 0}\) be a marked branching process with distribution \(\eta=(\xi,\zeta)\)
     satisfying~\autoref{cond:BP}}.
Then we have
\begin{equation}\label{IJMZ}
\p{
    (\cB_{t_\omega}\left(\gamma_p\right) )^{c}
    \cap
    \left\{t_\omega<\infty \right\}
}
\leq 
\omega^C p^{1-\varepsilon}
,
\end{equation}
for some constant $C$ only depending on $M$.


\end{cor}
\begin{proof}
     For \(t \leq \log^2 \omega\), \new{since $\omega$ is large enough with respect to $M$ and so to all parameters of the branching process (see \autoref{lem:bounded_away})}, we deterministically have
    \begin{equation}\label{LMAC}
        M^{-t} > M^{-\log^2\omega} > \gamma_p
        ,
    \end{equation}
    \new{and, since $\zeta\leq M$}, it implies
    \begin{equation}\label{BBCW}
        \p{
            (\cB_{t}(\gamma_p))^{c}
        }
        =
        \p{
            \cup_{i\in [X_{t}]}
            \left\{
                \Gamma_{i,t} < \gamma_p
            \right\}
        }
        \le
        \p{
            \cup_{i\in [X_{t}]}
            \left\{ 
                M^{-t} < \gamma_p
            \right\}
        }
        =
        0
        .
    \end{equation}
    \new{Recall that $\hnu\in (0,1)$} and let $t_* = {\abs{\log p}}/\abs{\log\hnu}< \sqrt{\omega}$.
   Choose \(t\in (\log^2 \omega,t_{*}]\) \new{so $\omega\in (t^2, e^{\sqrt{t}})$}. Set 
   $
   a\coloneqq \frac{\abs{\log p}}{t \phi\left(a_{0}\right)}$ \new{so $\gamma_p=e^{-(1+\varepsilon)a \hat{H} t}$}. As $\phi(a_{0})\leq \phi(1)=\abs{\log\hnu}$, we have $a\geq 1$. 
    It follows from~\autoref{thm:UB} \new{with $\delta=\varepsilon a \min\{ \hat{H},\phi(a)\}$} that there exists a constant $C_1>0$ such that    
    \begin{equation}\label{VVLT}
        \begin{aligned}    
            \p{\left(\cB_{t}\big(e^{-(1+\varepsilon)a \hat{H} t}\big)\right)^c\cap \{0<X_{t}<\omega\}}
            &
            \le
            \omega^{C_1}  \exp {-(1-\varepsilon)a\phi(a)t}
            \\
            &
            \le
            \omega^{C_1}  p^ {(1-\varepsilon)\frac{\phi\left(a\right)}{\phi\left(a_{0}\right)}}
            \le
            \omega^{C_1} p^{1-\varepsilon}
            ,
        \end{aligned}
    \end{equation}
    where the last step uses the fact that $a_{0}$ minimises $\phi(a)$ for all $a\geq 0$.
    Thus, 
    \begin{align*}
        \p{(\cB_{t_\omega}(\gamma_p))^c \cap \{t_\omega<\infty\}}
        &
        =
        \sum_{t\geq 0} \p{(\cB_{t_\omega}(\gamma_p))^c \cap\{ t_\omega=t+1\}}
        \\
        &
        \leq  
        \sum_{t= \log^{2}\omega}^{t_*} 
        \p{(\cB_{t}(\gamma_p))^c \cap \{0<X_{t}<\omega\}}
        +
        \p{t_{\omega} > t_{*}}
        \\
        &
        \leq 
        t_* \omega^{C_1} p^{1-\varepsilon}
        +
        C_0 \hnu^{t_*-2t_0}
        \\
        &
        \leq 
			\omega^{C} p^{1-\varepsilon}
        ,
    \end{align*}
    for some constant $C$ only depending on $M$. In the second last inequality, we used~\autoref{lem:DSKD} to bound
    \(\p{t_{\omega}>t_{*}}\). In the last inequality, we used that $t_0$ \new{satisfies $t_0/t_*\to 0$ as $\omega\to \infty$, since $|\log p|\geq \log^3 \omega$ by our assumptions}. 
\end{proof}

\subsection{A truncated Martingale}
\new{As explained in \autoref{sec:strategy}, in order to guarantee concentration of certain random variables, it will be convenient that all paths are assigned a small weight. In this direction, we introduce a truncated version of $\Gamma_t$.}

Fix $t_{0}\in \mathbb{N}$ and
$\gamma>0$. Recall that \( (p^{t-r}(i,t),r) \) is the index of the ancestor of \( (i,t)\) in
generation \(r\).  For every $t \ge
t_{0}$ and $i\in [X_{t}]$, define
\begin{equation}\label{XYYI}
\hat\Gamma_{i,t} \new{(t_0,\gamma)}
\coloneqq \gamma \prod_{\new{r=t_0+1}}^{\new{t}} \frac{1}{\zeta_{p^{t-r}(i, t),r}}
\text{ and }
\hat\Gamma_t \new{(t_0,\gamma)}\coloneqq \sum_{i\in [X_{t}]}\hat\Gamma_{i,t}\new{(t_0,\gamma)}
.
\end{equation}
\new{Comparing to the definition of $\Gamma_{i,t}$ in~\eqref{OSKE}, in $\hat \Gamma_{i,t}$ the contribution of the first $t_0$ generations is replaced by $\gamma$. Each time we consider $\hat \Gamma_{i,t}$, it will be under the event $\cB_{t_{0}}(\gamma)= \cap_{i\in [X_{t_0}]}\{\Gamma_{i,t_0}\geq \gamma\}$, so $\hat\Gamma_{i,t}\new{(t_0,\gamma)} \le
\Gamma_{i,t}$ for $i \in [X_{t}]$ and $ \hat\Gamma_t \new{(t_0,\gamma)}\le \Gamma_{t}$.}


\begin{prop}\label{prop:martingale}
Let \(\procall\) be a marked branching process with distribution \(\eta=(\xi,\zeta)\) satisfying~\autoref{cond:BP} with $M\in \mathbb{N}$. \new{Let $\omega, \ell\in \N$ and choose $\delta\in (0,1)$ sufficiently small such that $(1-2M\delta)^\ell\geq 1/2$. If $\abs{\E{\xi/\zeta}-1} \leq M \delta$, then for any $t_{0}\in \mathbb{N}$, $t\in [t_{0},t_0+\ell]$ and $\gamma>0$ we have 
\begin{equation}\label{OQFC}
\p{\hat\Gamma_t\new{(t_0,\gamma)} < \omega \gamma/2 \mid \{t_{\omega}=t_{0}\} \cap \cB_{t_{0}}(\gamma)} \leq  2\ell e^{-\delta^2\omega/4}
,
\end{equation}
where $t_\omega$ and $\cB_{t_{0}}(\gamma)$ are defined respectively as in~\eqref{SOEP} and in~\eqref{SPQW}.
}
\end{prop}

\new{This result states that, provided that there at least $\omega$ individuals in generation $t_0$ and all path connecting each of them to the root have weight at least $\gamma$, so $\Gamma_{t_0}\geq \omega \gamma$, at generation $t\geq t_0$ the total weight of the leaf-to-root paths is at least half of the weight at generation $t_0$. Equivalently, we can grow the branching process from generation $t_0$ without losing too much weight.}

\begin{proof}
\new{Throughout the proof, we condition on $\{t_{\omega}=t_{0}\}$ and $\cB_{t_{0}}(\gamma)$.  For $r\in [t_0,t]$, let $\hat\Gamma_r=\hat\Gamma_r(t_0,\gamma)$.} Consider the event
\begin{equation}\label{UOGH}
E_r=\{\hat \Gamma_r/\hat\Gamma_{t_{0}}\in (1/2,3/2) \}.
\end{equation}
We will lower bound its probability using Azuma's inequality; see~\cite[pp.~92]{molloy2002a}.

Recall that $m\coloneqq X_{\new{r}}$, $\hat\Gamma_{r-1}$ and $p(i,r)$ for $i\in [m]$ are all measurable with respect to $\cF_{r-1}$. 
\new{Let $\eta_{1,r}, \dots, \eta_{m,r}$, $\eta'_{1,r}, \dots, \eta'_{m,r}$ be iid copies of $\eta$.
Given $\cF_{r-1}$, $\hat\Gamma_{r}=g(\eta_{1,r}, \dots, \eta_{m,r})$, where \(g\) is a function depending on
\(\cF_{r-1}\). Note that for every choice of $\eta_{1,r}, \dots, \eta_{m,r}$ and $\eta_{i,r}'$, 
\begin{equation}\label{ZYVT}
|g(\eta_{1,r}, \dots,\eta_{i,r} ,\dots,\eta_{m,r})-g(\eta_{1,r}, \dots,\eta_{i,r}',\dots,\eta_{m,r})|=\big|\tfrac{1}{\zeta_{i,r}} -\tfrac{1}{\zeta'_{i,r}} \big| \cdot \hat\Gamma_{p(i,r),r-1}
\leq \hat\Gamma_{p(i,r),r-1}. 
\end{equation}
}

By~\autoref{rem:WBP}, $(\hat\Gamma_r \E{\xi/\zeta}^{-(r-t_{0})})_{r \ge t_{0}}$ is a martingale with
respect to \( (\cF_{r})_{r \ge t_{0}}\). Thus, by our hypothesis on $\eta$,
\begin{equation}\label{XQGJ}
    \abs{\hat{\Gamma}_{r-1}-\E{\hat\Gamma_{r} \mid \cF_{r-1}}}
    =
    \abs{\hat{\Gamma}_{r-1}- \E{\xi/\zeta}\hat{\Gamma}_{r-1}}
    \leq M \new{\delta} \hat{\Gamma}_{r-1} \eqqcolon s
    .
\end{equation}
It follows from~\eqref{XQGJ} and Azuma's inequality with the bounded difference condition ensured by~\eqref{ZYVT}, that
\begin{equation}\label{LSJF}
    \begin{aligned}
        &
        \p{\abs{\hat\Gamma_{r}-\hat{\Gamma}_{r-1}}> 2 s \mid \cF_{r-1}} 
        \\
        &
        \le
        \p{\abs{\hat\Gamma_{r}-\E{\hat\Gamma_{r} \mid \cF_{r-1}}}> s \mid \cF_{r-1}} 
        +
        \p{\abs{\hat{\Gamma}_{r-1}-\E{\hat\Gamma_{r} \mid \cF_{r-1}}}> s \mid \cF_{r-1}} 
        \\
        &
        \leq 
        2
        \exp{-\frac{s^2}{2\sum_{i=1}^m (\hat{\Gamma}_{p(i,r), r-1})^{2}}}.
    \end{aligned}
\end{equation}
Since $\hat \Gamma_{r-1}\leq \gamma$ by definition, 
\begin{equation}\label{VJDI}
\sum_{i=1}^m (\hat{\Gamma}_{p(i,r), r-1})^{2}\leq \sum_{j=1}^{X_{r-1}} M( \hat{\Gamma}_{j, r-1})^2\leq \gamma M \sum_{j=1}^{X_{r-1}} \hat{\Gamma}_{j,
r-1}  = \gamma M \hat\Gamma_{r-1},
\end{equation}
and thus
\begin{equation}\label{SNFK}
\ind{E_{r-1}}\p{|\hat\Gamma_{r}- \hat\Gamma_{r-1}|> 2 s \mid \cF_{r-1}} 
\leq 
2 e^{-\new{(\delta^{2}M/2\gamma)\hat\Gamma_{r-1}}}
\leq 
2 e^{-\new{\delta^2\omega/4}},
\end{equation}
where in the last inequality, we used \(\hat{\Gamma}_{r-1} \ge \omega \gamma/2\) on \(E_{r-1}\) .

Let \(F_{t_{0}}=  \{t_\omega=t_0\} \cap \cB_{t_{0}}(\gamma)\) and let
$F_r= \{\abs*{\hat\Gamma_{r}/\hat\Gamma_{r-1}-1} \le 2 M \new{\delta}\}$ for \(r \geq t_{0}+1\).
By~\eqref{SNFK},
\begin{equation}\label{FJEN}
    \p{\cup_{r=t_{0}+1}^t F_r^c \mid F_{t_{0}}} 
\leq 
\sum_{r=t_{0}+1}^{t} \p{F_{r}^c\mid \cap_{t_{0} \le r'< r} F_{r'}} 
\leq 
\new{2 \ell e^{-\delta^2\omega/4}},
\end{equation}
where we used $\cap_{t_{0} \le r'< r} F_{r'}\subseteq E_{r-1}$. So, with the desired probability
\begin{equation}\label{BQHU}
\hat\Gamma_t \geq \new{(1- 2M \delta)^{\ell}} \hat{\Gamma}_{t_0}\geq \frac{\omega\gamma}{2}.
\end{equation}
\end{proof}

\section{Exploring the random directed graph}\label{sec:graph}

\new{We now turn back our attention to the directed configuration model. In this section we describe an exploration process that reveals it while exploring the in-neighbourhood of an arbitrary head. Then, we show how to couple the process with a marked branching process, allowing us to transfer the results in \autoref{sec:branching} to the random graph setting, which will be done \autoref{sec:pi:min}.}

\subsection{The exploration process}
\label{sec:explore}


Recall that we refer to heads and tails respectively for the in-half and out-half edges used in the pairing of the directed configuration model. \new{We use superscripts $-$ and $+$ respectively to emphasise that a certain half-edge (or set of half-edges) is a head or a tail}. For a set of vertices \(\cV \subseteq [n]\), let \(\cE^{\pm}(\cV)\) be the set of heads/tails incident to
\(\cV\). Let \(\cE^{\pm}=\cE^{\pm}([n])\) be the set of all heads/tails. For a set of half-edges \(\cX\), let
\(\cV(\cX)\) be the vertices that are end-points of elements in \(\cX\). For $e^\pm\in
\cE^\pm$, we use \(v(e^\pm)\) to denote the end-point of \(e^\pm\).

\new{Our exploration process is a Breadth First Search (BFS) algorithm and, informally speaking, works as follows. It starts from an arbitrary head \(f\in \cE^-\), keeps each half-edge in exactly one of three states --- \emph{active,
paired}, or \emph{undiscovered} --- and sequentially pairs an active head with a random unpaired tail, updating the states afterwards.}

More precisely, let \(\cA_i^{\pm }\), \(\cP_i^{\pm}\)
and  \(\cU_i^{\pm }\) denote the set of tails/heads in the active, paired and undiscovered states
respectively, after the \(i\)-th pairing of half-edges. Initially, let
\begin{equation}\label{eq:APUF}
    \cA_{0}^{-}=\{f\},\,
    \cA_{0}^{+}=\cE^{+}(v(f)) ,\,
    \cP_{0}^{\pm}=\emptyset\text{ and }
    \cU_{0}^{\pm}=\cE^{\pm} \setminus (\cA^{\pm}_{0} \cup    \cP^{\pm}_{0})
    .
\end{equation}
Then set \(i=1\) and proceed as follows: \leavevmode
\begin{enumerate}[\normalfont(i)]
    \item Let \(e_{i}^{-}\) be one of the heads which became active earliest in
        \(\cA_{i-1}^{-}\).
    \item Pair \(e_{i}^{-}\) with a tail \(e_{i}^{+}\) chosen uniformly at random from
        \(\cE^+\setminus \cP^+_{i-1}\).  Let \(v_{i} = v(e_{i}^{+})\) and \(\cP_{i}^{\pm} =
        \cP_{i-1}^{\pm} \cup \{e^{\pm}_{i}\}\).
    \item 
Update the active sets as follows:
\begin{itemize}
\item[(a)] If \(e^{+}_{i} \in \cA_{i-1}^{+}\),
        then \(\cA_{i}^{\pm} = \cA_{i-1}^{\pm}\setminus \{e_i^\pm\}\).

\item[(b)] If \(e_i^{+}\in
        \cU^{+}_{i-1}\), then \(\cA_{i}^{\pm} = (\cA_{i-1}^{\pm} \cup \cE^{\pm}(v_{i})) \setminus
        \{e^{\pm}_{i}\}\).
\end{itemize}    
    
    \item If \(\cA_{i}^{-}\!\!=\emptyset\), terminate; otherwise, let \(\cU_i^{\pm}\! =\!\cE^{\pm}\setminus (\cA^{\pm}_{i} \cup
        \cP^{\pm}_{i})\), set \(i=i+1\) and go to (i).
\end{enumerate}
If at step $i$ we are in (iii.a), i.e. \(e^{+}_{i} \in \cA_{i-1}^{+}\), we say that a \emph{collision} has happened. \new{If there is no collision at any of the steps $1\leq j\leq i$, the pairing generated by the process up to step $i$, which corresponds to a partial in-neighbourhood of \(f\), induces a tree in the graph.}

With this motivation and in parallel to the exploration process, we construct a \new{nested sequence of rooted trees $(T_f^-(i))_{i\geq 0}$ as follows}.
Let \(T^{-}_{f}(0)\) be
a tree with one node, the root, which corresponds to \(f=e_1^-\).
For each $i\geq 1$, \(T^-_{f}(i)\) is constructed as follows: if \(e^{+}_{i} \in \cA_{i-1}^{+}\) then \(T^-_{f}(i)=T^-_{f}(i-1)\), and if \(e_{i}^{+} \in
\cU_{i-1}^{+}\) then obtain \(T^-_{f}(i)\) from  \(T^-_{f}(i-1)\) by adding
\(\abs{\cE^{-}(v_{i})}\) children to the node corresponding to \(e_{i}^{-}\), each
corresponding to a head in \(\cE^{-}(v_{i})\). See
\autoref{fig:explore} for an example.

\begin{figure}[ht]
\centering
			\begin{tikzpicture}[scale=0.8,square/.style={regular polygon,regular polygon sides=4}]

			\node[circle,draw,minimum width=0.8cm] (V)	at (0,12) {};	
			\node[circle,draw,minimum width=0.8cm]	(V1) at (0,9) {$v_1$};	
			\node[circle,draw,minimum width=0.8cm]	(V2) at (0,6) {$v_2$};	
			\node[circle,draw,minimum width=0.8cm]	(V3) at (3,6) {$v_3$};	
			\node[circle,draw,minimum width=0.8cm]	(V4) at (-3,3) {$v_4$};
			\draw[-Latex] (V1) --  (V) node[pos=0.75,right]{$e_1^-=f$} node[pos=0.25,right]{$e_1^+$};
			\draw (-0.1,10.5) -- (0.1,10.5);
			
			\draw[-Latex] (V2) -- (V1) node[pos=0.75,right]{$e_2^-$} node[pos=0.25,right]{$e_2^+$};
			\draw (-0.1,7.5) -- (0.1,7.5);
			
			\draw[-Latex] (V3) -- (V1) node[pos=0.75,right]{$e_3^-$} node[pos=0.25,right]{$e_3^+$};
			\draw (1.4,7.4) -- (1.6,7.6);
			
			\draw[-Latex] (V4) -- (V2) node[pos=0.75,right]{$e_5^-$} node[pos=0.25,right]{$e_5^+$};
			\draw (-1.4,4.4) -- (-1.6,4.6);
			
			\draw[-Latex] (V4) -- (V1) node[pos=0.75,left]{$e_4^-$} node[pos=0.25,left]{$e_4^+$};
			\draw (-1.6,6.05) -- (-1.4,5.95);

			\node (F01)	at (1.3,13.3) {};
			\draw[lightgray] (V) -- (F01);
			\node (F02)	at (-1.3,13.3) {};
			\draw[lightgray] (V) -- (F02);
			
			\node (F1)	at (1.3,10.3) {};
			\draw[lightgray] (V1) -- (F1);
			
			\node (F2)	at (-1.3,10.3) {};
			\draw[lightgray] (V1) -- (F2);
			
			\node (F3)	at (1.3,7.3) {};
			\draw[lightgray] (V2) -- (F3);
			
			\node (F31)	at (3.9,7.3) {};
			\draw[lightgray] (V3) -- (F31);
			
			\node (F32)	at (-4.5,4.7) {};
			\draw[lightgray] (V4) -- (F32);
			
			\node (F33)	at (-3.5,4.7) {};
			\draw[lightgray] (V4) -- (F33);

			\node (F4)	at (1.3,4.7) {};
			\draw[-Latex] (F4) -- (V2);
			
			\node (F5)	at (0,4.5) {};
			\draw[-Latex] (F5) -- (V2);
			
			\node (F6)	at (-1.3,10.7) {};
			\draw[-Latex] (F6) -- (V);
			
			\node (F7)	at (-3,1.5) {};
			\draw[-Latex] (F7) -- (V4);
			
			\node[square,draw] (E1) at (11,11) {$e_1^-$};
			\node[] (E10) at (12,11) {$3$};
			\node[square,draw] (E2) at (11,8) {$e_2^-$};
			\node[] (E20) at (12,8) {$2$};
			\node[square,draw] (E3) at (14,8) {$e_3^-$};
			\node[] (E30) at (15,8) {$2$};
			\node[square,draw] (E4) at (8,8) {$e_4^-$};
			\node[] (E40) at (9,8) {$4$};
			\node[square,draw] (E5) at (9,5) {$e_5^-$};
			\node[] (E50) at (10,5) {$4$};
			
			\node[square,draw,fill=lightgray] (E6) at (11,5) {};
			\node[square,draw,fill=lightgray] (E7) at (13,5) {};
			\node[square,draw,fill=lightgray] (E8) at (7,5) {};
			
			\draw (E1) -- (E2);
			\draw (E1) -- (E3);
			\draw (E1) -- (E4);
			\draw (E2) -- (E5);
			\draw (E2) -- (E6);
			\draw (E2) -- (E7);
			\draw (E4) -- (E8);
			\end{tikzpicture}
\caption{An ongoing exploration process \new{at step $i=5$} and the associated tree $T_f^-(5)$. Active nodes are depicted as small gray squares and marks are assigned to paired nodes.}
\label{fig:explore}
\end{figure}


\new{The nodes in \(T^-_{f}(i)\) correspond to heads that have been discovered, i.e. not in $\cU^-_i$, so we can assign them a marking $L_f^-(i)$ as follows. Label each node as \emph{paired} or \emph{active} depending on whether it belongs to $\cP^{-}_i$ or to $\cA^{-}_i$, respectively. Moreover, if the node corresponding to \(e_{i}^{-}\) is paired, assign it the \emph{mark} $\abs{\cE^{+}(v_{i})}$; see
\autoref{fig:explore}. To keep the notation light, we will abuse notation and write $(T_f^-(i))_{i\geq 0}$ for the sequence of marked trees $(T_f^-(i), L_f^-(i))_{i\geq 0}$.}

For half-edges \(e_1\) and \(e_2\), we define the \emph{distance} from $e_1$ to $e_2$, denoted by
\(\dist(e_1,e_2)\), to be the number of edges of the shortest path from \(v(e_1)\) to \(v(e_2)\) which starts
with the edge containing \(e_{1}\) if \(e_{1}\) is a tail, and which ends with the edge containing
\(e_{2}\) if \(e_{2}\) is head. Let us stress that this definition allows for head-to-head, tail-to-tail, tail-to-head and head-to-tail distances; for example, in~\autoref{fig:explore}, \(\dist(e_{2}^{-},
e_{1}^{-})=\dist(e_{2}^{+}, e_{1}^{+})=1\), \(\dist(e_{5}^{+}, e_{1}^{-})=3\), \(\dist(e_{4}^{-},
e_{1}^{+})=0\). 

\new{For $t\geq 0$, let \(i_t\) be the last step of the exploration process where a head at
distance \new{\(t\)} to \(f\) is paired. Then we can identify the heads at distance $r\leq t$ to $f$ in the graph, which we will later call the $r$-th in-neighbourhood of $f$, as the nodes in $T_f^-(i_{t})$ at distance $r$ from $f$.}

\new{We say that a rooted tree \(T\) of height $t$ where nodes are labelled as paired or active is \emph{feasible} if the set of active nodes is precisely the set of nodes at distance \(t\) from the root of the tree. If so, we let $p(T)$ denote the number of paired nodes of $T$. For instance, in the example displayed in~\autoref{fig:explore}, $T_f^-(4)$ is feasible with $p(T)=4$, while $T_f^-(5)$ is not.
In fact \(T^-_{f}(i_{t-1})\) is always a feasible tree of height $t$; we will use this when coupling the branching and the exploration processes. 
}

\subsection{Coupling the exploration and branching processes}
\label{sec:coupling}

The goal of this section is to couple the marked tree constructed during the exploration process with a marked branching process with distribution \(\eta = (\xi,\zeta) \eql \dout\), as defined in~\eqref{UHJC}, or close to it. \new{Let us first check that $\eta$ satisfies the conditions needed to use the results derived in~\autoref{sec:branching}.
On the one hand, by the hypothesis of~\autoref{thm:main}, we have \(\xi,\zeta \le M\) and \(\zeta\geq 2\), in particular $\new{\nu}=\E{\xi}= \E{\zeta}\geq 2>1$. On the other hand, as discussed in \autoref{sec:strategy}, we have restricted ourselves to $\delta^-\in\{0,1\}$, so $\p{\xi\in\{0,1\}}>0$. It follows that~\autoref{cond:BP} is satisfied by $\eta$.}

In addition, 
\begin{equation}\label{QPQX}
    \E{\xi/\zeta}
    = 
    \E{\doutm/\doutp}
    = 
    \sum_{k,\ell\geq 1} \frac{k}{\ell}\cdot \frac{\ell n_{k,\ell} }{m}
    =
    \frac{1}{m} \sum_{k,\ell\geq 1}  k n_{k,\ell}=1,
\end{equation}
\new{so it satisfies the additional condition required by \autoref{prop:martingale}.}

Starting here and throughout the rest of the paper, we will use asymptotic notation with respect to $n\to\infty$.


We will use a slightly perturbed version of \(\eta\).
The probability distribution \(\eta^{\ua}\) is defined by
\begin{equation}\label{VMLKM}
    \p{\eta^{\ua} = (k, \ell)}
    \coloneqq
    \begin{cases}
        c^{\ua} \p{\eta = (k, \ell)} & k \ge 1,
        \\
        c^{\ua} \p{\eta = (0, \ell)}+n^{-1/2} & k = 0,
    \end{cases}
\end{equation}
where \(c^{\ua} = 1-O(n^{-1/2})\) is a normalising constant.

\begin{rem}\label{rem:cont_parameters}

\new{The distribution $\eta^\ua$ also satisfies \autoref{cond:BP}. Indeed, (i) is satisfied as $\E{\xi^\ua} = (1+o(1))\nupm>1$, (iii) holds since, by construction, $\p{\xi^\ua\in\{0,1\}}\geq \p{\xi\in\{0,1\}}$ and (ii) and (iv) hold trivially.}
Moreover, an approximate version of~\eqref{QPQX} also holds,
\begin{align}\label{SDSO}
\left|\E{\xi^{\ua}/\zeta^{\ua}}-1\right| =O(n^{-1/2}).
\end{align}
\new{By taking $\delta\to 0$ as $n\to \infty$, it will be possible} to
apply~\autoref{prop:martingale} to marked branching processes with distribution $\eta^\ua$.

We can define \(\hnu^{\ua}\), \(\hat{H}^{\ua}\), \(\phi^{\ua}(a)\), \(a_{0}^{\ua}\) for \(\eta^{\ua}\)
analogously to \(\hnu\), \(\hat{H}\), \(\phi(a)\) and \(a_{0}\) for \(\eta\)
(see~\eqref{PXNE},~\eqref{EFND} and~\eqref{PBJK}). It is easy to verify that
$\hnu^\ua=(1+o(1))\hnu$, $\hat{H}^\ua=(1+o(1))\hat{H}$ and $\phi^\ua(a)=(1+o(1))\phi(a)$.
By continuity of $\phi(a)$, we have $\phi^\ua(a_0^\ua)=(1+o(1))\phi(a_0)$. So
\(\hat{H}^{\ua}/\phi^{\ua}(a_{0}^{\ua}) = (1+o(1)) \hat{H}/\phi(a_{0})\).
\end{rem}

Denote by \(\GW_{\eta}\) the tree associated with a marked branching process with distribution \(\eta\). \new{Let $T$ be a rooted tree where nodes are labelled as paired or active, which is feasible, and with non-negative integer marks on its paired nodes.}
We use the notation \(\GW_{\eta} \cong T\) to denote that \(T\)
is \emph{isomorphic} to \((\GW_{\eta})_t\), \new{where $t$ is the height of $T$}, in the following sense: for each paired node of \(T\), its degree and its mark agree with the ones in \(\GW_{\eta}\), \new{where individuals in $\GW_{\eta}$ are ordered according to their generation index first and their cousin index later}.  Similarly, we write \(T_f^-(p(T)) \cong T\) to denote that $T_f^-(p(T))$ is isomorphic to $T$ in the previous sense, \new{where $p(T)$ is the number of paired nodes in $T$}.
For a set of feasible marked trees
\(\cT\), let \(\{\GW_{\eta} \in \cT\} \coloneqq \cup_{T \in \cT} \{\GW_{\eta} \cong T\}\) and \(\{T_f^-
\in \cT\} \coloneqq \cup_{T \in \cT}  \{T^-_{f}(p(T)) \cong T\}\).  We will use the following coupling lemma to transfer the results on branching processes to random graphs.

\begin{lemma}\label{lem:eq_T_GW}
    Let $\beta \in (0, 1/4)$ and let \(\cT\) be a set of feasible marked trees such that
    \(p(T)\leq n^{\beta}\) for all \(T \in \cT\). For any $f\in \cE^-$, we have
    \begin{equation}\label{PLUL}
        \left(
        1+o(1)
        \right)
        \p{\GW_{\eta}\in \cT} +O(n^{2\beta -3/2})
        \leq 
        \p{T_f^- \in \cT}
        \leq  
        \left( 
            1+o(1)
        \right)
        \p{\GW_{\eta^{\ua}}\in \cT}
        .
    \end{equation}
\end{lemma}
\begin{proof}

\new{We first prove the lower bound. Our strategy is to couple the exploration process with an artificial process, with failing probability $O(n^{2\beta-3/2})$, and then analyse the latter process to show that a given tree $T\in \cT$ is produced with probability comparable to the probability with which a branching process with distribution $\eta$ produces it.}
  
\new{The artificial exploration process is defined analogously as the exploration process introduced in the previous section with one modification: instead of constructing a matching between heads and tails, it produces a number of non-necessarily disjoint pairs of them. For $k,\ell\geq 0$, we say that $(k,\ell)$ is a \emph{rare} degree pair if $\p{\eta=(k,\ell)}\leq n^{-1/2}$. Let $\cR^\pm$ be the set of heads/tails whose end-points have a rare degree pair.
}

\new{The artificial process is defined analogously as the exploration process.
Starting as in there, we set $i=1$ and proceed as follows:
\begin{enumerate}
 \item[(i$^*$)] Let \(e_{i}^{-}\) be one of the heads which became active earliest in \(\cA_{i-1}^{-}\).
    \item[(ii$^*$)] Construct the pair \((e_{i}^{-},e_{i}^{+})\) where $e_i^{+}$ is chosen uniformly at random from $(\cE^+\setminus \cP_{i-1}^+) \cup \cR^+$. Let \(v_{i} = v(e_{i}^{+})\) and \(\cP_{i}^{\pm} =
        \cP_{i-1}^{\pm} \cup \{e^{\pm}_{i}\}\).
    \item[(iii$^*$)] Update the active sets as follows:
\begin{itemize}
\item[(a)] If \(e^{+}_{i} \in \cA_{i-1}^{+} \cup (\cP_{i-1}^+ \cap \cR^+)\),
        then \(\cA_{i}^{\pm} = \cA_{i-1}^{\pm}\setminus \{e_i^\pm\}\).
\item[(b)] If \(e_i^{+}\in \cU^{+}_{i-1}\), then \(\cA_{i}^{\pm} = (\cA_{i-1}^{\pm} \cup \cE^{\pm}(v_{i})) \setminus
        \{e^{\pm}_{i}\}\).        
\end{itemize}
\item[(iv$^*$)] If \(\cA_{i}^{-}\!\!=\emptyset\), terminate; otherwise, let \(\cU_i^{\pm}\! =\!\cE^{\pm}\setminus (\cA^{\pm}_{i} \cup\cP^{\pm}_{i})\), set \(i=i+1\) and go to (i).
\end{enumerate}
}

\new{In words, the artificial process allows the tails in rare degree vertices to be paired multiple times, while the tails in degrees that are not rare can be paired only once.}

\new{Consider the sequence of trees $(T_f^*(i))_{i\geq 0}$ constructed from the artificial process as $(T_f^-(i))_{i\geq 0}$ was constructed from the exploration process, with the following difference: if $e_i^+\in \cP_{i-1}^+ \cap\cR^+$, then \(T^*_{f}(i)\) is obtained from  \(T^*_{f}(i-1)\) by adding \(\abs{\cE^{-}(v_{i})}\) children to the node corresponding to \(e_{i}^{-}\) (similarly as we did when $e_i^+\in \cU^+_{i-1}$) and considering this node as paired with mark $\abs{\cE^{+}(v_{i})}$.
}

\new{There is a natural coupling between $(T_f^-(i))_{i\geq 0}$ and $(T_f^*(i))_{i\geq 0}$ which fails at the first step $j\geq 1$ such that $e^+_j\in \cP_{j-1}^+\cap \cR^+$. For this to happen, at step $j$ there must be a tail of a rare degree vertex that has already been paired, i.e. a step $j'< j$ such that $e^+_{j'}\in \cR^+$. As the maximum degree is bounded, $|\cR^+|= O(\sqrt{n})$. Therefore, if $i\leq m/2$, the probability that the coupling fails before step $i$ is
\begin{equation}\label{MDYE}
\begin{aligned}
\p{ (T_f^-(j))_{i\geq j\geq 0}\neq (T_f^*(j))_{i\geq j\geq 0}} &\leq 
\sum_{j=1}^i \frac{|\cP_{j-1}^+|}{|(\cE^+\setminus \cP_{j-1}^+) \cup \cR^+|}\sum_{j'=1}^{j-1} \frac{|\cR^+|}{|(\cE^+\setminus \cP_{j'-1}^+) \cup \cR^+|} \\
&= O(1) \sum_{j=1}^i \frac{j|\cR^+|}{m^2} = O(i^2 n^{-3/2}).
\end{aligned}
\end{equation}
Fix $T\in \cT$. We proceed to compare the probabilities that $T$ is generated by the artificial process and by a branching process with distribution $\eta$.} 

\new{For $i\in [p(T)]$, let \(k_i\) be the number of children of the \(i\)-th node in \(T\) in the BFS order, and \(\ell_i\) its mark. Let \(K_{i}\) be the number of children of $e_i^-$ in $T_f^*(i)$ and let $L_i= |\cE^{+}(v_i)|$ be the mark given to $e_i^-$ in $T_f^*(i)$.
Let \(E_{i} =\cap_{j = 1}^{i} \{K_{j} = k_{j}, L_{j} = \ell_{j}\}\). Write $q_{k,\ell}=\p{\eta = (k, \ell)}$ and let \(q_{k,\ell}(i)\coloneqq\p{K_i=k,L_i=\ell\mid E_{i-1}}\).}

\new{If $(k,\ell)$ is rare, then at step (ii$^*$) all tails in $\cR^+$ are candidates for $e_i^+$ and thus,
\begin{equation}
        q_{k,\ell}(i) = \frac{\ell n_{k,\ell}
         }{|(\cE^+\setminus \cP_{i-1}^+) \cup \cR^+|}
        \geq 
        q_{k,\ell}.
    \end{equation}
If $(k,\ell)$ is not rare, then
   \begin{equation}
        q_{k,\ell}(i) \geq \frac{\ell n_{k,\ell} - i M
         }{m}
        = 
        q_{k,\ell}-O(n^{\beta-1})
        =
        (1+O(n^{\beta-1/2})) q_{k,\ell},
    \end{equation}
    where the third step uses that $q_{k,\ell}> n^{-1/2}$. 
}

\new{    
    It follows that
    \begin{equation}\label{SDKS}
        \p{T_f^*(p(T)) \cong T} 
        = \prod_{i=1}^{p(T)} q_{k_i,\ell_i}(i)
        \geq \prod_{i=1}^{p(T)} (1+O(n^{\beta-1/2})) q_{k_i,\ell_i}
        = (1+o(1))\p{\GW_{\eta}\cong T}.
    \end{equation}
    since $2\beta<1/2$.	By adding~\eqref{SDKS} over all $T\in \cT$, we obtain
        \begin{equation}\label{SDLS}
        \p{T_f^* \in  \cT} 
        \geq (1+o(1))\p{\GW_{\eta}\in \cT}.
    \end{equation}
    The lower bound on~\eqref{PLUL} follows from using the previous equation and~\eqref{MDYE} with $i=p(T)\leq n^\beta$,
    \begin{equation}\label{ODKS}
    \begin{aligned}
        \p{T_f^-\in \cT} &\geq  \p{T_f^* \in \cT} - \p{ (T_f^-(j))_{n^{\beta}\geq j\geq 0}\neq (T_f^*(j))_{n^{\beta}\geq j\geq 0}}\\
        &= (1+o(1))\p{\GW_{\eta}\in \cT} + O(n^{2\beta-3/2}). 
    \end{aligned}
    \end{equation}
    }
    Let us now prove the upper bound. Write $q_{k,\ell}^\ua=\p{\eta^\ua = (k, \ell)}$.
    For all \(k\geq 1\), \(\ell \ge 0\) and $i\in [p(T)]$
    \begin{align}\label{eq:rel_i_pnl}
        q_{k,\ell}(i) 
        \leq\frac{\ell \, n_{k,\ell}}{m - (i-1)} 
        =  (1+O(n^{\beta-1})) q_{k,\ell}^{} 
        =  (1+O(n^{-1/2})) q_{k,\ell}^{\ua}
        . 
    \end{align}
    Also, for all $\ell\geq 0$ and $i\in [p(T)]$
    \begin{align}\label{eq:rel_i_pn0}
        q_{0,\ell}(i) 
        \leq
        \frac{\ell\, n_{0,\ell} + p(T)M}{m - (i-1)} 
        =
        (1+O(n^{\beta-1/2}))q_{0,\ell}^{}
        \leq      
        (1+O(n^{\beta-1/2}))q_{0,\ell}^{\ua} 
        ,
    \end{align}
    where the second step uses that \(q_{0,\ell}^{\ua} \geq  n^{-1/2}\). The rest of the argument is analogous to the lower bound.
\end{proof}

\section{Stationary distributions}\label{sec:pi:min}

We will proceed to prove~\autoref{thm:main} as in~\cite{caputo2020a}, using
the results obtained in previous sections and some of our results in~\cite{cai2020, cai2020a}. In these papers, it was  assumed that \(D_{n}=D\), the random vector of the in- and out-degree of a uniform random vertex, converges in distribution as \(n \to \infty\). This is not implied by our assumptions. However, having bounded in- and out-degrees implies that $D_n$ is tight and by Prokhorov's theorem, we can extract a subsequence \( (a_{n})_{n \ge 1}\) such that \(D_{a_{n}}\) converges in distribution. Thus we can
    still apply the results in~\cite{cai2020, cai2020a} in the context of this paper.

Throughout the section, all the asymptotic notation must be understood as $n\to\infty$. We also assume that $n$ is sufficiently large in some of the inequalities displayed.

\subsection{The largest strongly connected component}\label{sec:SCC}

For $\delta^+\geq 2$, whp there is a linear size strongly connected component in
\(\vecGn\)~\cite{cai2020,cooper2004}. We will first show that whp this component is \new{the only one that is closed}, which implies that the simple random walk on $\vecGn$ has a unique stationary
distribution whp.

\new{Recall the definition of distance between half-edges given in~\autoref{sec:explore}.}
For a head/tail $e^{\pm} \in \cE^{\pm}$ and $k\geq 0$, define
\begin{equation}\label{AQIR}
    \begin{aligned}
    \cN^{+}_{k}(e^{+}) 
    &
    \coloneqq
    \left\{ 
        f^{+} \in \cE^{+}
        :
        \dist(e^{+}, f^{+}) = k
    \right\}
    ,
    \qquad
    &
    \cN^{-}_{k}(e^{-}) 
    \coloneqq
    \left\{ 
        f^{-} \in \cE^{-}
        :
        \dist(f^{-}, e^{-}) = k
    \right\}
    ,
    \\
    \cN^{+}_{\le k}(e^{+}) 
    &
    \coloneqq
    \left\{ 
        f^{+} \in \cE^{+}
        :
        \dist(e^{+}, f^{+}) \le k
    \right\}
    ,
    \qquad
    &
    \cN^{-}_{\le k}(e^{-}) 
    \coloneqq
    \left\{ 
        f^{-} \in \cE^{-}
        :
        \dist(f^{-}, e^{-}) \le k
    \right\}
    .
    \end{aligned}
\end{equation}
Similarly, for a
vertex \(u \in [n]\), let $\cN^\pm_k(u)$ and $\cN^\pm_{\le k}(u)$ be the sets of vertices at distance
$k$ and at most \(k\) from/to $u$, respectively.

Throughout this section, we set 
\begin{align}\label{eq:omega}
\omega\coloneqq \log^6 n.
\end{align}

For every head/tail $e^\pm \in \cE^\pm$, consider the
random variable
\begin{equation}\label{YBTI}
t_{\omega}^{\pm}(e^{\pm}) \coloneqq \inf\{t\geq 0:\,\abs{\cN^{\pm}_t(e^{\pm})} \geq \omega\};
\end{equation}
that is, the smallest distance $t$ for which there are at least $\omega$ half-edges in $t$-th in-/out-neighbourhood of $e^\pm$. 
\begin{prop}\label{prop:strongly}
    Let $M\in\mathbb{N}$ and suppose that $\delta^+\geq 2$ and $\Delta^{\pm} \leq M$.
    Let $\cC_{0}$ denote a largest strongly connected component in $\vecGn$.
    Let    
    \begin{equation}\label{NVEU2}
         \cE^-_{0}
        \coloneqq
        \{f\in \cE^-: t^-_\omega(f)<\infty\}
        .
    \end{equation}    
    Let \(E_{0}\) be the event that \(\cC_{0}\) is \new{the only closed strongly connected component} and has vertex set $\cV_0=\cV(\cE_0^-)$.  
    Then \(\p{E_{0}} = 1 -
    o(1)\).  Thus, whp the simple random walk on $\vecGn$ has a unique stationary distribution
    supported on $\cV_0$.
\end{prop}

\begin{proof}
	\new{The proposition can be easily proved combining some results existing in the literature. Here we do it, skipping some straightforward computations.}
    Let $h\coloneqq 1+\log_{2} (2\omega)= O(\log\log n)$. 
    Then, 
    \begin{equation}\label{QETJ}
        \begin{aligned}
            \p{\cup_{e \in \cE^{+}} \{t_{\omega}^{+}(e) > h\}}
            &
            \le
            \p{\cup_{v \in [n]} \left\{|\cN^+_{h-1}(v)|< \omega  \right\}}
            \\
            &
            \le
            \p{\cup_{v \in [n]} \left\{|\cN^+_{h-1}(v)|< \frac{1}{2}(\delta^+)^{h-1} \right\}}
            ,
        \end{aligned}
    \end{equation}
    where we used $|\cN^+_h(e)|= |\cN^+_{h-1}(v)|$, where $v$ is the vertex incident to the head paired with $e$. It follows from~\cite[Lemma 2.2]{caputo2020a} that the probability in~\eqref{QETJ} is $o(1/n)$.
   So, \(t^{+}_{\omega}(e) \leq  h\) for all \(e \in \cE^+\) whp.
   

    By~\cite[Lemma~6.2]{cai2020a} whp {every} \(f\in \cE^-\) either has \(t_{\omega}^{-}(f)=\infty\) or
    \(t_{\omega}^{-}(f)=O(\log n)\).  Conditioning on \(t_{\omega}^{+}(e) = O(\log \log n)\) and
    \(t_{\omega}^{-}(f)=O(\log n)\),~\cite[Proposition~7.2]{cai2020a} implies that there is a path
    from \(e\) to \(f\) with probability \(1-o(n^{-2})\). Thus, by a union bound over all
    choices of \(e\) and \(f\), we have that whp there is path from every  \(e \in
    \cE^{+}\) to every
    \(f\in \cE_0^-\); and in particular, from every $u\in[n]$ to every $v\in \cV_0$. In other words, whp \(\cC_{0}\) is the only closed strongly connected component and contains all vertices in $\cV_0$.
    
    \new{It remains to show that there are no more vertices in the component. For this purpose, we can use~\cite[Proposition 6.1]{cai2020a} which states that there exists $\alpha>0$ depending on the degree sequence, such that the probability of the event $E_f=\cap_{k\leq t}\{0<|\cN^{-}_{k}(f)|<\omega\}$ is at most $e^{-\alpha t}$, for $t=\Theta(\log{n})$. Choose $\beta>0$ satisfying $\alpha\beta>1$. Set $t=\beta \log{n}$ so the previous probability is $o(1/n)$ and, by a union bound, the event $\cap_{f\in \cE^-} E_f^c$ holds whp. If $t^-_\omega(f)=\infty$ for $f\in \cE^-$, the previous event implies that $|\cN^{-}_{t}(f)|=0$ and thus the number of heads from where $f$ can be reached is $O(\omega\log n)=o(n)$ and $f$ is not in $\cC_0$. It follows that whp $\cC_0$ has vertex set $\cV_0$.}
\end{proof}

\subsection{Random walk on in-half-edges (heads)}\label{sec:RW_heads}

It will be more convenient to think about the random walk as moving from head to head instead of from vertex to vertex. \new{We remark that the random walk is still the same, we simply track the heads that it traverses instead of the vertices it goes through}. So consider the
random process \(\walke\) {with state space \(\cE^{-}\) and, conditioning on \(\vecGn\) and
\(Z_{t}^{\bfe} = f\), let \(Z_{t+1}^{\bfe}\) be chosen uniformly at random among all heads
paired with tails of \(v(f)\), the endpoint of \(f\). That is, the random walk moves from head to head in the forward direction and the original walk can be recovered as $Z_t=v(Z_t^{\mathbf{e}})$. It follows from~\autoref{prop:strongly}
that \(\walke\) also has a unique stationary distribution supported on \(\cE^-_0\), denoted by
\(\pi^{\bfe}\).

Recall that $\pimin \coloneqq \min\{\pi(v):\, v\in [n],  \pi(v)>0\}$.  Let $\pimine \coloneqq
\min\{\pi^{\bfe}(f):\, f\in \cE^-, \pi^{\bfe}(f)>0\}$. Define 
\begin{equation}\label{MFOD}
    \pi_{0} \coloneqq \min\{\pi(v):\, v \in \cV_0\}
    ,
    \qquad
    \pi_{0}^{\bfe} \coloneqq \min\{\pi^{\bfe}(f):\, f \in \cE_{0}^{-}\}
    .
\end{equation}
\autoref{prop:strongly} also implies that whp $\pi_{0}=\pimin$ and \(\pi_{0}^{\bfe} = \pimine\).  The
following lemma shows that whp \(\pi_{0}^{\bfe}\) differs from \(\pi_{0}\) by a bounded
factor.  Thus, it suffices to prove~\autoref{thm:main} for $\pi_{0}^{\bfe}$. 

\begin{lemma}\label{lem:piminez}
Under the hypotheses of~\autoref{prop:strongly}, whp 
	\begin{equation}
	\pi_{0}^{\bfe} \le \pi_{0} \le M \pi_{0}^{\bfe}.
	\end{equation}
\end{lemma}

\begin{proof}
\new{
    Let \(\cV_{0}\), \(\cE_{0}^-\) and \(E_{0}\) be as in~\autoref{prop:strongly}. Since \(E_{0}\) holds
    whp, it suffices to prove the lemma conditioned on \(\vecGn \in E_{0}\). So let $\pi$ be the unique stationary distribution of $(Z_t)_{t\geq 0}$.}

\new{    If $v\in \cV_0$, then there exists $u\in \cV_0$ with $(u,v)\in E(\vecGn)$. Let $f_0\in \cE^-(v)$ be a head paired with a tail in $\cE^+(u)$, in particular $f_0\in \cE_0^-$. By the uniqueness of $\pi$ and $\pi^{\bfe}$, we obtain
    \begin{equation}\label{KDRL}
        \pi(v)
        =
        \sum_{f\in \cE^-(v)}\pi^{\bfe}(f)
        	\geq
        \pi^{\bfe}(f_0)
        \ge 
        \piminez
        .
    \end{equation}
     Since the choice \(v \in
    \cV_{0}\) is arbitrary, it follows that \(\pi_{0}\ge \piminez\).
}

\new{
    For the other direction, we now choose \(f_0 \in \cE_{0}^{-}\) with \(\pi^{\bfe}(f_0) = \pi_{0}^{\bfe}\).
    Let \(w\) be the vertex such that \(f_0\) is paired with a tail in $\cE^+(w)$, which in particular satisfies $w\in \cV_0$. By the stationary and uniqueness of $\pi$ and $\pi^{\bfe}$, we have
    \begin{equation}\label{POVN}
        \pi_{0}^{\bfe} = \pi^{\bfe}(f_0)
        =
        \frac{1}{d^{+}_{w}}
        \pi(w)
        \ge
        \frac{1}{M}
        \pi_{0}
        .
    \end{equation}
    Thus \(\pi_{0} \le M \piminez\).
    }
\end{proof}

\subsection{Lower bound for \texorpdfstring{$\piminez$}{Pi min zero}}
\label{sec:lower}

To prove a lower bound for $\piminez$ it suffices to understand, for each $f\in \cE^-_0$, the probability that the walk reaches $f$, uniformly for all starting points $e\in \cE^-$; see~\autoref{lem:PT} and the discussion after it. We use
the ideas introduced in~\cite{bordenave2018,bordenave2019,caputo2020a} to capture the weight of typical trajectories departing from $e$. Our main contribution lies in controlling the total weight of the trajectories landing at $f$, \new{a task that proves to be significantly more complex than in previous studies.}

Define the \emph{out-entropy} $H^+$ and the
\emph{entropic time} $\tent$ as
\begin{equation}\label{GNKU}
    H^+
    \coloneqq 
    \new{\frac{1}{\lambda}\sum_{k\geq 1, \ell \geq 0} k \log \ell \cdot \p{D=(k,\ell)}}
     \text{ and } \tent
    \coloneqq
    \frac{\log n}{H^+}.
\end{equation}
We stress the similarity between $H^+$ and $\hat{H}^-$ defined in~\eqref{SKFN}; in particular,
$H^+=\E{\log\din^+}$, where $\din^+$ is the out-degree of the in-size-biased distribution of $D$. While the subcritical in-entropy $\hat{H}^-$ relates to the weight of trajectories in atypical (subcritical) in-neighbourhoods, the out-entropy $H^+$ relates to the weight of trajectories in typical (supercritical) out-neighbourhoods.

Fix $\theta>0$ sufficiently small and let 
\begin{equation}\label{BMGX}
    h^+\coloneqq (1-\theta)\tent, 
    \qquad 
    h^-\coloneqq \frac{3\theta}{\log \delta^+}\log n
    .     
\end{equation}
For
$f\in \cE^-$,  define 
\begin{equation}\label{SOAS}
h(f) \coloneqq t^{-}_\omega(f) \wedge \omega,
\end{equation}
where $t^{-}_\omega(f)$ is as defined in~\eqref{YBTI}. Also define
\begin{equation}\label{IOWT}
    \tau(f) \coloneqq h^++h^-+h(f)
    .
\end{equation}
\begin{rem}
Bordenave, Caputo and Salez~\cite{bordenave2018,bordenave2019} showed that the mixing time of the
random walk on $\vecGn$ coincides with the \emph{entropic time} and exhibits cutoff. \new{In addition, there is a relation between the parameters considered and the distances in the random graph. First, $\tent\geq \log_{\nupm} n$, where $\nupm$ is defined in~\eqref{SPEK} (see~\cite[Eq. (3.3)]{caputo2020a}), which is whp the asymptotic value of the distance between two uniformly chosen vertices  of $\vecGn$. Second, by the results in~\cite{cai2020a}, \(\tent+\max_{f \in \cE^{-}_0} h(f)\) is whp at least
the asymptotic value of the diameter of $\vecGn$.}
\end{rem}

Throughout this section, we will use the letters $a,b$ for tails in $\cE^+$ and the letters $e,f,g$ for heads in $\cE^-$.
Define 
\begin{equation}\label{EUSL}
P^{t}(e, f)\coloneqq \p{Z^{\bfe}_{t} = f\mid Z^{\bfe}_{0}=e}.
\end{equation}   
To lower bound $\pi_0^\bfe$ it suffices to prove the following.
\begin{lemma}\label{lem:PT}
    \new{Under the hypotheses of~\autoref{prop:strongly},} for every $\varepsilon>0$, whp for all \(e \in \cE^{-}\) and \(f \in \cE_{0}^{-}\)
\begin{equation}\label{GSHI}
    P^{\tau(f)}(e,f) \ge n^{-(1+\hat{H}^-/\phi(a_{0}))-\varepsilon}
    .
\end{equation}
\end{lemma}

Indeed, assuming~\autoref{lem:PT} and by stationarity, whp for all \(f\in \cE^{-}_{0}\) we have
\begin{align}\label{EODJ}
\pi^{\bfe}(f) = \sum_{e\in \cE^-} \pi^\bfe(e) P^{\tau({f})}(e,f)\geq
    n^{-(1+\hat{H}^-/\phi(a_{0}))-\varepsilon} ,
\end{align}
and, by~\autoref{prop:strongly} and~\autoref{lem:piminez}, the lower bound in~\autoref{thm:main} follows.

The rest of this subsection is devoted to the proof of~\autoref{lem:PT}; here we give an outline of it.
From now on, we fix two heads $e\in \cE^-$ and $f\in\cE^-_0$.
First, we expose the out-neighbourhood of $e$ up to distance $h^+$ (\emph{out-phase}) using a modified exploration process, which is described in~\autoref{sec:out-neigh}. Second, we expose the in-neighbourhood of $f$ up to distance $h^-+h(f)$ (\emph{in-phase}) using the exploration process described in~\autoref{sec:explore}. This is done in two steps, first exposing the half-edges from which $f$ is at distance at most $h(f)$, and later the remaining ones; see~\autoref{sec:in-neigh}. In these first two sections, we also define the notion of weight for tails and heads.
A lower bound for $P^{\tau(f)}(e,f)$ is established in~\autoref{sec:connect} by considering a random sum of weights, see~\eqref{JVQW}. By restricting ourselves on partial pairings exposed by the out- and in-phases that are ``good'' in a certain sense (see~\eqref{QHKO}), we show that the number of edges from the $(h^+)$-th out-neighbourhood of $e$ to the $(h^-+h(f))$-th in-neighbourhood of $f$ is concentrated, which yields concentration of the random sum of weights around its expectation; see~\autoref{lem:concentration}. In~\autoref{sec:typ} we show that, indeed, most of the exploration processes produce ``good'' partial pairings; see~\autoref{lem:whp}. This is where most of the results derived in~\autoref{sec:branching} are used. Finally, we conclude the proof of~\autoref{lem:PT} in~\autoref{sec:LB_proof}.

\subsubsection{Out-neighbourhood of $e$}\label{sec:out-neigh}

In the out-phase we build a directed rooted tree $T^+_{e}$, partially exposing the out-neighbourhood
of $e$. The root of $T^+_{e}$ represents the head $e$ and all its other nodes represent tails in
$\cE^+$.  For a tail $a$ represented in $T^+_{e}$, let $\bfh(a)$ denote its height in $T^+_{e}$.
Define its \emph{weight} by
\begin{equation}\label{IBHZ}
\bfw(a) \coloneqq \prod_{i=1}^{\bfh(a)} \frac{1}{d^+_{v(a_i)}}.
\end{equation}
where $a_{1}, \dots, a_{\bfh(a)}\new{=a}$ are the tails in the path from $e$ to $a$ in
$T^+_{e}$.  \new{The weight of $a$ can be understood as the probability that a random walk starting at $e$ reaches $a$ by following the unique path in the tree $T_e^+$ that connects the two of them, so $P^{\bfh(a)}(e, g) \ge \bfw(a)$, where $g$ is the head paired with $a$.}

To construct $T^+_{e}$ we use a procedure similar to~\autoref{sec:explore}. For the sake of completeness, we present the full modified procedure. Let
\begin{equation}\label{eq:APUF2}
    \cA_{0}^{\pm}=\cE^{\pm}(v(e)),
    \cP_{0}^{\pm}=\emptyset \text{ and }
    \cU_{0}^{\pm}=\cE^{\pm} \setminus \cA^{\pm}_{0}
    .
\end{equation}
Then set \(i=1\) and proceed as follows: \leavevmode
\begin{enumerate}[\normalfont(i)]
    \item Let \(e_{i}^{+}\) be one of the tails in
        \(\cA_{i-1}^{+}\) that
    maximises $\bfw(a)$ among all $a\in \cA_{i-1}^+$ with $\bfh(a)\leq h^+-1$ and
    $\bfw(a)\geq n^{-1+\theta^2}$, \new{where $h^+$ and $\theta$ are as in~\eqref{BMGX}}.
    \item Pair \(e_{i}^{+}\) with a head \(e_{i}^{-}\) chosen uniformly at random from
        \(\cE^-\setminus \cP^-_{i-1}\).  Let \(v_{i} = v(e_{i}^{-})\) and \(\cP_{i}^{\pm} =
        \cP_{i-1}^{\pm} \cup \{e^{\pm}_{i}\}\).
    \item Update the active sets as follows:
\begin{itemize}
\item[(a)]  If \(e^{-}_{i} \in \cA_{i-1}^{-}\),
        then \(\cA_{i}^{\pm} = \cA_{i-1}^{\pm}\setminus \{e_i^\pm\}\); 
\item[(b)] If \(e_i^{-}\in
        \cU^{-}_{i-1}\), then \(\cA_{i}^{\pm} = (\cA_{i-1}^{\pm} \cup \cE^{\pm}(v_{i})) \setminus
        \{e^{\pm}_{i}\}\).
\end{itemize}        
    \item \new{If there is no \(a\in\cA_{i}^{+}\) with $\bfh(a)\leq h^+-1$ and
    $\bfw(a)\geq n^{-1+\theta^2}$}, terminate; otherwise, let \(\cU_i^{\pm}\! =\!\cE^{\pm}\setminus (\cA^{\pm}_{i} \cup
        \cP^{\pm}_{i})\), set \(i=i+1\) and go to (i).
\end{enumerate}

The two main differences with respect to the procedure defined in~\autoref{sec:explore} is that we explore the out-neighbourhood instead of the in-neighbourhood of a head, and that we use a different priority rule to select the current tail to be paired in step (i); \new{in particular, we halt the process when no tail satisfies the required conditions}. In~\cite[Lemma 7]{bordenave2019}, the authors showed that constructing
$T^+_{e}$ \emph{deterministically} pairs at most $\kappa^{+}\coloneqq n^{1-\theta^2/2}$ edges.

\new{To prove~\autoref{lem:PT}}, it will be convenient to consider trajectories that are not too heavy.  Following the ideas of Bordenave, Caputo and Salez~\cite{bordenave2018,bordenave2019} as described by Caputo and Quattropani~\cite{caputo2020a}, it suffices to only consider \emph{nice paths} in $T_{e}^{+}$, that is, we restrict our attention to the tails $a$ with $\mathbf{h}(a)=h^+$ such that $\mathbf{w}(a)\leq n^{2\theta-1}$.

%

\subsubsection{In-neighbourhood of $f$}\label{sec:in-neigh}
In the in-phase we build a directed rooted tree $T_f^-$, exposing $\cN^-_{\leq h^-+h(f)}(f)$ conditionally on $T^+_e$. We use
the exploration process defined in~\autoref{sec:explore} with one tiny modification: if $e^-_i$ has
already been paired with some $e^+$ in the exploration of \(T^+_e\), we let $e^+_i=e^+$ instead of choosing it randomly and we continue the exploration process \new{(including the exploration of the part already revealed by \(T^+_e\))}. We stop once
all heads at distance \(h^{-}+h(f)\) from $f$ have been activated and we let \(T_{f}^{-}\) be the feasible tree generated at this point.  For a head $g$ in $ T^-_{f}$, let $\bfh(g)$ be
its height in the tree, in particular, $\bfh(g)=\dist(g,f)$.  We define the \emph{weight} of $g$ by
\begin{equation}\label{TQIN}
    \bfw(g)\coloneqq \prod_{i=1}^{\bfh(g)} \frac{1}{d^+_{v(g_i)}}, 
\end{equation}
where $f=g_{0}, \dots, g_{\bfh(g)}\new{=g}$ are the heads in the path from $f$ to $g$ in
$T^-_{f}$. 
\new{The weight of $g$ can be understood as the probability that a random walk starting at $g$ reaches $f$ by following the unique path in the tree $T_f^-$ that connects the two of them}, so  $P^{\bfh(g)}(g, f) \ge \bfw(g)$. \new{Note the similarity between the weight of a head and the parameter $\Gamma_{i,t}$ defined in~\eqref{OSKE}, which can be thought as its branching process analogue as we will see later.} 

For any $\alpha>0$, we can let $\theta$ be small enough such that the number of edges exposed
in the in-phase is at most 
\begin{equation}\label{TGWY}
    \kappa^{-} \coloneqq \omega M h(f) + \new{\sum_{i=1}^{h^-}} \omega M^{i+1}= n^{3(\log{M}/\log
    \delta^{+})\theta+o(1)} = O(n^{\alpha}),
\end{equation}
\new{where we used that $|\cN^-_{h(f)}(f)|< \omega M$ and that $h(f)$ is a random variable bounded by $\omega$; see~\eqref{SOAS}.}

It will be convenient to define a truncated version of the head weights.
Consider 
\begin{align}\label{SOEF}
\gamma\coloneqq n^{-{(1+\theta)\hat{H}^-}/{\phi(a_{0})}}.
\end{align}
For $g$ in $T_f^-$ with $\bfh(g)\geq h(f)$, we define its \emph{truncated weight} by
\begin{equation}\label{FKDN}
    \hat{\bfw}(g)\coloneqq (\bfw(\hat{g})\wedge \gamma) \prod_{i=h(f)+1}^{\bfh(g)} \frac{1}{d^+_{v(g_i)}} .
\end{equation}
where $\hat{g}$ is the unique head in the path from $f$ to $g$ in $T_f^-$ with $\bfh(\hat{g})=h(f)$. In words, the truncated weight caps the contribution of the \new{first} $h(f)$ steps by $\gamma$. By the definition of \(h^{-}\) in
\eqref{BMGX}, for \(g\) with $\mathbf{h}(g)=h^-+h(f)$ we have 
\begin{align}\label{SLKD}
\hat{\bfw}(g)&\leq (\delta^+)^{-h^-}\gamma \leq \gamma n^{-3\theta}.
\end{align}

\subsubsection{Connecting the  two neighbourhoods}\label{sec:connect}
Let $\sigma_0$ be a partial realisation of the directed configuration model revealed during the out- and in-phases, and recall that at most $\kappa\coloneqq
\kappa^{+}+\kappa^{-}$ heads and tails have been paired in $\sigma_0$. Then $T_{e}^{+},
T_{f}^{-}$, $h(f)$, $\cN^-_{\leq h^-+h(f)}(f)$ and $\tau(f)$ are all measurable with respect to
$\sigma_0$. Given $\sigma_0$, let $\cA$ be the set of unpaired tails $a$ in \(T_{e}^{+}\) with
$\bfh(a)=h^+$, and let $\cG$ be the set of unpaired heads $g$ in $T_{f}^{-}$ which satisfy
$\bfh(g)=h^-+h(f)$. 
Let $\sigma$ be a complete pairing of half-edges compatible with $\sigma_0$.  Let \(\{\sigma(a)=g\}\) be the event that \(a\) is paired with \(g\) in \(\sigma\). Conditioned on $\sigma$ being compatible with $\sigma_0$ and $\{\sigma(a)=g\}$, we have $P^{\bfh(a)}(e, g) \ge \bfw(a)$. Thus, conditioned on $\sigma$ being compatible with \(\sigma_0\), we can lower bound the
desired probability as follows
\begin{equation}\label{JVQW}
P^{\tau(f)}(e,f) \geq \hat{P}^{\tau(f)}(e,f)\coloneqq \sum_{a\in \cA}\sum_{g\in \cG} \bfw(a) \mathbf{\hat
w}(g)  \mathbf{1}_{\sigma(a)=g}\mathbf{1}_{\bfw(a)\leq n^{2\theta-1}} .
\end{equation}
Consider the events
\begin{equation}\label{QHKO}
A_{e,f}=A_{e,f}(\sigma_0) \coloneqq \sum_{a\in\cA} \bfw(a) \mathbf{1}_{\bfw(a)\leq n^{2\theta-1}}
\text{ and }
G_{e,f}=G_{e,f}(\sigma_0) \coloneqq \sum_{g\in\cG}\mathbf{\hat{w}}(g)
,
\qquad
\end{equation}
which are measurable with respect to $\sigma_0$. \new{Despite the notation, observe that $A_{e,f}$ does not depend on $f$.} Consider the event 
\begin{equation}\label{DLJA}
\cX_{e,f}=
\left\{\sigma_0
    :\,
    A_{e,f}(\sigma_0)\geq \frac{1}{2}, 
    G_{e,f}(\sigma_0) \geq 
    \frac{\omega \gamma}{4}    
\right\}
.
\end{equation}
If $\sigma_0\in \cX_{e,f}$, then
\begin{align}\label{eq:FKEN}
    \E{\hat{P}^{\tau(f)}(e,f)\mid \sigma_0}\geq \frac{1}{m} A_{e,f}(\sigma_0) G_{e,f}(\sigma_0)
    \geq \frac{\omega \gamma}{8 m}
    \geq \frac{2\gamma}{n}
    .
\end{align}
We prove a concentration result similar to~\cite[Lemma 3.6]{caputo2020a}, exploiting the truncated nature of $\hat{P}^t$:
\begin{lemma}\label{lem:concentration}
For every $\sigma_0\in \cX_{e,f}$ and every $c\in (0,1)$
\begin{align*}
\p{\hat{P}^{\tau(f)}(e,f)\leq (1-c)\E{\hat{P}^{\tau(f)}(e,f)\mid \sigma_0}\mid \sigma_0}\leq
\exp(-\frac{c^2 n^{\theta}}{3} ),
\end{align*}
where $\theta$ is as in~\eqref{BMGX}.
\end{lemma}
\begin{proof} 
One can write $\hat{P}^{\tau(f)}(e,f)=\sum_{a\in \cE^+} w(a,\sigma(a))$ where 
\begin{equation}\label{JFSZ}
w(a,g) 
= \bfw(a) \hat{\bfw}(g)\mathbf{1}_{\bfw(a)\leq n^{2\theta-1}} \mathbf{1}_{a\in \cA,g\in \cG} .
\end{equation}
By~\eqref{SLKD}, it follows that
\begin{equation}\label{LTXS}
    \|w\|_\infty\coloneqq \max_{a\in\cE^+} w(a,\sigma(a)) \leq \gamma n^{-(1+\theta)}
    .
\end{equation}
Using the one-sided version of Chatterjee's inequality for uniformly random pairings~\cite[Proposition~1.1]{chatterjee2007}, we get that
\begin{align*}
\p{\hat{P}^{\tau(f)}(e,f)\leq (1-c)\E{\hat{P}^{\tau(f)}(e,f)\mid \sigma_0}\mid \sigma_0}\leq
\exp(-\frac{c^2\E{\hat{P}^{\tau(f)}(e,f)\mid \sigma_0}}{6\|w\|_\infty})\leq
e^{-\frac{c^2n^{\theta}}{3}} ,
\end{align*}
where we used~\eqref{eq:FKEN} in the last line.
\end{proof}

\subsubsection{Typical partial pairings}\label{sec:typ}
In this part we prove that whp the events $\cX_{e,f}$ simultaneously hold for all pairs $e,f\in \cE^-$ provided that $f\in \cE_0^-$, \new{or equivalently, that $\{t^-_\omega(f)<\infty\}$}. 

For this purpose, define the event
    \begin{equation}\label{JIZC}
        \cX=\bigcap_{e,f\in \cE^-} (\cX_{e,f}\cup \{t^-_\omega(f)=\infty\}).
    \end{equation}

\begin{lemma}\label{lem:whp} 
    We have
    \(\p{\cX}= 1-o(1).\)
\end{lemma}

In fact, we prove something slightly stronger that implies the lemma.}
Define the following events
\begin{equation}\label{KLCC}
    \cX_1 = \bigcap_{e,f\in \cE^-} 
    \left\{
        A_{e,f}\geq \frac{1}{2}
    \right\}
    ,
    \qquad
    \cX_2 = \bigcap_{e,f\in \cE^-} \Big(
    \left\{G_{e,f}\geq \frac{\omega \gamma}{4} \right\}\cup \{t^-_\omega(f)=\infty\}\Big)
    ,
\end{equation}
and note that $\cX\supseteq \cX_1\cap \cX_2$.

The vertex analogue of the event $\cX_1$ was shown to hold whp in~\cite[Lemma 3.7]{caputo2020a}. Its proof does not use any assumption on $\delta^-$ and is valid for tail-trees instead of vertex-trees.  Thus the conclusion
$\p{\cX_1}=1-o(1)$ still holds in our setting.

To prove~\autoref{lem:whp}, we are left with showing the following lemma, for which we apply some of the results derived in~\autoref{sec:branching}.
\begin{lemma}\label{lem:A2}
We have $\p{\cX_2}=1-o(1)$.
\end{lemma}
\begin{proof}
    In order to compute the probability of
    \(\cX_{2}\), it will be convenient to swap the order of the phases: we first run the in-phase
    unconditionally, and then the out-phase.  Write $h\coloneqq h(f)+h^-$. Define
\begin{align}\label{EKDJ}
\hat{\Gamma}^{-}_h(f) \coloneqq \sum_{g\in \cN^-_h(f)}\hat{\bfw}(g),
\end{align}
\new{where $\hat{\bfw}(g)$ is defined as in~\eqref{FKDN}.}

The difference between \(\hat{\Gamma}^{-}_{h}(f)\)
and \(G_{e,f}\), defined in~\eqref{QHKO},  is that the latter does not include the weight of heads
that are paired in the out-phase, so $\hat{\Gamma}^{-}_{h}(f)\geq G_{e,f}$. Our strategy consists on showing that (1) after the in-phase, \(\hat{\Gamma}^{-}_{h}(f)\) is large, and (2) after the out-phase, \(G_{e,f}\) is well-approximated by \(\hat{\Gamma}^{-}_{h}(f)\).

Let us proceed by showing (1). Consider the  distribution $\eta^{\ua}$, where  \(\eta=\dout\); \new{see~\eqref{VMLKM}}.   Recall the definition of \(\hnu^{\ua}\), \(\hat{H}^{\ua}\), \(\phi^{\ua}(a)\) and \(a_{0}^{\ua}\)  in~\autoref{rem:cont_parameters}.
Define
$\gamma^{\ua}\coloneqq n^{-{(1+\theta/2)\hat{H}^{\ua}}/{\phi^{\ua}(a_{0}^{\ua})}}$, for $\theta>0$ chosen previously.

\new{Consider the events 
\begin{equation}\label{QPVX}
E^
f_1\coloneqq \{\hat{\Gamma}^{-}_h(f)<\omega \gamma/2\}\text{ and }E^
f_2\coloneqq \{t^-_\omega(f)\leq\omega\},
\end{equation}
defined in the probability space of the directed configuration model $\vecGn$.}

Consider a marked branching process $(X^\ua_t)_{t\geq 0}$ with distribution $\eta^\ua$ and recall the definitions of $t_{\omega}$ in~\eqref{SOEP} and of $\hat{\Gamma}_t$
in~\eqref{XYYI}. Consider the events  
\begin{equation}\label{SPVX}
E_{1}^{\ua} \coloneqq
\{\hat{\Gamma}_{t_\omega+h^-}(t_\omega,\gamma^{\ua})<\omega\gamma^{\ua}/2\} \text{ and }E_{2}^\ua \coloneqq \{t_{\omega} \leq \omega\},
\end{equation}
\new{defined in the probability space of the marked branching process. We will abuse notation and use $\p{\cdot}$ for the probability in both probability spaces.}

\new{Recall that $\eta^\ua$ satisfies \autoref{cond:BP}.} We apply~\autoref{cor:DWPD} to $(X^\ua_t)_{t\geq 0}$ with \(p = n^{-(1+\theta/\new{2})}\) and $\varepsilon$ small enough with respect to $\theta$, to conclude that
\begin{equation}\label{NSOK}
    \p{(\cB_{t_\omega}(\gamma^{\ua}))^c \cap E_{2}^\ua} 
    \le
    \p{(\cB_{t_\omega}(\gamma^{\ua}))^c \cap \{t_\omega<\infty\}} 
    = 
    O(n^{-(1+\theta/4)}),
\end{equation}
\new{where in the application of the corollary we have} absorbed the term $\omega^C$ inside the polynomial part, as we have chosen $\omega$ to be poly-logarithmic in~\eqref{eq:omega}.

\new{We apply~\autoref{prop:martingale} to $(X^\ua_t)_{t\geq 0}$ taking $\ell=h^-=O(\log{n})$ and $\delta=\log^{-2}{n}$, which satisfy $(1-2\delta M)^\ell\geq 1/2$ for sufficiently large $n$. By some standard computations and~\eqref{SDSO}, the hypotheses of the proposition are satisfied.} We obtain
\begin{equation}\label{HZRX}
    \begin{aligned}
        \p{E_1^{\ua}\cap  E_2^\ua}
        &
        \leq 
        \p{E_1^{\ua} \cap E_{2}^\ua \cap \cB_{t_{\omega}}(\gamma^{\ua})}
        +
        \p{(\cB_{t_\omega}(\gamma^{\ua}))^c\cap E_{2}}
        \\
        &
        \le
        \sum_{t_0=1}^{\omega} 
        \p{\hat\Gamma_{t_0+h^-}(t_0,\gamma^\ua) \mid \{t_{\omega }=t_0\} \cap \cB_{t_0}(\gamma^{\ua})}
        \p{t_{\omega }=t_0 \mid t_{\omega }\leq\omega}
        +
        O(n^{-(1+\theta/4)})
        \\
        &
        \leq  
        2\ell e^{\delta^2 \omega/4}
        +
        O(n^{-(1+\theta/4)})
        = 
        O(n^{-(1+\theta/4)})
        ,
    \end{aligned}
\end{equation}
\new{where we used that $\omega=\log^6 n$ in the last step.}

We now transfer the probability from branching processes to the directed configuration model using the coupling established in~\autoref{sec:coupling}.
\new{Recall that \((T^-_{f}(i))_{i\geq 0}\) is the sequence of marked trees rooted at $f$
constructed by the exploration process; see \autoref{sec:explore}.
Let $\cT$ be the set of marked
feasible trees $T$ of height at most \(\omega+h^{-}\) such that if $\{T^-_{f}(p(T)) \cong T\}$ holds, then $\vecGn$ satisfies $E^
f_1 \cap E^
f_2$, where \(p(T)\) was defined as the number of paired nodes in \(T\).}
Note that $p(T)
\leq \kappa^- = O(n^{\alpha})$ for any constant \(\alpha>0\), \new{provided that $\theta$ is small}; see~\eqref{TGWY}. Since we  chose $\alpha>0$ sufficiently small, by~\autoref{lem:eq_T_GW} with \(\beta = \alpha\), we have
\begin{align}\label{XRVP}
        \p{E^
f_1\cap E^
f_2}
        &
        = 
            \p{T_{f}^-\in \cT}
        \le
            (1+o(1))
            \p{\GW_{\eta^{\ua}}\in \cT}
\end{align}
Under $E_2^\ua$, $h=t_\omega+h^-$. Thus, $\hat{\Gamma}^{-}_h(f),t^-_\omega(f)$ are the graph analogues of $\hat{\Gamma}_{t_\omega+h^-},t_{\omega}$ and $\{\GW_{\eta^{\ua}}\in \cT\}=E_1^\ua\cap E_2^\ua$. \new{(Note that $E_1^\ua$ can also be written as $\{\hat{\Gamma}_{t_\omega+h^-}(t_\omega,\gamma)<\omega\gamma/2\}$, as the event is invariant with respect to the truncating constant.)}
It follows from~\eqref{HZRX} that
\begin{align}\label{XRVP2}
        \p{E^
f_1\cap E^
f_2}
        &
        =O(n^{-(1+\theta/4)})
        .
\end{align}
By~\autoref{lem:DSKD}, we have \(\p{(E_2^\ua)^c\mid t_\omega<\infty} \leq \p{0<X_{\omega}<\omega} = o(n^{-2})\). Applying
\autoref{lem:eq_T_GW} similarly as before yields \(\p{(E^
f_2)^c\mid
t_\omega^-(f)<\infty}=o(n^{-2})\). Therefore,  it follows from~\eqref{XRVP2} that
\begin{equation}\label{WQCC}
  \p{E^
f_1\cap \{t_\omega^-(f) < \infty\}}\leq \p{E^
f_1\cap E^
f_2}+\p{ (E^
f_2)^c\mid t_\omega^-(f) < \infty}= O(n^{-(1+\theta/4)})  
  .
\end{equation}
Thus, by applying the union bound first over \(f \in \cE^{-}\) and then
over \(e \in \cE^{-}\), we have
\begin{equation}\label{WLMU}
\begin{aligned}
    \p{\cX_2^c}
 &= \p{\bigcup_{e,f\in \cE^-} \Big(
    \left\{G_{e,f}< \frac{\omega \gamma}{4} \right\}\cap \{t_\omega^-(f) < \infty\}\Big)}  \\  
    &
    \leq \sum_{f\in \cE^-} \p{E_1^f\cap \{t_\omega^-(f) < \infty\}} + \sum_{e,f\in \cE^-}  \p{
    \left\{G_{e,f}< \frac{\omega \gamma}{4}\right\}\cap (E_1^f)^c} \\  
    &
    \leq
    o(1)
    +
    \sum _{e, f\in \cE^{-}}
    \p{
        \left\{G_{e,f}<\frac{\omega \gamma}{4} \right\}
        \cap
        \left\{ \hat{\Gamma }^{-}_h(f)  \geq \frac{\omega \gamma }{2} \right\}
    }
    .
    \end{aligned}
\end{equation}

We will control the terms in the previous sum by showing (2), that \(G_{e,f}\) is well-approximated by \(\hat{\Gamma}^{-}_{h}(f)\). Let $\sigma^-$ be a partial pairing of the at most \(\kappa^{-}=o(n)\) half-edges that were paired in the in-phase.
We perform the out-phase to construct the tree $T^+_{e}$ conditioned on $\sigma^-$.  Recall that during the out-phase at most $\kappa^{+}= o(n)$ edges are formed.  Thus, 
\begin{equation}\label{NFNZ}
    \begin{aligned}
        \E{G_{e,f} \mid \sigma^-}
    &
    = 
    \sum_{g\in \cN^-_{h}(f)} 
    \hat{\bfw}(g)
    \E{
        \ind{g \in \cG}
        \mid \sigma^-
    }
    \\
    &
    \ge
    \left( 1 - \frac{\kappa^{+}}{m - \kappa^{+} - \kappa^{-}} \right)
    \sum_{g\in \cN^-_{h}(f)} 
    \hat{\bfw}(g)
    \\
    &
    =
    (1+o(1))
    \hat{\Gamma}^{-}_{h}(f)
    .
    \end{aligned}
\end{equation}

An application of Azuma's inequality (see~\cite[pp.~92]{molloy2002a}) to $G_{e,f}$, which is determined by the random vector
$\left(\hatw(g) \ind{g \in \cG}\right)_{g\in \cN^-_h(f)}$, implies that
\begin{equation}\label{XNJR}
    \p{G_{e, f} <
    \frac{\hat{\Gamma}^{-}_{h}(f)}{2} \mid \sigma^-}
    \le
    \exp\left\{
        -2
        \frac{(1+o(1)\hat{\Gamma}^{-}_{h}(f)^{2}}{
                \sum_{g\in \cN^-_{h}(f)} 
                \hat{\bfw}(g)^{2}
            }
        \right\}
        \le
        \exp\left\{
            -\gamma^{-1} n^{3\theta} \hat{\Gamma}^{-}_{h}(f)
        \right\}
        ,
    \end{equation}
 where we used that by~\eqref{SLKD} and~\eqref{EKDJ},
 \begin{align}
 \sum_{g\in \cN^-_{h}(f)} \hat{\bfw}(g)^{2}\leq \hat{\Gamma}^{-}_{h}(f) \max_{g\in \cN^-_{h}(f)} \hat{\bfw}(g)\leq \gamma n^{-3\theta} \hat{\Gamma}^{-}_{h}(f). 
 \end{align}
Since the event \(\{\hat{\Gamma }^{-}_h(f)>\frac{\omega \gamma }{2}\}\) is measurable with respect to
\(\sigma^-\), we have
\begin{equation}\label{NLQP}
    \p{
        \left\{G_{e,f}<\frac{\omega \gamma}{4} \right\}
        \cap
        \left\{\hat{\Gamma }^{-}_h(f)>\frac{\omega \gamma }{2} \right\}
    }
    \le
    \exp\left\{
        -\gamma^{-1} n^{3 \theta} \left( \frac{\omega \gamma}{2} \right)
    \right\}
    \le
    \exp\left\{
        - \omega n^{3 \theta}
    \right\}
    .
\end{equation}
Putting this into~\eqref{WLMU} finishes the proof.
\end{proof}

\subsubsection{Proof of~\autoref{lem:PT}}\label{sec:LB_proof}

    Write 
\begin{equation}
E_1^{e,f}\coloneqq \left\{\hat{P}^{\tau(f)}(e,f)\geq  \frac{\gamma}{n}\right\}
,
\qquad
E^f_2\coloneqq \{t_\omega^-(f)=\infty\}
.
\end{equation}
Using~\eqref{eq:FKEN} and applying~\autoref{lem:concentration} with $c=1/2$,
\begin{equation}
\p{(E_1^{e,f})^c\mid \cX_{e,f}}
=
\p{\hat{P}^{\tau(f)}(e,f)<  \frac{\gamma}{n}\mid \cX_{e,f}}
\leq \max_{\sigma_0\in \cX_{e,f}}\p{\hat{P}^{\tau(f)}(e,f)<  \frac{\gamma}{n}\mid \sigma_0}
= o(n^{-2}).
\end{equation}
\new{Recall the definition of $\cX$ in \eqref{JIZC}.} It follows that,
\begin{align*}
\p{\cap_{e,f\in \cE^-} (E_1^{e,f}\cup E^f_2)} 
&\geq \p{\cX\cap (\cap_{e,f\in \cE^-} (E_1^{e,f}\cup E^f_2))}\\
&= \p{\cX} - \p{\cX\cap \left(\cap_{e,f\in \cE^-}(E_1^{e,f}\cup E^f_2)\right)^c}\\
&\geq (1-o(1)) - \sum_{e,f\in \cE^-} \p{\cX\cap (E_1^{e,f})^c \cap (E^f_2)^c}\\
&\geq (1-o(1)) - \sum_{e,f\in \cE^-} \p{(E_1^{e,f})^c\cap \cX_{e,f}}\\
&\geq (1-o(1)) - \sum_{e,f\in \cE^-} \p{(E_1^{e,f})^c\mid \cX_{e,f}}\\
&=1-o(1)
,
\end{align*}
where we used that $\cX\cap (E^f_2)^c$ implies $\cX_{e,f}$. 

So, whp, if $t^-_\omega(f)<\infty$, then $E_1^{e,f}$ holds. In other words, we have
shown that for all \(e \in \cE^{-}\) and \(f \in \cE_{0}^{-}\), 
\begin{equation}\label{FNLD}
    P^{\tau(f)}(e,f) \ge \hat{P}^{\tau(f)}(e,f)\ge\frac{\gamma}{n} \geq n^{-(1+(1+\theta)\hat{H}^-/\phi(a_{0}))}.
\end{equation}
Choosing \(\theta=\varepsilon \phi(a_0)/\hat{H}^{-}\), we conclude the proof of~\autoref{lem:PT}.

\subsection{Upper bound for \texorpdfstring{$\piminez$}{Pi max zero}}
\label{sec:upper}

Let $\theta>0$ be small enough. \new{Recall that $\hnu^->0$ and }let $a_{0}$ be the value that minimises $\phi(a)$, defined as in~\eqref{ODRS}. Recall the definitions \new{of $\nupm$ in~\eqref{SPEK}} and of $\omega$ in~\eqref{eq:omega}. Let 
\begin{equation}\label{JOAY}
    h_1\coloneqq \frac{1-\theta}{a_{0}\phi(a_{0})}\log{n}, 
    \qquad 
    h_2\coloneqq  \log_\nupm (2\omega) = O(\log\log n) . 
\end{equation}
In this section we set 
\begin{align}\label{SPEO}
\gamma\coloneqq n^{-{(1-\new{3\theta/2})\hat{H}^-}/{\phi(a_{0})}}.
\end{align}
\new{Fix $f\in \cE^-$ and let $T_f^-$ be the tree revealed by the exploration process described in~\autoref{sec:explore}, stopping
once all heads at distance \(h_1+h_2\) to \(f\) have become active.}
During the construction of $T^-_f$ we deterministically expose at most $\kappa^-_{1} \coloneqq \omega M h_1+\omega
\new{\sum_{i=1}^{h_2}M^{i+1}}= O(\log^C n)$ pairings for some constant $C>0$.

Recall the definition of $\mathbf{w}(g)$ in~\eqref{TQIN} \new{as the weight of the head $g$ in $T_f^-$}.
Define
\begin{align}
    P^-_t(f) &\coloneqq \sum_{g \in  \cN_{t}^-(f)} P^{t}(g, f),
\label{NRFS2} \\
\Gamma^{-}_{t}(f) &\coloneqq  \sum_{g\in \cN^-_{t}(f)} \bfw(g)
.\label{NRFS}
\end{align}
\new{Both parameters are related to the probability of reaching $f$ from $\cN_{t}^-(f)$ in $t$ steps.} However, while $P^-_t(f)$ takes into consideration all the contributions of moving to $f$, $\Gamma^-_t(f)$ only considers the contribution of moving to $f$ using edges in $T_f^-$. Therefore, $P^-_t(f)\geq \Gamma^-_t(f)$ with equality if $\cN_{\leq t}^-(f)$ induces a tree.

For any \new{connected} subgraph $G$ with $V(G)\subseteq [n]$, let $\TX(G) \coloneqq |E(G)|-(|V(G)|-1)$ denote the \emph{excess} of $G$, that is the number of additional edges in $G$ with respect to an spanning tree of it. 

Consider the following events (see~\autoref{fig:events} for an example):
\begin{equation}\label{EQRP}
    \begin{aligned}
        E^
f_0&=\{t^-_\omega(f) \le h_1+h_2\}
        ,
        &
        &
        E^
f_1=\{0<\Gamma^{-}_{h_1}(f)< \gamma\}
        ,
        \\
        E^
f_2&=\bigcap_{r=1}^{h_1} \{0<|\cN^-_r(f)|< \omega\}
        ,
        &
        &
        E^
f_3=\{\TX(\cN^-_{\leq h_1}(f))=0\} \text{ and}
        \\
        \cY_f&=E^
f_0\cap E^
f_1\cap E^
f_2\cap E^
f_3.
    \end{aligned}
\end{equation}

\begin{figure}[ht]
\centering
			\begin{tikzpicture}[scale=0.6]

			\node[circle,draw,fill=blue,inner sep=0pt,minimum size=8pt] (V0)	at (18,0) {};	
			\node[circle,draw,fill=blue,inner sep=0pt,minimum size=8pt] (V1)	at (15,0) {};	
			\node[circle,draw,fill=blue,inner sep=0pt,minimum size=8pt] (V2)	at (12,0) {};	
			\node[circle,draw,fill=blue,inner sep=0pt,minimum size=8pt] (V3)	at (9,0) {};	
			\node[circle,draw,fill=black,inner sep=0pt,minimum size=8pt] (V4)	at (6,0) {};	
			\node[circle,draw,fill=black,inner sep=0pt,minimum size=8pt] (V5)	at (3,0) {};	
			\node[circle,draw,fill=black,inner sep=0pt,minimum size=8pt] (V6)	at (0,0) {};	
			\node[circle,draw,fill=black,inner sep=0pt,minimum size=8pt] (V7)	at (-3,0) {};	
			
			\draw[-Latex,blue] (V1) --  (V0);
			\draw[-Latex,blue] (V2) --  (V1);
			\draw[-Latex,blue] (V3) --  (V2);
			\draw[-Latex] (V4) --  (V3);
			\draw[-Latex] (V5) --  (V4);
			\draw[-Latex] (V6) --  (V5);
			
			\node[circle,fill=blue,inner sep=0pt,minimum size=8pt] (V11) at (15,4) {};	
			\draw[-Latex,blue] plot [smooth, tension=1] coordinates { (15.2,4) (16.5,3)  (17.9,0.2)};;
			
			\node[circle,fill=blue, inner sep=0pt,minimum size=8pt] (V21) at (12,4) {};	
			\node[circle,draw,fill=blue, inner sep=0pt,minimum size=8pt] (V22) at (12,3) {};	
			\node[circle,fill=blue, inner sep=0pt,minimum size=8pt] (V23) at (12,-2) {};	
			
			\draw[-Latex,blue] (V21) --  (V11);
			\draw[-Latex,blue] plot [smooth, tension=1] coordinates { (12.2,3) (13.5,2.3)  (14.9,0.2)};;
			\draw[-Latex,blue] plot [smooth, tension=1] coordinates { (12.2,-2) (13.5,-1.5)  (14.9,-0.2)};;

			\node[circle,draw,fill=blue, inner sep=0pt,minimum size=8pt] (V31)	at (9,3) {};	
			\node[circle,draw,fill=blue, inner sep=0pt,minimum size=8pt] (V32)	at (9,-3) {};	
			
			\draw[-Latex,blue] (V31) --  (V22);
			\draw[-Latex,blue] plot [smooth, tension=1] coordinates { (9.2,-3) (10.5,-2.3)  (11.9,-0.2)};
			
			\node[circle,draw, very thick,fill=black, inner sep=0pt,minimum size=8pt] (V41)	at (6,3) {};
			\node[circle,fill=black, inner sep=0pt,minimum size=8pt] (V42)	at (6,1) {};		
			\node[circle,fill=black, inner sep=0pt,minimum size=8pt] (V43)	at (6,-2) {};		
			\node[circle,draw,fill=black, inner sep=0pt,minimum size=8pt] (V44)	at (6,-3) {};
			
			\draw[-Latex] (V41) --  (V31);
			\draw[-Latex] (V44) --  (V32);
			\draw[-Latex] plot [smooth, tension=1] coordinates { (6.2,1) (8,0.75)  (8.9,0.2)};
			\draw[-Latex] plot [smooth, tension=1] coordinates { (6.2,-2) (8,-2.25)  (8.9,-2.8)};

			\node[circle,fill=black, inner sep=0pt,minimum size=8pt] (V51)	at (3,4) {};		
			\node[circle,draw,fill=black, inner sep=0pt,minimum size=8pt] (V52)	at (3,3) {};	
			\node[circle,draw,fill=black, inner sep=0pt,minimum size=8pt] (V53)	at (3,2) {};	
			\node[circle,fill=black, inner sep=0pt,minimum size=8pt] (V54)	at (3,-1) {};		
			\node[circle,draw,fill=black, inner sep=0pt,minimum size=8pt] (V55)	at (3,-3) {};	
			
			\draw[-Latex] (V52) --  (V41);
			\draw[-Latex] (V55) --  (V44);
			\draw[-Latex] plot [smooth, tension=1] coordinates { (3.2,-1) (5,-0.75)  (5.9,-0.2)};
			\draw[-Latex] plot [smooth, tension=1] coordinates { (3.2,2) (5,2.25)  (5.9,2.8)};
			\draw[-Latex] plot [smooth, tension=1] coordinates { (3.2,4) (5,3.75)  (5.9,3.2)};
			\draw[-Latex] plot [smooth, tension=1] coordinates { (0.2,-0.1) (1,-0.75)  (2.8,-1)};

			\node[circle,draw,fill=black, inner sep=0pt,minimum size=8pt] (V6e)	at (0,5) {};
			\node[circle,draw,fill=black, inner sep=0pt,minimum size=8pt] (V60)	at (0,4) {};	
			\node[circle,draw,fill=black, inner sep=0pt,minimum size=8pt] (V61)	at (0,3) {};	
			\node[circle,draw,fill=black, inner sep=0pt,minimum size=8pt] (V62)	at (0,2) {};	
			\node[circle,draw,fill=black, inner sep=0pt,minimum size=8pt] (V63)	at (0,1) {};	
			\node[circle,draw,fill=black, inner sep=0pt,minimum size=8pt] (V65)	at (0,-3) {};	
			\node[circle,draw,fill=black, inner sep=0pt,minimum size=8pt] (V66)	at (0,-4) {};				
			
			\draw[-Latex] (V60) --  (V51);
			\draw[-Latex] (V61) --  (V52);
			\draw[-Latex] (V62) --  (V63);
			\draw[-Latex] (V62) --  (V53);
			\draw[-Latex] (V61) --  (V52);
			\draw[-Latex] (V65) --  (V55);
			
			\draw[-Latex] plot [smooth, tension=1] coordinates { (0.2,-4) (2,-3.75)  (2.9,-3.2)};
			\draw[-Latex] plot [smooth, tension=1] coordinates { (0.2,1) (2,.75)  (2.9,0.2)};

			\node[circle,draw,fill=black, inner sep=0pt,minimum size=8pt] (V70)	at (-3,5) {};	
			\node[circle,draw,fill=black, inner sep=0pt,minimum size=8pt] (V71)	at (-3,4) {};	
			\node[circle,draw,fill=black, inner sep=0pt,minimum size=8pt] (V72)	at (-3,3) {};	
			\node[circle,draw,fill=black, inner sep=0pt,minimum size=8pt] (V73)	at (-3,2) {};	
			\node[circle,draw,fill=black, inner sep=0pt,minimum size=8pt] (V74)	at (-3,-1) {};	
			\node[circle,draw,fill=black, inner sep=0pt,minimum size=8pt] (V75)	at (-3,-3) {};	
			\node[circle,draw,fill=black, inner sep=0pt,minimum size=8pt] (V76)	at (-3,-4) {};				\node[circle,draw,fill=black, inner sep=0pt,minimum size=8pt] (V77)	at (-3,-5) {};

			\draw[-Latex] (V71) --  (V60);
			\draw[-Latex] (V72) --  (V61);
			\draw[-Latex] (V73) --  (V62);
			\draw[-Latex] (V7) --  (V6);
			\draw[-Latex] (V75) --  (V65);
			\draw[-Latex] (V76) --  (V66);
			
			\draw[-Latex] plot [smooth, tension=1] coordinates { (-2.8,-5) (-1,-4.75)  (-0.1,-4.2)};
			\draw[-Latex] plot [smooth, tension=1] coordinates { (-2.8,5) (-1,4.75)  (-0.1,4.2)};
			\draw[-Latex] plot [smooth, tension=1] coordinates { (0.2,5) (2,4.75)  (2.9,4.2)};
			\draw[-Latex] plot [smooth, tension=1] coordinates { (-2.8,-1) (-1,-.75)  (-0.1,-0.2)};

			\node at (18.7,0) {$f$};
			\node at (9,5.7) {$h_1$};
			\node at (-3,5.7) {$h_1+h_2$};

			\draw[dashed] (-4,4.5) -- (19,4.5);	
			\draw[dashed] (-4,-3.5) -- (19,-3.5);

			\draw[-Latex,blue] (V1) -- ++(0.9,0.3);
			\draw[-Latex,blue] (V1) -- ++(0.7,0.7);
			\draw[-Latex,blue] (V1) -- ++(0.7,-0.7);
			
			\draw[-Latex,blue] (V11) -- ++(1,0);
			\draw[-Latex,blue] (V11) -- ++(0.9,0.3);
			
			\draw[-Latex,blue] (V2) -- ++(0.9,0.3);
			\draw[-Latex,blue] (V2) -- ++(0.9,-0.3);
			\draw[-Latex,blue] (V2) -- ++(0.7,0.7);
			\draw[-Latex,blue] (V2) -- ++(0.7,-0.7);

			\draw[-Latex,blue] (V21) -- ++(0.9,0.3);	
			
			\draw[-Latex,blue] (V22) -- ++(1,0);
			\draw[-Latex,blue] (V22) -- ++(0.9,0.3);
			\draw[-Latex,blue] (V22) -- ++(0.7,0.7);
			\draw[-Latex,blue] (V22) -- ++(0.7,-0.7);
						
			\draw[-Latex,blue] (V23) -- ++(0.9,0.3);
			\draw[-Latex,blue] (V23) -- ++(0.9,-0.3);
			\draw[-Latex,blue] (V23) -- ++(0.7,0.7);
			\draw[-Latex,blue] (V23) -- ++(0.7,-0.7);

			\draw[-Latex,blue] (V3) -- ++(0.9,0.3);
			\draw[-Latex,blue] (V3) -- ++(0.7,0.7);
			\draw[-Latex,blue] (V3) -- ++(0.7,-0.7);
			\draw[-Latex,blue] (V3) -- ++(0.9,-0.3);
			\draw[-Latex,blue] (V3) -- ++(0.3,0.9);

			\draw[-Latex,blue] (V31) -- ++(0.9,0.3);
			\draw[-Latex,blue] (V31) -- ++(0.7,0.7);
			\draw[-Latex,blue] (V31) -- ++(0.9,-0.3);

			\draw[-Latex,blue] (V32) -- ++(0.8,0.5);
			\draw[-Latex,blue] (V32) -- ++(0.9,-0.3);

			\end{tikzpicture}

\caption{Schematic drawing of the event $\cY_f$ for $h_1=3$, $h_2=4$ and $\omega=7$. In it, $\cN^-_{\leq h_1+h_2}(f)$ is depicted \new{and  $\cN^-_{\leq h_1}(f)$ is coloured in blue; we only add the out-degrees in the latter as the other ones are irrelevant for $\cY_f$}. Note that $t_{\omega}^-(f)= 6$. Events $E^
f_0$ and $E^
f_2$ hold as $h_1<t_{\omega}^-(f)\leq h_1+h_2$. \new{Using the out-degrees of the vertices in $\cN^-_{\leq h_1}(f)$, one can compute the weight of each of the three heads at distance $h_1$ of $f$ to obtain $\Gamma_{h_1}^-(f)= \tfrac{1}{4\cdot 5 \cdot 4} + \tfrac{1}{6\cdot 5 \cdot 4} + \tfrac{1}{3\cdot 5 \cdot 4} = \tfrac{3}{80}$. 
So event $E^f_1$ will hold provided that $\gamma>\tfrac{3}{80}$}. Event $E^f_3$ holds as $\cN^-_{\leq h_1}(f)$ induces a tree.}
\label{fig:events}
\end{figure}

Next proposition estimates the probability that a head $f$ satisfies the desired properties.
\begin{prop}\label{prop:SKFD}
    Uniformly for all $f\in \cE^-$, we have
        \begin{equation}\label{RHEN}
    n^{-1+\theta/2}\leq \p{\cY_f}\leq  n^{-1+2\theta}
    .
    \end{equation}
\end{prop}
\begin{proof}

Let $(X_r)_{r\geq 0}$ be a marked branching
process with distribution $\eta= \dout$.
Consider the events
\begin{equation}\label{VNSW}
E_0 =\{t_{\omega} \leq  h_{1}+h_{2}\},\; E_1=\{0<\Gamma_{h_1}< \gamma\} \text{ and } E_2= \bigcap_{r=1}^{h_1} \{0<X_r<\omega\},
\end{equation}
\new{where $t_{\omega}$ and $\Gamma_{h_1}$ are as in \eqref{SOEP} and~\eqref{OSKE} respectively. These events can be seen as the analogues of $E_0^f$, $E_1^f$ and $E_2^f$ in the branching process setting.}

Choose $\delta>0$ sufficiently small with respect to $\theta$ \new{so $2\delta\leq \theta_0 a_0 \hat{H}^-$ and $\delta< \theta a_0\phi(a_0)$}. By \autoref{thm:LB} with
$a=a_{0}$ and $t=h_1$, \new{there exists $c>0$ such that}
\begin{equation}\label{XXEI}
    \p{E_1\cap E_2} \geq \new{c} e^{-(a_0\phi(a_0)+\delta)	h_1} \geq  n^{-1+3\theta/4}.
\end{equation}
When \(E_2\) happens, there is at least one individual in generation \(h_{1}\). Thus, there exists a constant \(c_0>0\) only depending on $M$ such that
\begin{equation}\label{JDPA}
    \p{E_{0} \mid E_{1} \cap E_2}
    \geq
    \p{X_{h_{2} + h_{1}} \geq \omega  \mid X_{h_{1}} \new{= 1}}\new{=} \p{X_{h_2\new{-1}}\geq \omega}>c_0,
\end{equation}
where we used~\eqref{NRIW} and $h_2= \log_\nupm (2\omega)$. 
Therefore,
\(\p{E_{0} \cap E_{1} \cap E_2} \ge c_0 n^{-1+3\theta/4}\).

Let $\cT$
be the set of feasible marked trees $T$ of height at most $h_1+h_2$ such that if $\{T^-_{f}(p(T))\cong T\}$ then $\vecGn$ satisfies $E^
f_0\cap E^
f_1\cap E^
f_2$.
%
 \new{Recall that $p(T)\leq \kappa^-_1= O(\log^C n)$ for any $T\in \cT$.
Using~\autoref{lem:eq_T_GW} with $\beta>0$ arbitrarily small,
we have}
\begin{equation}
    \p{E^
f_0 \cap E^
f_1\cap E^
f_2}
    \geq 
    (1+o(1))
    \p{\GW_{\eta} \in \cT} + \new{O(n^{2\beta-3/2})}
    \geq n^{-1+2\theta/3}
    .
\end{equation}


We next bound the probability of \( E^
f_2\cap (E^
f_3)^{c} \) from above. Observe that $\TX(\cN_{\leq
h_1}^-(f))$ is the number of collisions in the exploration process of $T^-_f$, that is, steps where $e_i^{+}\in
\cA^+_{i-1}$; see~\autoref{sec:explore}. At step $i$, the probability of a collision is at most $\frac{M i}{m-i+1}$. So,
provided $\ell\leq m/2$, the number of collisions in the first $\ell$ steps is dominated by a
binomial random variable with $\ell$ independent trials of probability $2\ell M/m$. Note that the probability of
\(E^
f_2\cap (E^
f_3)^{c}\) is at most the probability of at least one collision
happening when running the exploration process for at most $\ell=(\omega-1)h_1=O(\log^7
n)$ steps. Thus we have
\begin{align}
\p{E^
f_2\cap (E^
f_3)^{c}}=O(\ell^2/m)= n^{-1+o(1)}.
\end{align}
Thus, we obtain the desired lower bound:
\begin{align*}
    \p{\cY_f} &= \p{E^
f_0 \cap E^
f_1\cap E^
f_2\cap E^
f_3} \\ 
    &\geq \p{E^
f_0 \cap E^
f_1\cap E^
f_2} -\p{E^
f_2\cap (E^
f_3)^{c}}
\geq  n^{-1+\theta/2}.
\end{align*}	
For the second part of the proposition, the upper bound, we consider a marked branching process $(X^\ua_t)_{t\geq 0}$ with distribution $\eta^\ua$, \new{and the events $E_1$ and $E_2$ defined as in~\eqref{VNSW}. Let $\gamma^{\ua}\coloneqq n^{-(1-2\theta)\hat{H}^{\ua}/\phi^{\ua}(a_0^{\ua})}$. By \autoref{rem:cont_parameters}, the parameters $\hat{H}^{\ua},\phi^{\ua},a_0^{\ua}$ are well-approximated by $\hat{H},\phi,a_0$, so $\gamma^{\ua}\geq \gamma$ provided that $n$ is sufficiently large. Also, the remark implies that $\eta^\ua$ satisfies \autoref{cond:BP}.}
An application of~\autoref{thm:UB} with $\delta$ sufficiently small with respect to $\theta$, $a=a^\ua_{0}$, and $t=h_1$, and of~\autoref{lem:eq_T_GW} \new{with $\beta>0$ arbitrarily small}, gives
that 
\begin{equation}\label{XFDZ}
    \begin{aligned}
    \p{\cY_f}
    &
    \le
    \p{E^
f_1\cap E^
f_2}
    \\
    &\leq  
    (1+o(1))
    \p{E_1\cap E_2}
  	\\
    &\le
        \p{(\cB_{t}(\gamma))^{c} \cap \{0 < X_{t} < \omega\}}
        \\    &\new{\le
        \p{(\cB_{t}(\gamma^{\ua}))^{c} \cap \{0 < X_{t} < \omega\}}}
        \\
    &\le
    n^{-1+2\theta}
    .
    \end{aligned}
\end{equation}
    \qedhere
\end{proof}

We will show that $Z\coloneqq \sum_{f\in \cE^-} \mathbf{1}_{\cY_f}$ satisfies $\{Z>0\}$ whp, using a second moment calculation. As $m\geq n$,~\autoref{prop:SKFD} implies that $\E{Z}\geq n^{\theta/2}$. 
For a partial pairing $\sigma_0$ and following the notation introduced in~\autoref{sec:explore}, we write $f\in \cP^-(\sigma_0)$ if $f$ has been paired by $\sigma_0$. 
For $f_1,f_2\in \cE^-$ with $f_1\neq f_2$, we have
\begin{align}\label{BFGH}
\p{\cY_{f_1}\cap \cY_{f_2}}
&\leq 
\sum_{\sigma_0\in \cY_{f_1}}\p{\sigma_0}\left(\mathbf{1}_{f_2\in \cP^-(\sigma_0)}+\mathbf{1}_{f_2\notin \cP^-(\sigma_0)}\p{\cY_{f_2}\mid \sigma_0}\right)\\
&\leq \p{\cY_{f_1}}\left(O\left(\frac{\kappa_1^-}{n}\right)+ \max_{\substack{\sigma_0\in \cY_{f_1} \\ f_2\notin \cP^-(\sigma_0)}}\p{\cY_{f_2}\mid \sigma_0}\right).
\end{align}
We briefly describe how to compute $\p{\cY_{f_2}\mid \sigma_0}$ for $f_2\notin \cP^-(\sigma_0)$. Start the exploration
process in~\autoref{sec:explore} from \(f_{2}\) with one modification: if \(e^{-}_{i}\) has already
been paired in \(\sigma_{0}\), then we choose \(e^{+}_{i}\) according to \(\sigma_{0}\) instead of
uniformly at random. This modification \new{is similar to the one detailed in \autoref{sec:out-neigh}} and is needed since the original process in \autoref{sec:explore} only allows for $\cP^\pm_0=\emptyset$, but here we start with a non-empty partial pairing $\sigma_0$. If a head has already been paired, we lose control of the event. So we need to control the probability of this happening.
Since at most $\kappa_{1}^-$ half-edges are paired in $\sigma_0$, the probability of the event \(\{\cN_{\le
h_{1}+h_{2}}(f_{2}) \cap \cP^-(\sigma_{0}) \ne \emptyset\}\) is at most $O({(\kappa_{1}^{-})^{2}}/{n})$. 
This implies that
\begin{equation}\label{QVFU}
    \p{\cY(f_{2}) \mid \sigma_{0}} \le \p{\cY(f_{2})} +
    O\left(\frac{(\kappa_{1}^{-})^{2}}{n}\right)
        .
\end{equation}
By~\autoref{prop:SKFD}, we obtain
\begin{align}\label{BFGH2}
\p{\cY(f_1)\cap \cY(f_2)} &
\leq \p{\cY(f_1)}\left(O\left(\frac{(\kappa_{1}^{-})^{2}}{n}\right)+\p{\cY(f_2)}\right)
= O(n^{-2+4\theta}).
\end{align} 
So $\E{Z^2}=(1+o(1))\E{Z}^2$, and Chebyshev's inequality implies that whp $Z>0$.

Conditioned on $Z>0$, let $f_{0}$ be such that $\cY(f_{0})$ holds.  As $E_0^{f_{0}}$ holds, we have
$f_{0}\in \cE^-_{0}$; see~\eqref{NVEU2}. Moreover, $E_1^{f_{0}}\cap E_3^{f_{0}}$ implies that $ \sum_{g\in
\cE^-}P^{h_1}(g,f_{0})= P^-_{h_1}(f_{0})=\Gamma^{-}_{h_1}(f_{0})< \gamma$, where we used $E^{f_0}_3$ twice: once in the first equality to argue that the only way to reach $f$ in $h_1$ steps is to start at $g\in \cN_{h_1}^-(f)$ and once in the second equality. 
Moreover,~\autoref{rem:SDKA} implies that for every $\varepsilon'>0$, whp $\pi^\bfe(g)\leq \new{\max\{\pi^\bfe(e):e\in \cE^-\}\leq \pi_{\max}\leq}
n^{-1+\varepsilon'}$ for any $g\in \cE^-$. By stationarity at time $h_1$, the minimum stationary value of the random walk on heads satisfies
\begin{align*}
\pi^\bfe_{0}\leq \pi^\bfe(f_{0}) &= \sum_{g\in \cE^-}P^{h_1}(g,f_{0})\pi^\bfe(g) 
\leq n^{-(1+(1-3\theta/2)\hat{H}^-/\phi(a_{0}))+\varepsilon'}.
\end{align*}
Choosing $\theta=\varepsilon \phi(a_0)/3\hat{H}^{-},\varepsilon'=\varepsilon/2$, by~\autoref{prop:strongly} and~\autoref{lem:piminez} the upper bound in~\autoref{thm:main} follows.

\subsection{The proof of~\autoref{rem:empirical}}
\label{sec:empirical}

To show the upper bound in~\eqref{EABU}, we can follow the line of arguments in
\autoref{sec:lower}. Here we briefly sketch the changes needed for it to work. 

Let $\varepsilon>0$ be a sufficiently small constant. For \(\alpha \in [0, {\hat{H}^-}/{\phi \left(a_0\right)}]\),
let \(\gamma_{\alpha} \coloneqq n^{-\alpha}\) and \(\beta=\frac{\alpha
\phi \left(a_0\right)}{\hat{H}^-} \le 1\). \new{Let $\omega$ as in~\eqref{eq:omega}. Let $(X_r)_{r\geq 0}$ be a marked branching
process with distribution $\eta= \dout$.}
By
\autoref{cor:DWPD} with \(p = n^{-\frac{\beta}{(1+\varepsilon)}}\), which satisfies the hypotheses, we have
\begin{equation}\label{SVMQ}
    \p{(\cB_{t_\omega}(\gamma_{\alpha}))^c \cap \{t_\omega<\infty\}} 
    \le
    \new{\omega^{O(1)}}n^{-\big(\frac{1-\varepsilon}{1+\varepsilon}\big)\beta}
    \le
    n^{-(1-2\varepsilon)\beta}
    .
\end{equation}

\new{Let $\sigma_0$ be a partial realisation of the directed configuration model revealed during the out- and in-phases.}
Let 
\begin{equation}\label{JVEU}
    \cX_{f}^\alpha \coloneqq \bigcap_{e \in \cE^{-}} \left(\left\{A_{e,f} \ge \frac{1}{2}  \right\}
    \cap \left\{G_{e,f} \ge \frac{\omega \gamma_{\alpha}}{4} \right\} \right)
    ,
\end{equation}
\new{where $A_{e,f},G_{e,f}$ are defined as in~\eqref{QHKO}.}

Then, by the same argument used in~\autoref{lem:whp} with~\eqref{SVMQ}, we have
\begin{equation}\label{USOG}
    \p{(\cX_{f}^\alpha)^{c} \cap \{t^-_\omega(f) < \infty\}}
    \le
    n^{-(1-3\varepsilon)\beta}
    .
\end{equation}
Let \(\cZ_f \coloneqq \left\{0 < \pi^\bfe(f) < \frac{\gamma_{\alpha}}{2n}\right\}\).
It follows that
\begin{equation}\label{GFND}
    \begin{aligned}
        \p{
            \cZ_f
        }
        &
        =
        (1+o(1))
        \p{
            \left\{
                \pi^\bfe(f)  < \frac{\gamma_{\alpha} }{2n}
            \right\} 
            \cap 
            \left\{ t_{\omega}^{-}(f) < \infty \right\}
        }
        \\
        &
        \le
        (1+o(1))
        \left( 
        \p{
            \left\{
               \pi^\bfe(f)  < \frac{\gamma_{\alpha} }{2n}
            \right\} 
            \cap 
            \cX^{\alpha}_f
        }
        +
        \p{(\cX^{\alpha}_f)^{c} \cap \{t^-_\omega(f) < \infty\}}
        \right)
        \\
        &
        =
        (1+o(1))
        \p{
            \cup_{e \in \cE^{-}} \left\{\hat{P}^{\tau(f)}(e,f) < \frac{\gamma_{\alpha} }{2n} \right\}
            \cap 
            \cX^{\alpha}_f
        }
        +
        n^{-(1-3\varepsilon)\beta}
        \\
        &
        =
        n^{-(1-4\varepsilon)\beta}
        ,
    \end{aligned}
\end{equation}
where in the first equality we used~\autoref{prop:strongly}, \new{in the second equality we used the relation between $\pi^\bfe(f)$ and $\hat{P}^{\tau(f)}(e,f)$ that appeared in~\eqref{EODJ}} and in the last equality we used~\autoref{lem:concentration}.

Recall the definition of
\(\cY_f\) in~\eqref{EQRP}. Define \(\cY^{\alpha}_f\) analogously with \(h_{1}\) replaced by \(h_{1,\alpha} = \beta h_{1}\) and \(\gamma\) replaced by \(\gamma_{\alpha}\). The
argument displayed in~\autoref{prop:SKFD} gives
\begin{equation}\label{UNME}
    \p{
       \cZ_f
    }
    \ge
    (1+o(1))
    \p{\cY^{\alpha}_f} 
    \ge 
    n^{-\beta-\varepsilon}.
\end{equation}
\new{As the choice of $\varepsilon$ was arbitrary,~\eqref{EABU} follows from~\eqref{GFND} and~\eqref{UNME}, and so does~\autoref{rem:empirical}.} Concentration of $\psi((0, n^{-\alpha}])$ can be shown by the same second moment argument used in \autoref{prop:SKFD}.

\section{Applications}\label{sec:app}

\subsection{Hitting and cover times, and the proof of \autoref{thm:hitcov}}

We now prove~\autoref{thm:hitcov}, i.e., whp the hitting and the cover time of \(\vecGn\) are both
\(n^{1+\hat{H}^-/\phi(a_{0})+o(1)}\). Clearly, $\thit \leq \tcov$, so it suffices to lower
bound $\thit$ and upper bound $\tcov$.

For the lower bound, let \(C < (1+\hat{H}^{-}/\phi(a_{0}))\) be a constant. Then by
\autoref{thm:main}, we have \(\pimin \le n^{-C}\) whp.
Recall the definition of \(\tau_{v}\) in the introduction and
define \new{$\tau_{v}^+\coloneqq \inf\{t\geq 1: Z_t=v\}$}.
By the well-known relation between the expected time of the first return to the origin and the stationary distribution, we have whp
\begin{equation}\label{MYHQ}
    \max_{u \in [n]} 
    \E{\tau_{u}^+ \mid \vecGn, \new{Z_0=u}} 
    = 
    \max_{u \in [n]} 
    \frac{1}{\pi(u)} 
    =
    \frac{1}{\pimin}
    \ge 
    n^{C}
    .
\end{equation}
Thus, whp there exists a vertex \(u_{0} \in [n]\) such that \(\E{\tau_{u_{0}}^+\mid \vecGn, \new{Z_0=u_0}} \ge
n^{C}\), which implies that
\begin{equation}\label{GHMJ}
	n^{C}
	\leq 
    \E{\tau_{u_{0}}^+ \mid \vecGn, \new{Z_0=u_0}}
    =
    1 + 
    \frac{1}{d_{u_{0}}^{+}}
    \sum_{v \in \cN^{+}_{\leq 1}(u_{0})} m(u_0,v)
    \E{\tau_{u_{0}} \mid \vecGn, \new{Z_0=v}}
    ,
\end{equation}
where $m(u,v)$ is the multiplicity of the directed edge $(u,v)$ in $\vecGn$.
Thus, whp there exists two vertices \(u_{0}\) and \(v_{0}\) such that
\begin{equation}\label{LJPS}
    \E{\new{\tau_{u_{0}}} \mid \vecGn, \new{Z_0=v_{0}}}
    \ge
    n^{C} -1
    .
\end{equation}
It follows that \(\thit \ge n^{C}-1\) whp, as desired.

For the upper bound, let \(C > (1+\hat{H}^{-}/\phi(a_{0}))\) be a fixed constant and let \(t =
\omega^{2} = \log^{12} n\). Recall the definition of $\cV_0$ in~\autoref{prop:strongly}.
\autoref{lem:PT} implies that whp for all \(u \in [n]\) and \(v \in \cV_0\), there exists
$t_v\leq t$ (by~\eqref{IOWT}) such that  \( P^{t_v}(u,v) \ge n^{-C}\). So the
probability to hit \(v\) in at most \(t\) steps, which we call a \emph{try}, is at least \(n^{-C}\)
uniformly for any starting point $u$. Thus, the number of tries needed to hit \(v\) starting at $u$ is
stochastically dominated by a geometric random variable with success
probability \(n^{-C}\). It follows that  whp
\begin{equation}\label{VEDI}
    \thit= \max_{\substack{u\in [n] \\ v\in \cV_0}} \E{\tau_{v} \new{ \mid Z_0=u}}
 \le n^{C} t
    .
\end{equation}
Therefore, by Matthews' bound~\cite[Theorem~2.6]{matthews1988}, we have
\begin{equation}\label{MIKZ}
    \tcov \le H_{n} \thit
	=    
    n^{C+o(1)}
    ,
\end{equation}
where \(H_{n}\) is the \(n\)-th harmonic number.

\subsection{Explicit constants for particular degree sequences}\label{sec:examples}

In this section we discuss two particular examples where the polynomial exponent can be made explicit.

\subsubsection{\texorpdfstring{\(r\)}{r}-out digraph}

For any fixed integer $r\geq 2$, an \(r\)-out digraph \(\mathbb{D}_{n,r}\) is a random directed graph with \(n\) vertices in
which each vertex chooses \(r\) out-neighbours uniformly at random. It is used as a model for studying
uniformly random Deterministic Finite Automata~\cite{cai2017b}.  For \(r \ge 2\), Addario-Berry,
Balle and the second author~\cite{addario-berry2020} showed that in \(\mathbb{D}_{n,r}\), for every $\varepsilon>0$ and whp,
\begin{equation}\label{JFJT}
   n^{-(1 + \log(r)/(s r - \log r))-\varepsilon }\leq  \pimin \leq  n^{-(1 + \log(r)/(s r - \log r))+\varepsilon }
    ,
\end{equation}
where \(s\) is the largest solution of \(1- s = e^{-s r}\).

Note that \(D^{-}=D^-_n\) (the in-degree of a uniform random vertex in \(\vecGn\)) converges to a Poisson distribution with mean \(r\), whereas
the \(D^{+} =D_n^+ = r \geq 2\). Although in \(\mathbb{D}_{n, r}\) the in-degrees are random, in this context, by~\cite[Lemma 9.2]{cai2020a}, events that hold whp in \(\vecGn\) also hold
whp in \(\mathbb{D}_{n,r}\).

One can check that whp the maximum in-degree
of \(\mathbb{D}_{n,r}\) has order \(\frac{\log{n}}{\log\log n}\). A careful inspection of the proof of
\autoref{thm:main} shows that the bounded maximum degree condition can be relaxed to $\Delta^\pm =
o(\log n)$, provided that there are not many large degree vertices. Therefore, the conclusion of~\autoref{thm:main} holds in this setting. As $\p{D^+=r}=1$,
we have
$I(a\hat{H}^-)=\infty$ for any $a\neq 1$, so $a_0=1$ and for every $\varepsilon>0$ and whp 
\begin{align}\label{EONW}
n^{-(1+\abs{\log \hnu^-})- \varepsilon} \leq \pimin\leq n^{-(1+\abs{\log \hnu^-})+ \varepsilon},
\end{align}
which coincides with~\eqref{JFJT}.

\subsubsection{A toy example}

Here we show how the explicit value for $\pimin$ can be computed for a simple distribution, providing an example where $a_0\neq 1$. For $m\in \mathbb{N}$ and $n=4m$, consider a degree distribution $\vbfd_n$ that contains $m$ vertices of degrees $(0,2), (0,3), (5,2)$ and $ (5,3)$.
As $D^+$ is a uniform random variable supported on $\{2,3\}$ and is independent from $D^-$, we have $\log\dtildep= \log 2 +
\log(3/2)X$, where $X$ is a Bernoulli random variable with probability $p=3/5$.  The large deviation
rate function for a Bernoulli random variable with probability $p$ is $I_{\Be}(z)=z
\log\left(\frac{z}{p}\right)+(1-z) \log\left(\frac{1-z}{1-p}\right)$ for $z\in [0,1]$ (see, e.g.,~\cite[Exercise 2.2.23]{dembo2010}). Thus, we have $I(z)=I_{\Be}\left((\log(3/2))^{-1} (z-\log 2)\right)$.  

With the help of interval arithmetic libraries~\cite{sanders2020}, we get
\begin{equation}\label{QUEW}
    \hat{H}^{-} \doteq 0.936426
    \,,
    \hnu^{-} \doteq 0.181095
    \,,
    a_{0} \doteq 1.06671
    \,,
    \phi(a_{0}) \doteq 1.65129
    \,,
    1+\frac{\hat{H}^-}{\phi (a_{0})} \doteq 1.56708
    \,,
\end{equation}
with errors guaranteed to be at most \(10^{-6}\) by the algorithm.
As shown in~\autoref{fig:phi}, the function \(\phi(a)\) attains minimum at \(a_{0} > 1\).   

\begin{figure}[h]
    \centering
    \includegraphics[width=0.6\textwidth]{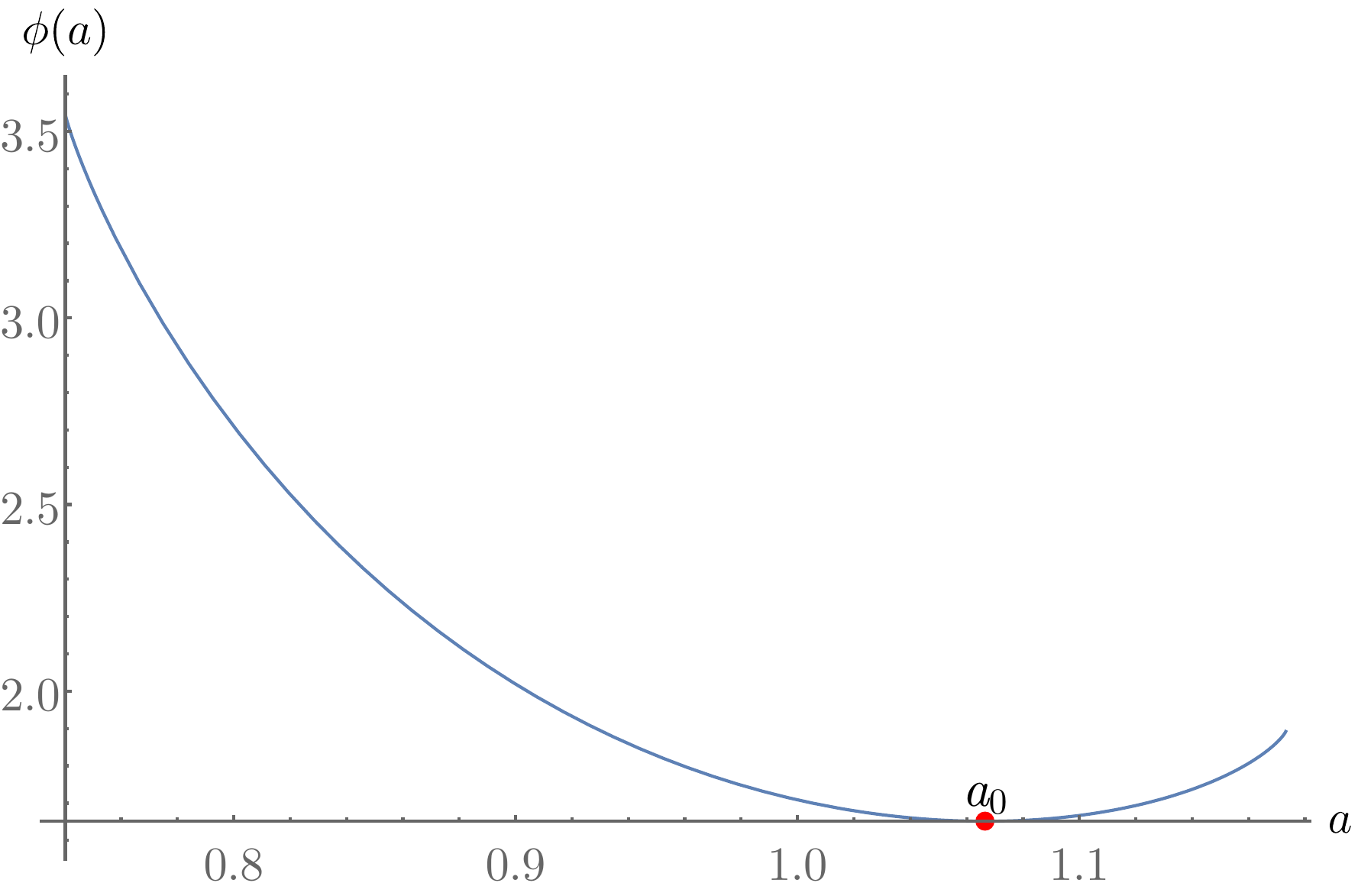}
    \caption{Plot of the function \(\phi(a)\)}\label{fig:phi}
\end{figure}

Recall the discussion on the critical distance in~\autoref{rem:crit}. For the degree sequence $\vbfd_n$ presented in this section and $n$ large enough, we obtain
\begin{align}\label{SDIE}
\frac{\text{crit}(\vecGn(\vbfd_n))}{\diam(\vecGn(\vbfd_n))} \doteq  0.989552.
\end{align}
Heuristically, this can be understood as follows: the distance from the bulk of the graph to the vertex attaining $\pimin$ is $98.95\%$ of the distance from the bulk to the furthest vertex to it. In particular, the vertex that is hardest to reach is not the
vertex that is furthest away from the others.


\paragraph{Acknowledgements.} We would like to thanks Pietro Caputo and Matteo Quattropani for insightful discussions on the topic. 
We are also grateful to the three anonymous reviewers, whose insightful comments have greatly contributed to improving this paper.


\begin{thebibliography}{29}
\providecommand{\natexlab}[1]{#1}
\providecommand{\url}[1]{\texttt{#1}}
\expandafter\ifx\csname urlstyle\endcsname\relax
  \providecommand{\doi}[1]{doi: #1}\else
  \providecommand{\doi}{doi: \begingroup \urlstyle{rm}\Url}\fi

\bibitem[{Addario-Berry} et~al.(2020){Addario-Berry}, Balle, and
  Perarnau]{addario-berry2020}
L.~{Addario-Berry}, B.~Balle, and G.~Perarnau.
\newblock Diameter and stationary distribution of random $r$-out
  digraphs.
\newblock \emph{The Electronic Journal of Combinatorics}, P3.28, 2020.
\newblock \doi{10/ghd74q}.

\bibitem[Amini(2010)]{amini2010}
H.~Amini.
\newblock Bootstrap {{Percolation}} in {{Living Neural Networks}}.
\newblock \emph{J Stat Phys}, 141\penalty0 (3):\penalty0 459--475, 2010.
\newblock \doi{10/c53hx4}.

\bibitem[Athreya and Ney(1972)]{athreya1972}
K.~B. Athreya and P.~E. Ney.
\newblock \emph{Branching {{Processes}}}.
\newblock Grundlehren Der Mathematischen {{Wissenschaften}}. {Springer-Verlag},
  {Berlin Heidelberg}, 1972.
\newblock \doi{10/dft4}.

\bibitem[Blanchet and Stauffer(2013)]{blanchet2013}
J.~Blanchet and A.~Stauffer.
\newblock Characterizing optimal sampling of binary contingency tables via the
  configuration model.
\newblock \emph{Random Structures \& Algorithms}, 42\penalty0 (2):\penalty0
  159--184, 2013.
\newblock \doi{10/f4mtxh}.

\bibitem[Bordenave et~al.(2018)Bordenave, Caputo, and Salez]{bordenave2018}
C.~Bordenave, P.~Caputo, and J.~Salez.
\newblock Random walk on sparse random digraphs.
\newblock \emph{Probab. Theory Relat. Fields}, 170\penalty0 (3):\penalty0
  933--960, 2018.
\newblock \doi{10/gc8nxk}.

\bibitem[Bordenave et~al.(2019)Bordenave, Caputo, and Salez]{bordenave2019}
C.~Bordenave, P.~Caputo, and J.~Salez.
\newblock Cutoff at the ``entropic time'' for sparse {{Markov}} chains.
\newblock \emph{Probab. Theory Relat. Fields}, 173\penalty0 (1):\penalty0
  261--292, 2019.
\newblock \doi{10/ghcrhr}.

\bibitem[Cai and Devroye(2017)]{cai2017b}
X.~S. Cai and L.~Devroye.
\newblock The graph structure of a deterministic automaton chosen at random.
\newblock \emph{Random Structures \& Algorithms}, 51\penalty0 (3):\penalty0
  428--458, 2017.
\newblock \doi{10/gbtqgb}.

\bibitem[Cai and Perarnau(2020{\natexlab{a}})]{cai2020}
X.~S. Cai and G.~Perarnau.
\newblock The giant component of the directed configuration model revisited.
\newblock \emph{ALEA, Lat. Am. J. Probab. Math. Stat.}, 18\penalty0 :\penalty0 1517--1528, 2021.
\newblock \doi{10/gk49g2}.

\bibitem[Cai and Perarnau(2020{\natexlab{b}})]{cai2020a}
X.~S. Cai and G.~Perarnau.
\newblock The diameter of the directed configuration model.
\newblock \emph{Ann. Inst. H. Poincaré Probab. Statist.}, 59\penalty0 (1):\penalty0 244--270, 2023.
\newblock \doi{10/jxhj}.

\bibitem[Cai et~al.(2021)Cai, Caputo, Perarnau, and Quattropani]{cai2021a}
X.~S. Cai, P.~Caputo, G.~Perarnau, and M.~Quattropani.
\newblock Rankings in directed configuration models with heavy tailed
  in-degrees.
\newblock \emph{Annals of Applied Probability}, 33\penalty0 :\penalty0 5613--5667, 2023.
\newblock \doi{10/nj7t}.

\bibitem[Caputo and Quattropani(2020)]{caputo2020a}
P.~Caputo and M.~Quattropani.
\newblock Stationary distribution and cover time of sparse directed
  configuration models.
\newblock \emph{Probab. Theory Relat. Fields}, 178\penalty0 :\penalty0 1011--1066, 2020.
\newblock \doi{10/ghd74v}.

\bibitem[Chatterjee(2007)]{chatterjee2007}
S.~Chatterjee.
\newblock Stein's method for concentration inequalities.
\newblock \emph{Probab. Theory Relat. Fields}, 138\penalty0 (1-2):\penalty0
  305--321, 2007.
\newblock \doi{10/fm2x4r}.

\bibitem[Chen and Olvera-Cravioto(2016)]{chen2016}
N.~Chen and M.~Olvera-Cravioto.
\newblock Coupling on weighted branching trees.
\newblock \emph{Advances in Applied Probability}, 48\penalty0 (2):\penalty0
  499--524, 2016.
\newblock \doi{10/f86bdg}.

\bibitem[Chen et~al.(2017)Chen, Litvak, and Olvera-Cravioto]{chen2017}
N.~Chen, N.~Litvak, and M.~Olvera-Cravioto.
\newblock Generalized {{PageRank}} on directed configuration networks.
\newblock \emph{Random Structures \& Algorithms}, 51\penalty0 (2):\penalty0
  237--274, 2017.
\newblock \doi{10/gbrth6}.

\bibitem[Cooper and Frieze(2004)]{cooper2004}
C.~Cooper and A.~Frieze.
\newblock The {{size}} of the {{largest strongly connected component}} of a
  {{random digraph}} with a {{given degree sequence}}.
\newblock \emph{Combinatorics, Probability and Computing}, 13\penalty0
  (3):\penalty0 319--337, 2004.
\newblock \doi{10/cn8q5j}.

\bibitem[Cooper and Frieze(2012)]{cooper2012}
C.~Cooper and A.~Frieze.
\newblock Stationary distribution and cover time of random walks on random
  digraphs.
\newblock \emph{J. Comb. Theory Ser. B}, 102\penalty0 (2):\penalty0 329--362, 2012.
\newblock \doi{10/cv9wbh}.

\bibitem[Dembo and Zeitouni(2010)]{dembo2010}
A.~Dembo and O.~Zeitouni.
\newblock \emph{Large {{Deviations Techniques}} and {{Applications}}}.
\newblock Stochastic {{Modelling}} and {{Applied Probability}}.
  {Springer-Verlag}, {Berlin Heidelberg}, second edition, 2010.
\newblock \doi{10/bcszkm}.

\bibitem[From(2007)]{from2007}
S.~G. From.
\newblock Some New Bounds on the Probability of Extinction of a Galton–Watson Process with Numerical Comparisons.
\newblock \emph{Communications in Statistics—Theory and Methods}, 36\penalty0
  (10):\penalty0 1993--2009, 2007.
  \newblock \doi{10/bhhgxn}

\bibitem[Graf(2016)]{graf2016}
A.~Graf.
\newblock \emph{On the Strongly Connected Components of Random Directed Graphs
  with given Degree Sequences}.
\newblock Master thesis, University of Waterloo, 2016.
\newblock \url{http://hdl.handle.net/10012/10681}.


\bibitem[{van der Hofstad}(2016)]{vanderhofstad2016}
R.~van der {Hofstad}.
\newblock \emph{Random {{Graphs}} and {{Complex Networks}}}, volume~1 of
  \emph{Cambridge Series in Statistical and Probabilistic Mathematics}.
\newblock {Cambridge University Press}, {Cambridge, England}, 2016.
\newblock \doi{10/ggv8q7}.

\bibitem[van~der Hoorn and {Olvera-Cravioto}(2018)]{hoorn2018}
P.~van~der Hoorn and M.~{Olvera-Cravioto}.
\newblock Typical distances in the directed configuration model.
\newblock \emph{Ann. Appl. Probab.}, 28\penalty0 (3):\penalty0 1739--1792,
  2018.
\newblock \doi{10/ggh2ch}.

\bibitem[Janson(2009)]{janson2009}
S.~Janson.
\newblock The probability that a random multigraph is simple.
\newblock \emph{Combinatorics, Probability and Computing}, 18\penalty0
  (1-2):\penalty0 205--225, 2009.
\newblock \doi{10/bg4m2c}.

\bibitem[Li(2018)]{li2018}
H.~Li.
\newblock Attack {{Vulnerability}} of {{Online Social Networks}}.
\newblock In \emph{2018 37th {{Chinese Control Conference}} ({{CCC}})}, pages
  1051--1056, 2018.
\newblock \doi{10/ggh2kg}.

\bibitem[Matthews(1988)]{matthews1988}
P.~Matthews.
\newblock Covering {{Problems}} for {{Brownian Motion}} on {{Spheres}}.
\newblock \emph{Ann. Probab.}, 16\penalty0 (1):\penalty0 189--199, 1988.
\newblock \doi{10/c8q2r8}.

\bibitem[Molloy and Reed(2002)]{molloy2002a}
M.~Molloy and B.~Reed.
\newblock \emph{Graph {{Colouring}} and the {{Probabilistic Method}}}.
\newblock Algorithms and {{Combinatorics}}. {Springer-Verlag}, {Berlin
  Heidelberg}, 2002.
\newblock \doi{10/hgcj}.

\bibitem[Petrov(1975)]{petrov1975}
V.~Petrov.
\newblock \emph{Sums of {{Independent Random Variables}}}.
\newblock Ergebnisse Der {{Mathematik}} Und Ihrer {{Grenzgebiete}}. 2.
  {{Folge}}. {Springer-Verlag}, {Berlin Heidelberg}, 1975.
\newblock \doi{10/hgck}.

\bibitem[Riordan and Wormald(2010)]{riordan2010}
O.~Riordan and N.~Wormald.
\newblock The diameter of sparse random graphs.
\newblock \emph{Combin. Probab. Comput.}, 19\penalty0 (5-6):\penalty0 835--926, 2010.
\newblock \doi{10/dgp6hh}.

\bibitem[Rosler(1993)]{rosler1993}
U.~R\"osler.
\newblock The weighted branching process.
\newblock \emph{Dynamics of complex and irregular systems (Bielefeld, 1991)}, Bielefeld Encounters in Mathematics and Physics VIII, World Science Publishing, River Edge, NJ, 1 154--165, 1993.
\newblock \doi{10/nj7s}.

\bibitem[Sanders et~al.(2020)Sanders, Benet, Agarwal, Gupta, Richard, Forets, Hanson, {van Dyk}, Rackauckas, {Miclua-C{\^a}mpeanu}, Koolen,
  Wormell, V{\'a}zquez, Grawitter, TagBot, O'Bryant, Carlsson, Piibeleht,
  {Reno}, Deits, Olver and Holy]{sanders2020}
D.~P. Sanders, L.~Benet, K.~Agarwal, E.~Gupta, B.~Richard, M.~Forets, E.~Hanson, B.~{van Dyk}, C.~Rackauckas,
  S.~{Miclua-C{\^a}mpeanu}, T.~Koolen, C.~Wormell, F.~A. V{\'a}zquez,
  J.~Grawitter, J.~TagBot, K.~O'Bryant, K.~Carlsson, M.~Piibeleht, {Reno},
  R.~Deits, S.~Olver and T.~Holy.
\newblock {{JuliaIntervals}}/{{IntervalArithmetic}}.jl: V0.17.5.
\newblock Zenodo, 2020.
\newblock URL \url{https://github.com/JuliaIntervals/IntervalArithmetic.jl}.


\end{thebibliography}
\end{document}